\documentclass[11pt, twoside]{amsart}
\usepackage[cp1251]{inputenc}
\usepackage{color}
\usepackage{amsthm,amssymb,latexsym}
\usepackage[centertags]{amsmath}

\usepackage[T2A]{fontenc}
\usepackage{MnSymbol}
\usepackage{bbm}
\usepackage{centernot}
\usepackage{mathrsfs}
\usepackage{amsthm}
\usepackage{amsfonts}
\usepackage{mathtools}
\usepackage{textcomp}
\usepackage{newlfont}
\usepackage{enumitem}  
\usepackage[normalem]{ulem} %for strikethrough text
\usepackage{cancel}%for strike out math

\newif\ifPDF
\ifx\pdfoutput\undefined\PDFfalse
\else \ifnum \pdfoutput > 0 \PDFtrue
        \else \PDFfalse
        \fi
\fi

\ifPDF
  \usepackage[pdftex]{xcolor, graphicx}
  \usepackage[colorlinks=true,linkcolor=blue, citecolor=blue]{hyperref}%

\else
  \usepackage{color}
  \usepackage[dvips]{graphicx}
  \usepackage[dvips]{hyperref}
\fi

\usepackage{tkz-graph}
\tikzset{EdgeStyle/.style = {->}}
\tikzset{LabelStyle/.style= {fill=yellow}}
\usetikzlibrary{shapes,snakes,calendar,matrix,backgrounds,folding}

\usepackage{bbm}

\usepackage[top=1in, bottom=1.25in, left=1.4in, right=1.4in]
{geometry}
\theoremstyle{definition}
\newtheorem{definition}{Definition}[section]
\theoremstyle{remark}

\newtheorem{remark}[definition]{Remark}

\theoremstyle{plain}
\newtheorem{thm}[definition]{Theorem}
\newtheorem{theorem}[definition]{Theorem}
\newtheorem{prop}[definition]{Proposition}
\newtheorem{lemma}[definition]{Lemma}

\newtheorem{corol}[definition]{Corollary}
\newtheorem{claim}[definition]{Claim}

\newcommand{\tcb}{\textcolor{blue}}
\newcommand{\tcr}{\textcolor{red}}

\def\wt{\widetilde}
\def\ol{\overline}
\def\ov{\overline}

\def\wh{\widehat}

\newcommand{\N}{\mathbb N}
\newcommand{\Z}{\mathbb Z}

\newcommand{\om}{\omega}

\newcommand{\be}{\begin{equation}}
\newcommand{\ee}{\end{equation}}
\newcommand{\ba}{\begin{aligned}}
\newcommand{\ea}{\end{aligned}}
\newcommand{\mc}{\mathcal}

\newcommand{\al}{\alpha}

\newcommand{\De}{\Delta}

\newcommand{\R}{{\mathbb R}}

\numberwithin{equation}{section}
\newcommand{\ignore}[1]{}

\begin{document}

\hfill{\normalsize{\textit{Dedicated to the memory of
Anatoly Vershik}}}
\\

%\title[Reducible generalized Bratteli diagrams] {Wealthiness of the set of invariant measures on  \\ 
%reducible generalized Bratteli diagrams}

\title[Inverse limit method for generalized Bratteli diagrams] %
%\title[Reducible generalized Bratteli diagrams] %{Dynamics and invariant measures on \\
%reducible generalized Bratteli diagrams}
{Inverse limit method for generalized Bratteli diagrams and invariant measures}

\author{Sergey Bezuglyi}
\address{Department of Mathematics,
University of Iowa, Iowa City, IA 52242-1419
USA}
\email{sergii-bezuglyi@uiowa.edu}

\author{Olena Karpel}
\address{AGH University of Krakow, Faculty of Applied Mathematics, al. Adama Mickiewicza~30, 30-059 Krak\'ow, Poland \&
B. Verkin Institute for Low Temperature Physics and Engineering,
47~Nauky Ave., Kharkiv, 61103, Ukraine}

\email{okarpel@agh.edu.pl}

\author{Jan Kwiatkowski}

\address{Faculty of Mathematics and Computer Science, Nicolaus Copernicus University, ul. Chopina 12/18, 87-100 Toru\'n, Poland}

\email{jkwiat@mat.umk.pl}

\author{Marcin Wata} 
\address{Faculty of Economic Sciences and Management, Nicolaus Copernicus University, ul. Chopina 12/18, 87-100 Toru\'n, Poland}

\email{marcin.wata@umk.pl}



\begin{abstract}
Generalized Bratteli diagrams with a countable set of vertices in every  
level are models for aperiodic Borel automorphisms.  
This paper is devoted to the description of all ergodic 
probability tail invariant measures on the path spaces of 
generalized Bratteli diagrams. Such measures can be 
identified with inverse limits of infinite-dimensional 
simplices associated with levels in generalized Bratteli diagrams. 
Though this method is general, we apply it to several classes of 
reducible generalized Bratteli diagrams. In particular, we explicitly describe all ergodic tail 
invariant probability measures for (i) the infinite Pascal 
graph and give the formulas for the 
values of such measures on cylinder sets, (ii) generalized Bratteli diagrams formed by a countable set of odometers,
(iii) reducible generalized Bratteli diagrams with uncountable
set of ergodic tail invariant probability measures. 
We also consider the method of measure extension by tail invariance from subdiagrams. 
We discuss the properties of the Vershik map defined on reducible generalized
Bratteli diagrams.

\end{abstract}

\maketitle

\tableofcontents

%%%%%
\section{Introduction}\label{sect Intr}
%%%%%
\textbf{Motivation and main results.} This paper is focused on the study of Borel dynamical 
systems that are realized on the path space of \textit{generalized 
Bratteli diagrams}. The notion of a generalized Bratteli 
diagram is a natural extension of the notion of a 
standard Bratteli diagram: we consider the Bratteli diagrams
in which each level consists of infinitely countably many vertices. 
The structure of such diagrams is completely determined by a 
sequence of countably infinite incidence matrices.

Because we refer to the notion of a Bratteli diagram 
practically in every paragraph, we give here a very concise definition of this object. A \textit{Bratteli diagram} is a countable 
graph $B = (V,E)$ whose vertices $V$ and edges $E$ are partitioned into subsets $V = \bigcup_n V_n$, $E = \bigcup_n
E_n$, where every $V_n$ is finite (for a standard diagram)
or countably infinite (for a generalized Bratteli diagram) and $E_n$
is the set of edges connecting vertices of levels $V_n$
and $V_{n+1}$. It is required that the set of incoming edges is finite for every vertex. The set $E_n$ determines the 
\textit{incidence matrix} 
$F_n = (f_{vw}^{(n)})$ where $f_{vw}^{(n)}$ is the number of 
edges connecting $v \in V_{n+1}$ and $w\in V_n$. For every diagram $B$, we consider 
the path space $X_B$ that is formed by infinite sequences of concatenating edges. Two infinite paths are called 
\textit{tail equivalent} if they coincide 
below some level. This defines the \textit{tail equivalence relation} $\mathcal R$, a dynamical system on $X_B$. 
The question about the existence and description of all 
$\mathcal R$-invariant measures is one of the most important
problems of the theory of Bratteli diagrams. 
More detailed definitions and facts related to generalized 
Bratteli diagrams can be found in 
Section \ref{sect_GBD}. 

Discrete combinatorial structures have been used in 
ergodic theory for many decades, e.g. in the papers
\cite{Katok2003}, \cite{Krieger1976}, \cite{Vershik1981}, \cite{Vershik1982}, and others. 
Such structures are useful, in particular, 
for the construction of various 
approximations of a transformation. Bratteli diagrams became  
a key tool in dynamics after the paper \cite{HPS1992} and the following series of papers 
\cite{GiordanoPutnamSkau1995}, \cite{GiordanoPutnamSkau2004}, \cite{GPS2008}, \cite{GPS2010}, 
\cite{DurandHostSkau1999}. 

Bratteli diagrams' role in \textit{Cantor and Borel dynamics} is crucial. The reason is that 
every homeomorphism of a Cantor set (and every aperiodic 
automorphism of a standard Borel space) can be realized
as an \textit{adic} transformation, called the \textit{Vershik map}, acting on the path space of a standard Bratteli diagram
(generalized Bratteli diagram, respectively). This means that all 
properties of a Cantor or Borel dynamical system $(X, T)$ can be seen on a corresponding Bratteli diagram. In particular, tail invariant measures on a Bratteli diagram are exactly the $T$-invariant measures. 
In this connection, we refer to the papers 
\cite{HPS1992}, \cite{Medynets2006}, 
\cite{DownarowiczKarpel_2019}, 
\cite{Shimomura2020}, 
\cite{BezuglyiDooleyKwiatkowski2006} where these results and
numerous applications have been proved.  
The reader can find a comprehensive exposition of this 
subject also in the recent books \cite{Putnam2018}, 
\cite{DurandPerrin2022} and surveys \cite{Durand2010}, 
\cite{BezuglyiKarpel2016}. 
The existing literature on Bratteli diagrams, 
corresponding to dynamical systems, invariant path-space measures, and other areas used in the paper
is very extensive. We refer, in particular,  to  \cite{Bratteli1972}, 
\cite{Kechris1995}, \cite{Nadkarni1995}, \cite{Kitchens1998},
\cite{Vershik1981} and other fundamental works cited 
below in the text where the reader can see the basic ideas and 
methods. 

The idea to work with a realization of a transformation on 
the path space of a Bratteli diagram has proved to be very 
useful and productive. This approach allowed one to 
classify minimal homeomorphisms of a Cantor set up to orbit 
equivalence \cite{GiordanoPutnamSkau1995}, 
\cite{GlasnerWeiss1995}. Furthermore, it turns out that the
structure of a Bratteli diagram makes it possible to see 
distinctly several important invariants of a transformation.
They are, in particular, the set of minimal components,
the support of every ergodic measure $\mu$, the values of 
the measure $\mu$ on clopen sets, etc. As was mentioned above, Bratteli diagrams can be studied independently of their relations to transformations defined
on a Cantor or Borel space if we use the tail equivalence relation as a prototype of a dynamical system. In fact, the class of such dynamical systems is wider than that generated 
by transformations because there are Bratteli diagrams that
do not support continuous Vershik maps, see 
\cite{Medynets2006}, \cite{BezuglyiKwiatkowskiYassawi2014}. 
We do not know whether there are generalized Bratteli diagrams 
that do not support a Borel Vershik map. 

The main theme of this paper is an explicit or 
algorithmic description of probability 
tail invariant measures on the path space of 
generalized Bratteli diagrams. This work was initiated in 
\cite{BezuglyiJorgensenKarpelSanadhya2023} where we were mostly 
interested in irreducible generalized Bratteli diagrams. The current
paper contains several intriguing examples of reducible Bratteli diagrams with uncountably many probability ergodic measures. 
We discuss in detail the method based on the study of inverse limits of infinite-dimensional simplices and show how probability measures can be found by
this method. Then we apply this approach to some classes of generalized Bratteli diagrams and find all ergodic tail 
invariant probability measures. In particular, we explicitly describe all ergodic tail 
invariant probability measures for (i) the infinite Pascal 
graph and give the formulas for the 
values of such measures on cylinder sets; (ii) generalized Bratteli diagrams formed by a countable set of odometers;
(iii) reducible generalized Bratteli diagrams with uncountable
set of ergodic tail invariant probability measures.
Two other topics, traditional for the study of dynamics on a Bratteli diagram, are considered in the paper. They are: (a)
the properties of the corresponding
Vershik maps and (b) the measure extensions from subdiagrams.
\vskip 0.3cm

\textbf{Generalized Bratteli diagrams.}
Why do we need generalized Bratteli diagrams with countable 
levels? The following result explains one obvious reason to study such diagrams. In 
\cite{BezuglyiDooleyKwiatkowski2006}, it is 
proved that every aperiodic Borel automorphism of an uncountable standard Borel space admits a realization as a Vershik map on the 
path space of a \textit{generalized Bratteli diagram}. 
A recent result in this direction was obtained in 
\cite{BezuglyiJorgensenSanadhya2024}, where the authors proved that there is a wide class of substitution dynamical systems on countable alphabets that 
can be realized as Vershik maps acting on stationary generalized Bratteli 
diagrams. We also refer to related recent works \cite{Manibo_Rust_Walton_2022}, \cite{Manibo2023}, \cite{Frettloh_Garber_Manibo_2022} where substitutions 
are considered on a compact alphabet. 
Among other possible applications of generalized 
Bratteli diagrams, we can mention Markov chains with 
countable sets of states, random walks, 
iterated function systems
\cite{https://doi.org/10.48550/arxiv.2210.14059}, 
harmonic analysis on the path space
of generalized Bratteli diagrams \cite{Bezuglyi_Jorgensen_2021}, etc. 

It is worth noting that the difference between standard 
and generalized Bratteli diagrams is essential. Even though their definitions are very similar, these two 
classes of diagrams represent different kinds of dynamics:
Cantor dynamics for standard diagrams and Borel dynamics 
for generalized Bratteli diagrams. In particular, the path space 
of a standard Bratteli diagram is compact and, for a generalized 
diagram, it is a zero-dimensional 
Polish space (non-locally compact, in general). This 
difference lies in the base of various phenomena that distinguish the
corresponding dynamical systems. For example, there are 
generalized Bratteli diagrams that do not admit probability 
tail invariant measures. There are stationary generalized Bratteli diagrams with uncountably many ergodic invariant probability measures.
Also, we note that the Vershik map 
cannot be made continuous for a class of generalized Bratteli
diagrams. More results of this kind can be found in 
\cite{BezuglyiJorgensenKarpelSanadhya2023} and in the present
paper. 

Our main results are related to finding ergodic tail invariant 
measures on the path space $X_B$ of a generalized Bratteli diagram $B$.
These measures can also be viewed as ergodic measures invariant for the corresponding Vershik maps. 
Remembering that every aperiodic Borel automorphism can be 
represented as a Vershik map, our results are about ergodic 
invariant measures of aperiodic Borel automorphisms of standard Borel spaces. 

This circle of problems is 
traditional for the study of dynamical systems. There are many well-developed methods in this direction and impressive 
achievements. In particular, for a stationary Bratteli 
diagram (standard or generalized) one can use the Perron-Frobenius theory to define a tail invariant measure,
see \cite{BezuglyiKwiatkowskiMedynetsSolomyak2010} for 
the standard case and 
\cite{BezuglyiJorgensenKarpelSanadhya2023} for generalized
diagrams. Another approach is based on the measure 
extension procedure from a subdiagram \cite{BezuglyiKarpelKwiatkowski2015}, \cite{AdamskaBKK2017}, \cite{BezuglyiKarpelKwiatkowski2024}. If $\ol B$ is 
a subdiagram of a Bratteli diagram $B$ and $\nu$ is a 
probability measure on the path space $X_{\ol B}$, then 
$\nu$ can be extended to a measure $\wh \nu$ 
(finite or infinite) supported
by $\mathcal R(X_{\ol B})$, the smallest tail invariant set
containing $X_{\ol B}$. 

A substantial part of our paper is devoted to the study of ergodic 
probability tail invariant measures on the 
\textit{infinite  Pascal graph} $B=(V, E)$ and
the Vershik map for some natural orders on $B$.
The Pascal graph is one of the most popular graphs related to dynamical systems. It has been extensively studied in
the context of Cantor dynamics, see the references in 
Section \ref{sect Bratteli-Pascal}. We quote 
\cite{Vershik2011}:
``Transformations generated by classical graded 
graphs, such as the ordinary and multidimensional
Pascal graphs, the Young graph, the graph of walks in Weyl 
chambers, etc., provide examples
of combinatorial origin of the new, very interesting class of adic transformations.''

We consider the infinite Pascal graph where the $n$-th level is formed by the vertices  $\ol s$:
$$
V_{n} = \{\overline{s}= (s_{1}, s_{2}, \ldots ) : 
\sum_{i = 1}^{\infty}s_{i} = n , \ s_i \in \N_0\}.
$$
The set $E_n$ of edges between the vertices of 
levels $V_{n}$ and $V_{n+1}$ is determined by the following property:
for $\ol s \in V_n$, $\ol t \in V_{n+1}$, the set $E(\ol s, 
\ol t)$ of edges between $\ol s$ and $\ol t$ consists of exactly one edge if
and only if $\ol t = \ol s + \ol e^{(i)}$ for some $i 
\in \N$ where $\{e^{(i)} : I \in \N\}$ is the standard basis. 

Another class of generalized Bratteli diagrams in our focus is 
the class of reducible diagrams whose incidence matrices are
triangular. Even for stationary Bratteli diagrams, such diagrams 
have remarkable properties; we study them in our paper.  
\vskip 0.3cm

\textbf{Inverse limit method.}
In this paper, we consider and systematically apply a new approach that is based on identifying measures with inverse limits. 
We consider in this article a method (called the ``inverse limit method'') that gives the possibility to describe the set of all invariant probability measures of a generalized Bratteli diagram. This method is a non-trivial generalization of a similar approach 
developed for standard Bratteli diagrams. 
In \cite{BezuglyiKwiatkowskiMedynetsSolomyak2010}, we proved 
that, for a classical Bratteli diagram with the sequence of incidence stochastic matrices $(F_n)$, every tail invariant measure $\mu$ was completely 
determined by a sequence of non-negative probability 
vectors $(\ol q^{(n)})$ such
that $F_n^T \ol q^{(n+1)} = \ol q^{(n)}$ for all $n \geq 1$
($F_n^T$ is the transpose matrix). 
In other words, we prove that
the set $M_1(\mc R)$ of all probability tail invariant 
measures on $X_B$ is identified with the inverse limit
of the sets $(\De_1^{(n)}, F_n^T)$:
$$
M_1(\mc R) = \varprojlim_{n\to \infty} (\De_1^{(n)}, F_n^T)
$$
where $\De_1^{(n)}$ is the finite-dimensional simplex 
indexed by the vertices of the $n$-th level. 

If we apply a similar approach to generalized Bratteli diagrams, then we come across the following difficulties. 
Firstly, the infinite-dimensional simplex $\De_1^{(n)}$ is not closed so we should work with the set
$\De^{(n)} = \{\ol x = \langle x_i \rangle : \sum_{i \in V_n}
x_i \leq 1\}$. Secondly, 
when we consider the intersection of the convex sets
$\Delta^{(n,\infty)} =\bigcap_{m = 1}^{\infty}\Delta^{(n,m)}$, 
it can be empty where 
$\Delta^{(n,m)} := {G^{(n,m)}}^T(\Delta^{(n + m)})$ and ${G^{(n,m)}}^T = F_{n}^{T}\cdot  \ldots \cdot  
F_{n + m - 1}^{T}$.

To avoid these obstacles, we take $cl(\Delta^{(n,m)})$, the closure of the set 
$\Delta^{(n,m)}$ in $\De^{(n)}$. Since  
$cl(\Delta^{(n,m)}) \supset cl(\Delta^{(n,m+1)})$, 
we get the closed nonempty set 
$$
\Delta^{(n,\infty,cl)} := 
\bigcap_{m = 1}^{\infty}cl(\Delta^{(n,m)}), n \in \N. 
$$
In general, the sequence  
$\{ (\Delta^{(n,\infty,cl)}, F_{n}^{T}) \}$ 
does not form the inverse limit.  
But if there exists a sequence of probability vectors 
$\{\ol q^{(n)}\}$, where $\ol q^{(n)} \in 
\Delta^{(n,\infty,cl)}$, 
such that $F_n^T (\ol q^{(n+1)}) = \ol q^{(n)}$ for all $n$, 
then this sequence produces a probability tail invariant measure. 

The next problem is to describe the elements of the set
$\Delta^{(n,\infty,cl)}$. We prove that 
$$
\Delta^{(n,\infty,cl)} = 
\left\{ \int_{L^{(n)}} \overline{z} \; d\mu(\ol z)\ :\ \mu 
\in M_1(L^{(n)}) \right\},  
$$
where 
$$ 
L^{(n)} = \{ \ol x \in \Delta^{(n)} : 
\ol x = \lim_{m\to \infty}
\ol g_{v_{m}},\ v_{m} \in V_{n + m} \}, \quad \overline{g}_{v} = {G^{(n,m)}}^T(\ol e_{v}^{(n + m)}). 
$$ 
Finally, we show how all vectors from $L^{(n)}$ can be found
explicitly. 

In this paper, we apply the inverse limit method to several classes 
of generalized Bratteli diagrams. 
\vskip 0.2cm

\textbf{More about tail invariant measures.}
Discussing the methods of finding tail invariant measures (finite
or infinite) on a generalized Bratteli diagram, we should mention two other approaches to this problem apart from the inverse limit method.

The first approach is based on the Perron-Frobenius theory. 
For a stationary Bratteli diagram 
with the incidence matrix $F$ (finite or infinite), we find the 
Perron eigenvalue 
$\lambda$ and a positive eigenvector $\xi$ satisfying $F^T \xi =
\lambda \xi$. Then this data allows one to determine a tail invariant
measure $\mu$ on a cylinder set $[\ol e]$ corresponding to a finite path ending at a vertex $v \in V_n$ by setting 
$\mu([\ol e]) = \lambda^{-n} \xi_v$.
The measure $\mu$ can be finite or infinite depending on the 
summability of the entries of the vector $\xi$.
This method works only for diagrams with finite Perron eigenvalue
and is not universal (though very useful), see 
\cite{BezuglyiKwiatkowskiMedynetsSolomyak2010},
and \cite{BezuglyiJorgensen2022}, and
\cite{BezuglyiJorgensenKarpelSanadhya2023} where the reader can find
numerous applications. The case of reducible stationary Bratteli diagrams is not covered by the discussed method.

Another method to construct a tail invariant measure is based on 
the procedure called \textit{measure extension} from a subdiagram. 
Let $\ol B$ be a subdiagram of a generalized Bratteli diagram $B$. 
Take a probability tail invariant measure $\nu$ on the path space $X_{\ol B}$ of the subdiagram $\ol B$. Consider the set 
$\wh X_{\ol B} = \mathcal R(X_{\ol B})$, the smallest tail invariant
set containing $X_{\ol B}$. Extend the measure $\nu$ by tail invariance to the set $\mathcal R(X_{\ol B})$. We obtain in such a way a new tail invariant measure $\wh\nu$ on the diagram $B$. The 
main difficulty consists of finding out whether the extended measure $\wh\nu$ is finite or infinite.
This method works very well for many classes of diagrams, see 
\cite{BezuglyiKwiatkowskiMedynetsSolomyak2013},
\cite{AdamskaBKK2017}, \cite{BezuglyiKarpelKwiatkowski2024}, and the results of this paper below.

\vskip 0.3cm

\textbf{Outline of the paper.}
We describe our main results and outline of the paper.
In Section \ref{sect prel}, we consider infinite-dimensional simplices and prove some facts about their extreme points 
in the spirit of Krein-Milman theory. These results will be 
used in the next sections. 
Section \ref{sect_GBD} contains the basic definitions of the objects 
related to (generalized) Bratteli diagrams. Among them, we 
mention, first of all, the notions of tail invariant measures,
the Vershik map, stationary and bounded size Bratteli diagrams, 
vertex and edge subdiagrams, and measure extension from a 
subdiagram. The definitions and results from this section are used in the other sections. 
Section \ref{sect_inverse limits} contains the 
details of the inverse limit method that is briefly described above.
The second part of the paper, Sections \ref{sect Bratteli-Pascal} 
- \ref{ssect measures on B_infty}, are devoted 
to some applications of the inverse limit method to several classes
of generalized Bratteli diagrams. For these classes,  
we completely describe the set of ergodic   
tail invariant probability measures. We note that we focus 
mostly on examples of reducible generalized Bratteli diagrams. 
In Section \ref{sect Bratteli-Pascal}, 
the infinite Pascal graph whose vertices are indexed ether by $\N$ or $\Z$ is studied. We show that there are 
uncountably many ergodic tail invariant probability measures 
on the infinite Pascal graph and give explicit formulas for
the values of such measures on cylinder sets. In Section 
\ref{sect bounded size}, we deal with a bounded size Bratteli 
diagram (the incidence matrices are banded matrices). 
In particular, we consider
a generalized Bratteli diagram $B_k$ with the entries of incidence 
matrices defined by the formula:
$$
f_{vw}^{(n)} = \begin{cases}
    1, & |v-w| \leq k \\
    0, &  \mbox{otherwise},
\end{cases}  \ \ w\in V_n, v \in V_{n+1},
$$ 
and prove that there is no probability tail invariant 
measure on the diagram $B_k$. More examples are given in 
Section \ref{sect examples}. There we discuss the 
connections of our results about ergodic measures 
with substitution dynamical systems 
on a countable alphabet and with the diagram obtained as 
the union of infinitely many odometers as subdiagrams, each of the odometers is connected
with its right neighbor.  For this
class of diagrams, we give a criterion for the existence of probability tail invariant measures and describe the set of all 
ergodic measures explicitly.  We show that all ergodic probability
tail invariant measures are extensions of measures sitting 
on odometers.
Section \ref{ssect measures on B_infty} 
is devoted to an amazing example of a ``triangular'' generalized
stationary Bratteli diagram $B_\infty$ and its subdiagrams. 
The incidence matrix $F$ of $B_\infty$ is the 0-1 low-triangular matrix whose non-zero entries equal 1. 
This diagram has 
amazing properties and unexpected connections with other 
notions like completely monotonic sequences and contains 
interesting subdiagrams. We prove that the diagram $B_\infty$  supports uncountably many ergodic probability tail invariant measures and describe them explicitly. We also consider triangular 
standard Bratteli subdiagrams and find their internal 
tail invariant probability measures. 
The final Section \ref{sect VershikMapPascal} 
discusses the Vershik map on the infinite
Pascal graph and the diagram $B_\infty$ with respect to some natural orders.

\section{Preliminaries on convex sets and invariant 
measures}\label{sect prel}

In this section, we discuss some known and new facts 
about the structure of convex sets generated by 
infinite-dimensional positive vectors. These results will 
be used below in our study of tail invariant measures on 
the path space of a generalized Bratteli diagram. 

\subsection{Infinite-dimensional convex sets} 
We use the standard notation $\mathbb N, \Z, \R,$  
$\mathbb{N}_0 = \N \cup \{0\}$ for the sets of numbers, $\mathbb R_+$ denotes non-negative reals. Let $V$ be an infinite countable set. We prefer to work 
with $V$ (not $\N$ or $\Z$) because the Bratteli diagrams 
considered in this paper have to have vertices enumerated 
by countable 
sets of a rather complicated structure. Nevertheless, 
we will need to use a bijection $a: V \to \N$ that 
identifies the elements of $V$ with natural numbers. 
This function $a(v)$ is chosen and fixed for this paper. 

Let $\mathbb{R}^{V}$ be a linear space of infinite vectors 
$\overline{x} = \langle x_{v} \in \mathbb{R} : \ v \in V
\rangle$. With some abuse of terminology, we will use the
word ``sequence'' considering the element $\overline{x}$
of $\mathbb{R}^{V}$. 
% \tcr{Can we say ``sequence'' here?}

Let $I = [0,1]$. Consider the subset $I^{V}$ of  
$\mathbb{R}^{V}$, that is 
$$
I^{V} = \{\overline{x}  =\langle x_{v} \rangle \in \ 
\mathbb{R}^{V} : 0 \leq \ x_{v} \leq 1 \}.
$$
For $\overline{x} \in I^{V}$,  define 
\begin{equation}\label{eq_r1'}
|\overline{x}| = \sum_{v \in V}{\frac{1}
{2^{a(v)}}  |x_{v}|} 
\end{equation}
and set
\be\label{eq_dist d}
d(\overline{x},\overline{y}) : = |\overline{x} - \overline{y}| = 
\sum_{v \in V}\frac{1}{2^{a(v)}} 
|x_{v} - y_{v}|,
\ee
where $\overline{x} = \langle x_{v} \rangle$ and 
$\overline{y} = \langle y_{v} \rangle$ are in $I^V$. 
Then $(I^V, d)$ is a compact metric space. 

We remark that a sequence $\{\overline{x}^{(k)}\} = 
\{\langle x_{v}^{(k)}\rangle \}$ converges to
a vector $\ol x = \langle x_v \rangle$ if and only if 
$x_{v}^{(k)} \rightarrow  x_{v}$ as $k \rightarrow \infty$ 
for every $v \in V$. 

The set $I^{V}$ contains vectors
$\ol e^{(u)}, u \in V$, such that  $\ol e^{(u)} = \langle  
e_{v}^{(u)} \rangle,\ e_{v}^{(u)}= 1$ if $v = u$ and
$e_{v}^{(u)} = 0$ if $v \neq u$. Then, for any $\ol x \in 
I^V$, we can write 
$$\overline{x} = \sum_{v \in V} x_{v}\ol e^{(v)}.
$$

For a vector $\overline{x} = \langle x_{v} : v \in V 
\rangle \in I^{V}$, we define the vectors 
${\overline{x}}^{[s]}, s=1, 2, \ldots $ as follows:  

${\overline{x}}^{[s]} = \langle x_{v}^{\lbrack s\rbrack} 
\rangle$, where $x_{v}^{[s]} = x_{v}$ if $a(v) \leq s$ and 
$x_{v}^{[s]} =0$ if $a(v) > s$.

\begin{lemma} Let $\{\alpha_k : k \in K\}$ be a sequence 
of non-negative numbers such that $\sum_{k \in K} \al_k =1$, 
where $K$ is a subset of $\N$. Then, for
${\overline{x}}^{(k)}, {\overline{y}}^{(k)} \in I^{V}$,  the following inequality holds:
\begin{equation}\label{e1}
d\left(\sum_{k \in K} \alpha_{k}  {\overline{x}}^{(k)}, \sum_{k \in K} \alpha_{k}  {\overline{y}}^{(k)}\right) \leq \sum_{k \in K} \alpha_{k}  d({\overline{x}}^{(k)}, {\overline{y}}^{(k)}).     
\end{equation}
\end{lemma}

\begin{proof} Straightforward.
\end{proof}

Define the subsets $\Delta$ and $\Delta_{a}$ of $I^{V}$, 
$0 \leq a \leq 1$,  where 
$$
\Delta = \{ \overline{x} =\langle x_{v} \rangle : 
\sum_{v \in V} x_{v} \leq 1 \}, \quad
\Delta_{a} =  \{\overline{x} = \langle x_{v} \rangle :  \sum_{v \in V}x_{v} = a \}.$$

Then 
$\Delta = \bigcup_{0 \leq a \leq 1}  \Delta_{a} $, 
the sets $\Delta_{a}$ and $\Delta$ are convex, and 
$\Delta$ is a closed subset of $I^{V}$. 

We will use the following abbreviations: the symbols 
``cl'', ``conv'', and 
``ext'' denote the ``closure'',  ``convex 
hull'', and ``extreme points'' of a set, respectively.

\begin{lemma}
(1) $\Delta = cl[conv(ext \ \Delta)]$.

(2) $\Delta_a $ is dense in $\bigcup_{0 \leq b \leq a}
\Delta_{b}$ for any $0 < a \leq 1$. In particular,
$cl(\Delta_1) = \Delta$.
\end{lemma}

\begin{proof}
(1) This is the Krein-Milman theorem.

For (2), take $\overline{x} =\langle x_{v} \rangle \in 
\Delta_{b}$, i.e., $\sum_{v \in V}{x}_{v} = b < a$ 
and define
$\ol x^{(k)} = \langle x_{v}^{(k)} \rangle$ such that 
$x_{v}^{(k)} = x_{v}$ if $a(v) \leq k$ and $x_{v}^{(k)} = 
y_{v}^{(k)}$ if $a(v) >  k$,
where 
$$
\sum_{a(v) > k} y_{v}^{(k)} = a - \sum_{a(v)\leq k}{x}_{v},
$$
and $y_{v}^{(k)} \geq 0,\ k = 1, 2, \ldots$. 
Then
$\overline{x}^{(k)} \in \Delta_{a}$ for every $k = 1, 2, 
\ldots $. 
Moreover, 
$$d( \overline{x}^{(k)}, \overline{x}) = \sum_{a(v) > k}
\frac{1}{2^{a(v)}}|x_{v} - y_{v}^{(k)}| \leq
 \sum_{a(v) > k}\frac{1}{2^{a(v)}} =  \frac{1}{2^{k}} 
 \rightarrow 0$$
as $k \rightarrow \infty$.
\end{proof}

\subsection{Convex sets generated by a sequence of vectors} 
Let $\{\overline{b}^{(i)} : i \in \N\}$ be a sequence of 
vectors from the set $\De_1$. Denote by 
$\De(\{\ol b^{(i)}\})$ the smallest convex and closed subset 
of $\De$ that contains all vectors 
${\overline{b}}^{(1)}$, ${\overline{b}}^{(2)}$,
${\overline{b}}^{(3)}, \ldots$. We also define by 
$C = C(\{\ol b^{(i)}\})$ the set of all infinite convex 
combinations of the vectors $\{\overline{b}^{(i)}\}$, 
that is
$$
C = \{ \overline{x}_{u} = \sum_{k = 1}^{\infty}u_{k} 
{\overline{b}}^{(k)} : u_{k} \geq 0, \ 
\sum_{k = 1}^{\infty}u_{k} = 1 \}.
$$

\begin{prop} \label{prop_C in Delta}
The set $C(\{\ol b^{(i)}\})$ is a dense subset of 
$\Delta(\{\overline{b}^{(i)}\})$.
\end{prop}
\begin{proof}

To prove this result we need a formula for the distance 
between any two vectors from $C = C(\{\ol b^{(i)}\})$.

\begin{claim}\label{Claim:distance}
Let $\overline{x}_{u} = \sum_{k = 1}^{\infty}u_{k} {\overline{b}}^{(k)}$ and
$\overline{x}_{w} = \sum_{k = 1}^{\infty}w_{k} \overline{b}^{(k)}$ be two 
vectors from the set $C$. Then
\begin{equation}\label{eq_dist in C}
d(\overline{x}_{u}, \ol x_{w}) =
\sum_{v \in V}\frac{1}{2^{a(v)}} \left|\sum_{k = 1}^{\infty}{(u_{k} - w_{k}) b_{v}^{(k)}}\right|.    
\end{equation}    
\end{claim}

\begin{proof}[Proof of Claim~\ref{Claim:distance}]
Indeed, every vector $\ol x_u$ from $C$ can be represented 
as follows:
$$\overline{x}_{u} = \sum_{k = 1}^{\infty}u_{k} 
\overline{b}^{(k)} = 
\sum_{k = 1}^{\infty}u_{k}\ \sum_{v \in V}{\lbrack 
b}_{v}^{(k)} \ol e^{(v)}\rbrack =
\sum_{v \in V}{\left\lbrack \sum_{k = 1}^{\infty}{u_{k}  
b_{v}^{(k)}} \right\rbrack  \ol e^{(v)}}.$$
A similar formula holds for $\ol x_w$. Hence, 
applying \eqref{eq_dist d}, we obtain \eqref{eq_dist in C}.
\end{proof}

%% Check it again

We show now that 
\begin{equation}\label{eq_C in}
C(\{\ol b^{(i)}\}) \subset \Delta(\{\overline{b}^{(i)}\}).
\end{equation}
This inclusion can be deduced also from 
Theorem \ref{thm t1}.

It is clear, that any vector $\ol x_u =\sum_{k=1}^{\infty} 
u_{k} \overline{b}^{(k)}$  such that only finitely many 
$u_k$'s are non-zero is automatically in 
$\Delta({\overline{b}}^{(k)})$. 

Suppose that $\ol x_u$ is such that 
$\sum_{k = 1}^{\infty} u_{k} = 1$ 
and $u_{k} > 0$ for infinitely many $k$.
%$k = 1, 2, \ldots$. 
Define the vectors
$$
\ol x_{u}^{(l)} = \sum_{k \leq l}\frac{u_{k}}{s(l)} 
\overline{b}^{(k)},\ \ \  l = 1, 2, \ldots,
$$ 
where $s(l) = \sum_{k \leq l}u_{k}$. 
Using \eqref{eq_dist in C}, we can write
\be
\ba
d(\ol x_{u}^{(l)}, \ol x_{u}) = & 
\sum_{v\in V}\frac{1}{2^{a(v)}} \left|\left(\sum_{k \leq l} 
\left(\frac{1}{s(l)} - 1\right) u_{k}  b_{v}^{(k)} - \sum_{k > l}{u_{k} 
b_{v}^{(k)}}\right)\right|\\ 
\leq & \left(\frac{1}{s(l)} - 1\right) \sum_{v \in V}
\frac{1}{2^{a(v)}} \left|\sum_{k \leq l} u_{k}b_{v}^{(k)}
\right| + \sum_{v \in V}\frac{1}{2^{a(v)}}\left| 
\sum_{k > l}u_{k}b_{v}^{(k)} \right|\\
\leq & \left(\frac{1}{s(l)} - 1\right)\sum_{v \in V}\frac{1}{2^{a(v)}}
\sum_{k \leq l}u_{k} + \sum_{v \in V}\frac{1}{2^{a(v)}}   
\sum_{k > l} u_{k}\\
\leq & \left( \frac{1}{s(l)} - 1\right)  + \sum_{k > l}u_{k} 
\longrightarrow 0, 
\ea
\ee 
as $l \rightarrow \infty$.

Because $\ol x_{u}^{(l)} \in \Delta(\{\overline{b}^{(k)}\})$
and this set is closed, we conclude that 
$\ol x_{u}\in \Delta(\{\overline{b}^{(k)}\})$. 
This proves \eqref{eq_C in}. The fact that 
$C(\{\ol b^{(i)}\})$ is dense in 
$\Delta(\{\overline{b}^{(i)}\})$ is obvious. 
\end{proof}

\begin{remark}
Denote by $cl(\{\overline{b}^{(k)}\})$ 
the closure of the set $\{\overline{b}^{(k)}\}$ in $\Delta$. 
It is obvious that 
$$
cl(\{\overline{b}^{(k)}\}) \subset 
\Delta(\{ \overline{b}^{(k)}\}).
$$

The set $\Delta(\{ \overline{b}^{(k)}\})$ is closed and 
convex. By the Krein-Milman theorem, it is the 
closure of the convex hull of its extreme points.   
\end{remark}

We recall the following well-known fact, see e.g. 
\cite{Diestel_1984}.

\begin{theorem}\cite[Theorem 2, p. 149]{Diestel_1984}\label{Thm:Diestel}
Let $Y$ be a real locally convex
Hausdorff linear topological space and $K = cl(conv (Z))$ 
be a closed convex hull of a set $Z \subset Y$. 
For $Z$ compact, 
$x \in K$ if and only if there exists a regular Borel 
probability
measure $\mu$ on $Z$ whose barycenter exists and is $x$.
In other words, for any continuous linear functional 
$f\in Y^*$, one has 
\be\label{eq_barycenter}
f(x) = \int_Z f(z) \; d\mu(z).
\ee 
\end{theorem}

We also recall the so-called Milman's Converse of the 
Krein-Milman theorem: 

\begin{theorem}\cite [Corollary 4, 
p.151]{Diestel_1984}\label{Corol:Diestel}: Let $K$ be a compact convex subset of a real locally convex
Hausdorff linear topological space $Y$. If $K$ is the closed convex hull of a set $Z$, then
every extreme point of $K$ lies in the closure of $Z$.
\end{theorem}

Applying these facts to the set 
$\Delta(\{\overline{b}^{(k)}\})$, we obtain the following  
results.

\begin{theorem}\label{thm t1}
(1) The set $ext(\Delta(\{\overline{b}^{(k)}\})$ 
of all extreme points of
$\Delta(\{\overline{b}^{(k)}\})$ is contained
in $cl(\{\overline{b}^{(k)}\})$. 

(2) A vector $\ol x$ belongs to 
$\Delta(\{\overline{b}^{(k)}\})$ if and only if 
there exists a Borel 
probability measure $\mu$ on 
$cl(\{\overline{b}^{(k)}\})$
such that 
\be\label{eq_barycenter modified}
\ol x = \int_{cl(\{\overline{b}^{(k)}\})} \ol z\; 
d\mu(\ol z)\footnote{We clarify the meaning of 
this formula in the proof.}.
\ee 
\end{theorem}

\begin{proof}
We need to prove (2) only. We use \eqref{eq_barycenter} as
a starting point. Note that the set $\Delta(\{ 
\overline{b}^{(k)}\})$ is also a closed convex hull of the set
$cl(\{\overline{b}^{(k)}\})$.

Using Theorem~\ref{Thm:Diestel}, 
we see that for 
every $\overline{x} \in \Delta(\{\overline{b}^{(k)}\})$
there is a regular Borel probability measure $\mu= \mu_x$ on 
 $cl(\{\overline{b}^{(k)}\})$ whose barycenter exists and 
is $\overline{x}$. Therefore, we can write
\begin{equation}\label{eq_barycenter 2}
f(\overline{x}) = \int_{cl(\{\overline{b}^{(k)}\})}
f \left( \overline{z} \right) \; d\mu(\ol z).
\end{equation}
For $\ol z = \sum_{u \in V}z_u \ol e^{(u)} \in 
\Delta(\{\overline{b}^{(k)}\})$, we consider the projection
$f_v : \ol z \mapsto z_v$ and substitute it into  
\eqref{eq_barycenter 2}: 
\be\label{eq_barycenter 3}
x_v = \int_{cl(\{\overline{b}^{(k)}\})}
z_v \; d\mu(\ol z), \quad v \in V. 
\ee
Relation \eqref{eq_barycenter 3} is the coordinate-wise 
form of \eqref{eq_barycenter}. Another form of this 
statement is:
$$
\Delta(\{\overline{b}^{(k)}\}) =
\left\{\int_{cl(\{\overline{b}^{(k)}\})} \ol z\; 
d\mu(\ol z) : \mu \in M_1(cl(\{\overline{b}^{(k)}\})) 
\right\},
$$
where $M_1(\cdot)$ is the set of all Borel probability
measures.

\end{proof}

\begin{remark}
We note that the integrals considered above can be viewed  as 
Bochner integrals for the  functions  
$f$ from $cl(\{\overline{b}^{(k)}\})$ to $B$, 
where  $B$ is the Banach space generated by the vectors  
$\overline{b}^{(k)}, k \in \N$,  with the norm given in 
(\ref{eq_r1'}).  
\end{remark}

\subsection{Continuity of matrix maps} 
We consider here special maps (\textit{matrix maps})
between the metric spaces
$\Delta = \{\overline{x}\overline{} = \langle x_{v} 
\rangle,\ \sum_{v \in V}x_{v} \leq 1,\ x_{v} \geq 0\}$ 
and
$\Delta' = \{ \overline{y} =\langle y_{w}\rangle, \sum_{w 
\in W}y_{w} \leq 1,\ y_{w} \geq 0\}$,
where $V$ and $W$ are countably infinite sets. 

Let
$F = \{ f_{vw} \}, v \in V, \ w \in W$,  
be an infinite matrix such that
$f_{vw} \geq 0, \sum_{w \in W}f_{vw} = 1$ for every $v \in
V$, and $f_{vw} > 0$ if and only if $w \in W_{v}$,
where $W_{v} \subset W$ is a finite set for every $v \in V$.
The transpose matrix ${F}^{T}$ of $F$ defines the map
${F}^{T}\colon\ \Delta \rightarrow \Delta^{'}$: 
$$
F^T(\ol x) = \ol y,\ \ \ y_w = \sum_{v\in V} f_{vw} x_v, 
$$
where $\overline{x} = \langle x_{v}\rangle$ and 
$\overline{y} = \langle y_{w}\rangle$. 
We are interested in the continuity of the map  
${F}^{T}\colon\ \Delta \rightarrow \Delta^{'}$ because 
this property plays an important role in generalized 
Bratteli diagrams, see 
Section \ref{sect_inverse limits}.

It is easy to check  that
${F}^{T}(\Delta_{a}) = {\Delta}'_{a}$, 
for every $0 < a \leq 1$. Indeed, let 
$\overline{x} \in \Delta_{a}$. 
Then we have
$$
\sum_{w \in W}y_{w} = \sum_{w \in W}{\sum_{v \in V}f_{vw}} 
  x_{v}
 = \sum_{v \in V}x_{v} \cdot  \sum_{w \in W}f_{vw}
=\sum_{v \in V}x_{v} = a.
$$

In general, the map ${F}^{T}$ is not continuous. 
We give a necessary and sufficient condition for the continuity of this map. 

Let
$\overline{g}_{v} := \langle f_{vw} : w \in W \rangle$  
be vectors of
$\Delta_{1}^{'}$ determined by $v$-th rows of the matrix 
$F$.

\begin{theorem}\label{t2}
The map $F^{T} \colon \Delta \rightarrow \Delta^{'}$  is continuous
if and only if $|{\overline{g}}_{v}| \rightarrow 0$ as $a(v) 
\rightarrow \infty$.
\end{theorem}

\begin{proof}
Sufficiency.
Take two vectors $\overline{x} =\langle x_{v} \rangle$ and 
$\overline{x'} =\langle {x}'_{v} \rangle$  
from the metric space $\Delta$ and set $\overline{y} =\langle y_{w} \rangle  = {F}^{T}(\overline{x})$,
$\overline{y}' = \langle {y}'_{w} \rangle  = {F}^{T}(\overline{x'})$.
We recall that the functions $a : V \to \N$ and $a' : W \to
\N$ enumerate elements of the sets $V$ and $W$. The 
distance $d$ is given by the formula
$d(\overline{x},\overline{x'}) = \sum_{v \in V}{\frac{1}{2^{a(v)}}  |x_{v}} -  {x}'_{v}|$. We  compute 
$$
\ba
d(\overline{y},\overline{y'}) = & \sum_{w \in W}\frac{1}{2^{a'(w)}} |y_{w} - {y}'_{w}|\\
= &  \sum_{w \in W}{\frac{1}{2^{a'(w)}} |\sum_{v \in V}f_{vw}  x_{v}} - \sum_{v \in V}f_{vw}  {x}'_{v}| \\
 \leq & \sum_{w\in W}{\frac{1}{2^{a'(w)}} \sum_{v \in V}f_{vw} {|x}_{v} - {x}'_{v}\ }|\\
= &  \sum_{v \in V}|x_{v} - {x}'_{v}| \sum_{w \in W} \frac{1}{2^{a^{'}(w)}}   f_{vw} \\
= & \sum_{v \in V}|x_{v} - {x}'_{v}| |{\overline{g}}_{v}|.
\ea
$$
Take $\varepsilon > 0$, choose $k_{0} \in \N$ 
such that $|{\overline{g}}_{v}|<\frac{\varepsilon}{2}$ 
if $a(v) >  k_{0}$, and set
$$
\delta = \frac{\varepsilon}{2} \left[ \sum_{a(v) \leq 
k_{0}}2^{a(v)}|{\overline{g}}_{v}| \right]^{- 1}.
$$ 
It is obvious that the inequality 
$d(\overline{x},\overline{x'}) < \delta$ implies $|x_{v} - {x}'_{v}| < \delta 2^{a(v)}$ for every $v\in V$. 

Finally, we have
$$
\ba
d(\overline{y}, \overline{y'}) &\leq  \sum_{v \in V}
|x_{v} - {x}'_{v}|   |{\overline{g}}_{v}|\\
& = \sum_{a(v) \leq k_{0}}|x_{v} - {x}'_{v}| 
| {\overline{g}}_{v}| + \sum_{a(v) > k_{0}}|x_{v} - {x'}_{v}|  |{\overline{g}}_{v}| \\
& \leq   \delta \sum_{a(v) \leq k_{0}}2^{a(v)} 
|{\overline{g}}_{v}| 
+ \frac{\varepsilon}{2} \sum_{a(v) > k_{0}}|x_{v} - {x}'_{v}|\\
& <  \frac{\varepsilon}{2} + \frac{\varepsilon}{2} = \varepsilon.
\ea 
$$
This proves that $F^{T}$ is continuous.

 Necessity. Assume that $F^{T}$ is continuous and let $\ol e^{(v)}$ be the basis vectors from $\Delta$. 
Obviously, $|\ol e^{(v)}| \rightarrow 0$ as $a(v) \to 
\infty$.
This implies that $| {\overline{g}}_{v} |  \rightarrow  0$ because $F^{T}(\ol e^{(v)}) =  {\overline{g}}_{v}$.

\end{proof} 

We will give more examples of continuous and discontinuous liner mappings $F$ related to some 
classes of generalized Bratteli diagrams in Sections 
\ref{sect bounded size} - \ref{ssect measures on B_infty}.

%%%%%
\section{Generalized Bratteli diagrams. Overview}
\label{sect_GBD}

In this section, we recall the main definitions and results 
concerning generalized Bratteli diagrams. For more details 
see \cite{BezuglyiJorgensen2022}, 
\cite{BezuglyiJorgensenKarpelSanadhya2023}. 

\subsection{Main definitions}

We will use the notation $|\cdot |$ for the cardinality of a set.

\begin{definition}\label{Def:generalized_BD} A 
\textit{generalized Bratteli diagram} is a graded graph 
$B = (V, E)$ such that the 
vertex set $V$ and the edge set $E$ are represented as 
partitions $V = \bigsqcup_{i=0}^\infty  V_i$ and $E = 
\bigsqcup_{i=0}^\infty  E_i$ satisfying the following 
properties: 
\vspace{2mm}

\noindent $(i)$ The number of vertices at each level 
$V_i$, $i 
\in \N_0$, is countably infinite (if necessary, we will 
identify each $V_i$ with $\Z$ or $\N$). The set $V_i$ is called the $i$th 
level of the diagram $B$. For all $i \in \N_0$, the set $E_i$
of all edges between $V_i$ and $V_{i+1}$ is countable.
\vspace{2mm}

\noindent $(ii)$ For every edge $e\in E$, we define the 
range and 
source maps $r$ and $s$ such that $r(E_i) = V_{i+1}$ and 
$s(E_i) = V_{i}$ for $i \in \N_0$. It is required that  
$s^{-1}(v)\neq \emptyset $ for all $v\in V$, and  
$r^{-1}(v)\neq\emptyset$ for all $v \in V\setminus V_0$. 
\vspace{2mm}

\noindent $(iii)$ For every vertex $v \in V_n$ and every $n 
\geq 1$, we have $|r^{-1}(v)| < \infty$.  
\end{definition} 

When we index the vertices 
at each level by $\Z$, the generalized Bratteli diagram 
$B$ is called 
\textit{two-sided infinite}, and when the vertices are indexed by 
$\N$ or $\N_0$, then we call $B$ 
\textit{one-sided infinite}.

The structure of a generalized Bratteli diagram $B$ is 
completely determined by a sequence of non-negative 
countably infinite matrices. 
For a vertex $v \in V_m$ and a vertex $w \in V_{n}$, denote 
by $E(v, w)$ the set of all finite paths between $v$ and 
$w$ (this set may be empty). 
Set $f'^{(n)}_{v,w} = |E(v, w)|$ for every $w \in V_n$ and 
$v \in V_{n+1}$. We denote
\begin{equation}\label{Notation:f^i}
    F'_n = (f'^{(n)}_{v,w} : v \in V_{n+1}, w\in V_n),\ \   
    f'^{(n)}_{v,w}  \in \N_0.
\end{equation} 
The matrices $F'_n$, $n \in \N_0$, are called 
\textit{incidence matrices}. 
We reserve the notation $F_n$ for the corresponding 
stochastic matrix (see below). The assumption $r^{-1}(v) <
\infty$ implies that in every row $v$, all but finitely many entries of
$F'_n$ are zeros. 
We will use the notation $B = B(F'_n)$.  If  $F'_n = F'$ 
for every $n \in \N_0$, then the diagram $B$ is called 
\textit{stationary.} We use $f'^{(n)}_{vw}$ in two cases:
for the $(vw)$-entry of $F'^n$ (stationary diagram) and for 
the $(vw)$-entry of the product $F'_{n-1}\ \cdots\ F'_0$ 
(non-stationary diagram). It will be clear from the 
context what case is considered.

To define the path space of a generalized Bratteli diagram
$B$, we consider a finite or infinite sequence of edges 
(it is called a \textit{path}) $(e_i: e_i\in E_i)$ such 
that $s(e_i)=r(e_{i-1})$. Denote the set of all infinite
paths starting at $V_0$ by $X_B$ and call it the 
\textit{path space}. For a finite path $\ol e = (e_0, ... , 
e_n)$, we write
$s(\ol e) = s(e_0)$ and $r(\ol e) = r(e_n)$. The set
$$
    [\ol e] := \{x = (x_i) \in X_B : x_0 = e_0, ..., x_n = e_n\}, 
$$ 
is called the \textit{cylinder set} associated with $\ol e$. 

The \textit{topology} on the path space $X_B$ is generated by
cylinder sets.
This topology coincides with the topology defined by the 
following metric on $X_B$: for $x = (x_i), \, y = (y_i)$, set 
$$
\mathrm{dist}(x, y) = \frac{1}{2^N},\ \ \ N = \min\{i \in \N_0 : 
x_i \neq y_i\}.
$$
The path space $X_B$ is a zero-dimensional Polish space and 
therefore a standard Borel space.

\subsection{Tail invariant measures}

In this paper, we consider \textit{tail invariant measures}
on the path space $X_B$ of a generalized Bratteli diagram
$B$. The term a \textit{measure} is always used for 
non-atomic positive Borel measures. We are mostly 
interested in \textit{full measures}, i.e., every cylinder 
set must be of positive measure.

\begin{definition}\label{Def:Tail_equiv_relation}
Two paths $x= (x_i)$ and $y=(y_i)$ in $X_B$ are called 
\textit{tail equivalent} if there exists an $n \in \mathbb{N}_0$ 
such that $x_i = y_i$ for all $i \geq n$. This notion defines a 
\textit{countable Borel equivalence relation} $\mathcal R$ on the
path space $X_B$ which is called the \textit{tail equivalence 
relation}.
\end{definition}

\begin{definition}\label{def: tail inv meas} Let 
$B =(V, E)$ be 
a generalized Bratteli diagram and $\mathcal R$ 
the tail equivalence relation on the path space $X_B$. 
A measure $\mu$ on 
$X_B$ is called \textit{tail invariant} if, for any cylinder sets
$[\ol e]$ and $[\ol e']$ such that $r(\ol e) = r(\ol e')$, we have
$\mu([\ol e]) = \mu([\ol e'])$.
\end{definition}

The set of probability tail invariant measures is denoted by $M_1(\mc R)$. We note that if a tail invariant Borel measure $\mu$ on $X_B$ takes 
finite values on all cylinder sets, then $\mu$ is uniquely determined by its values on cylinder sets in $X_B$. Thus, every such tail invariant measure can be characterized in terms of a sequence of 
positive vectors associated with vertices of each level, see Theorem  \ref{BKMS_measures=invlimits}. 

For every generalized Bratteli diagram, there exists a
sequence of \textit{Kakutani-Rokhlin towers}.

\begin{definition} \label{Def:Kakutani-Rokhlin} Let 
$B =(V, E)$ 
be a generalized Bratteli diagram, for $w \in V_n, 
n \in \N_0$, denote 
$$
X_w^{(n)} = \{x = (x_i)\in X_B : s(x_{n}) = w\}.
$$
The collection of all such sets forms a partition 
$\zeta_n$ of $X_B$ into  
\textit{Kakutani-Rokhlin towers}
corresponding to the vertices from $V_{n}$.
Each finite path $\ov e = (e_0, \ldots, e_{n-1})$ with 
$r(e_{n-1}) 
= w$, determines a ``level'' of this tower
$$
X_w^{(n)}(\ov e) = \{x = (x_i)\in X_B : x_i = e_i,\; i = 
0,\ldots, n-1 \}.
$$
Clearly, 
$$
X_w^{(n)} = \bigcup_{\ol e \in E(V_0, w)} X_w^{(n)}(\ov e),
$$ 
and the partition $\zeta_{n+1}$ refines $\zeta_n$.
\end{definition}

\begin{definition}\label{Def:Height}  For $v \in V_n$ and $v_0 \in
V_0$, we set $h^{(n)}_{v_0, v} = |E(v_0, v)| $ and define 
$$
H^{(n)}_v = \sum_{v_0 \in V_0} h^{(n)}_{v_0, v}, \ \ n \in \N.
$$ 
Set $H^{(0)}_v = 1$ for all $v\in V_0$.
This gives us the vector $H^{(n)} = \langle H^{(n)}_{v} : 
v \in V_n \rangle$ associated with every level $n\in \N_0$. 
Since $H^{(n)}_v = |E(V_0, v)|$, we call
$H^{(n)}_v$ the \textit{height of the tower} $X_v^{(n)}$ 
corresponding to the vertex $v\in V_n$.
 \end{definition} 

It is easy to see from the structure of a Bratteli diagram that, for every $n \in \mathbb{N}_0$ and every $v \in V_{n+1}$, we have
\begin{equation}\label{eq_H e6}
H_{v}^{(n + 1)} = \sum_{w \in V_{n}}{f'}_{vw}^{(n)}  H_{w}^{(n)}.
\end{equation}
Thus, the relation $H^{(n+1)} = F'_n H^{(n)}$ holds
for every $n$.
Remark that the fact that $r^{- 1}(w) < \infty$ 
for every $w \in V_{n}$ and $n \geq 1$ implies that 
$H_{w}^{(n)} < \infty$ 
and ${f'}_{vw}^{(n)} < \infty$.

 Next, we define the sequence of \textit{stochastic incidence 
matrices} $(F_n)$ that plays a key role in our quantitative
analysis of generalized Bratteli diagrams. We set
$F_n = (f_{vw}^{(n)} : v \in V_{n+1}, w \in V_n)$, where 
\be\label{eq_stoch matrix F_n}
f_{vw}^{(n)} =  {f'}_{vw}^{(n)}\ \cdot \frac{H_{w}^{(n)}}
{H_{v}^{(n + 1)}}.
\ee
Then we get from \eqref{eq_H e6} that 
\begin{equation}\label{e7}
\sum_{w\epsilon V_{n}}f_{vw}^{(n)} =1, \quad v \in\ 
V_{n + 1}.    
\end{equation}

\begin{thm}\label{BKMS_measures=invlimits}
 Let $B = (V,E)$ be a Bratteli diagram (generalized or classical) 
 with the sequence of incidence matrices $(F'_n)$. Then:
\begin{enumerate}

\item Let  $\mu$ be a tail invariant measure on $B$ which takes finite values on all cylinder sets. For every 
$n\in \N_0$, define two sequences of vectors $\ol p^{(n)} = \langle p^{(n)}_w : w \in V_n \rangle$, 
and $\ol q^{(n)} = \langle q^{(n)}_w : w \in V_n \rangle$,
where 
\begin{equation}\label{eq:def_p_n}
    p^{(n)}_w= \mu(X_w^{(n)}(\ov e)), \quad q^{(n)}_w= \mu(X_w^{(n)}), \ \ w\in V_n. 
\end{equation} 
Then the vectors $\ol p^{(n)}$ and $\ol q^{(n)}$ satisfy the relations 
\begin{equation}\label{eq:formula_p_n}
(F'_n)^{T} \ol p^{(n+1)} =\ol p^{(n)}, \quad F_n^{T} \ol q^{(n+1)} =\ol q^{(n)}, \quad n\geq 0, 
\end{equation}
or
\begin{equation}\label{eq_for p and q(e9)}
p_{w}^{(n)} = \sum_{v \in V_{n + 1}}f'^{(n)}_{vw}
p_{v}^{(n + 1)},   \quad
q_{w}^{(n)} = \sum_{v \in V_{n + 1}}f_{vw}^{(n)}
q_{v}^{(n + 1)}.     
\end{equation}

\item Suppose that $\{\ol p^{(n)}= (p_w^{(n)}) 
\}_{n \in \mathbb{N}_0}$ is a sequence of non-negative vectors such 
that $(F'_n)^{T}\ol p^{(n+1)} =\ol p^{(n)}$ for all $n \in 
\N_0$. Then there exists a uniquely determined tail invariant measure $\mu$ such that $\mu(X_w^{(n)}(\ov e))= 
p_w^{(n)}$ for 
$w\in V_n, n \in \mathbb{N}_0$.

\item Suppose that $\{\ol q^{(n)}= (q_w^{(n)}) 
\}_{n \in \N_0}$ is a sequence of non-negative vectors such 
that $F_n^{T}\ol q^{(n+1)} =\ol q^{(n)}$ for all $n \in 
\N_0$. Then there exists a uniquely determined tail invariant measure $\mu$ such that $\mu(X_w^{(n)})= 
q_w^{(n)}$ for 
$w\in V_n, n \in \mathbb N_0$.
\end{enumerate}
%&
\end{thm}  
The \textit{proof} of Theorem \ref{BKMS_measures=invlimits} 
is 
straightforward and can be found in 
\cite{BezuglyiKwiatkowskiMedynetsSolomyak2010} (for 
classical Bratteli diagrams) and 
\cite{BezuglyiJorgensen2022} (for generalized Bratteli 
diagrams). Theorem \ref{BKMS_measures=invlimits} is a form of
Kolmogorov consistency theorem. 

\subsection{Vershik map}\label{ssect Vmap}
In order to define a Borel dynamical system on a 
generalized Bratteli diagram, we will need the notion 
of an ordered generalized Bratteli diagram.
An \textit{ordered generalized Bratteli diagram} $
B = (B, V, >)$ is a generalized Bratteli diagram $B=(V,E)$
together with a partial order $>$ on $E$ such that edges 
$e,e'$ are comparable if and only if $r(e)=r(e')$ (see 
\cite{BezuglyiJorgensenKarpelSanadhya2023} for more 
details). 
We observe that a partial order "$>$" is a family (product) of linear orders "$>_{v}$" on the finite sets $r^{- 1}(v)$, $v \in V \setminus V_{0}$, that are pairwise independent.
A (finite or infinite) path $e= (e_0,e_1,..., 
e_i,...)$ is called \textit{maximal} (respectively 
\textit{minimal}) if every $e_i$ has 
a maximal (respectively minimal) number among all elements from 
$r^{-1}(r(e_i))$. We denote the sets of all infinite 
maximal and of all infinite minimal paths by $X_{\max}$ and 
$X_{\min}$ respectively.

For a diagram $B=(V,E,>)$, first define a Borel 
transformation 
$\varphi_B : X_B \setminus X_{\max} \rightarrow X_B 
\setminus X_{\min}$ as follows: given $x = (x_0, x_1,...)\in 
X_B\setminus X_{\max}$, 
let $m$ be the smallest number such that $x_m$ is not 
maximal. Let 
$y_m$ be the successor of $x_m$ in the finite set $r^{-1}
(r(x_m))$.
Then we set $\varphi_B(x)= (y_0, y_1,...,y_{m-1}, 
y_m,x_{m+1},...)$
where $(y_0, ..., y_{m-1})$ is the unique minimal path from 
$s(y_{m})$ to $V_0$. In such a way, the Borel map $\varphi_B$ is
a bijection from $X_B\setminus X_{\max}$ onto $X_B\setminus X_{\min}$.  Moreover, it follows from the definition that $\varphi_B$ is a homeomorphism. 
\begin{definition}
If the map $\varphi_B : X_B\setminus X_{\max} \to X_B\setminus X_{\min}$
admits a Borel bijective extension to the entire path space $X_B$, then this extension is called a \textit{Vershik map.} 
The corresponding Borel dynamical system $(X_B,\varphi_B)$ is 
called a generalized \textit{Bratteli-Vershik} system.
\end{definition} 

In some cases, we may be interested in surjective extensions (not necessarily bijections) of $\varphi_B$ to the entire path space $X_B$. 
In this context, we distinguish a class of Bratteli-Vershik maps called 
$p$-continuous maps (partially continuous).
For this, we need the notions the \textit{successor},  $Succ(x)$, of
$x \in X_{\max}$, and \textit{predecessor}, $Pred(y)$, of $y \in X_{\min}$.
With every $x = (x_n) \in X_{\max}$ and $y = (y_n) \in X_{\min}$, we associate the sequences of vertices 
$\ol v = (v_n)$ and $\ol w = (w_n)$ where $v_n = s(x_n)$ and 
$w_n = s(y_n)$.  
It is said that $y \in Succ(x)$ if for infinitely many $n$ there exists
$z \in V_{n + 1}$ and edges $e, e' \in r^{-1}(z)$ such that 
$s(e) =  v_{n}$, $s(e') = {w}_{n}$ and $e'$ is the successor of $e$ in 
the linear order defined on $r^{- 1}(z)$. 
Similarly, we define the set $Pred(y)$ for $y \in X_{\min}$. 
Of course, $x \in Pred(y)$ if and only if $y \in Succ(x)$. 

The following theorem clarifies the role of the defined notions.

\begin{theorem}\label{(3.6a)}
Let $\varphi$ be a Bratteli-Vershik extension of $\varphi_{B}$
such that $\varphi$ is continuous at $x$, where
$x \in X_{\max}$ and $Succ(x) \neq \emptyset$. 
Then $\varphi(x) \in Succ(x)$. On the other hand, if $\psi$ is a 
Bratteli-Vershik extension of $\varphi_{B}^{-1}$ and
$\psi $ is continuous at a path $y \in X_{\min}$, then
$\psi (y) = x \in Pred(y)$.
\end{theorem}

\begin{proof}
We can find two sequences of paths $\{ x^{(n)}\}$,
$\{ y^{(n)}\}$ such that
$\varphi_{B}(x^{(n)}) = y^{(n)}$, 
$\lim\limits_{n \to \infty} x^{(n)} = x$ and
$\lim\limits_{n \to \infty} y^{(n)}  = y \in Succ(x)$.
Then $\varphi (x) =
\lim\limits_{n \to \infty} \varphi( x^{(n)} )  =
\lim\limits_{n \to \infty}  \varphi_{B}( x^{(n)} )  =
\lim\limits_{n \to \infty} y^{(n)} = y \in Succ(x)$. 

Of course, the second part of the theorem can be proved in the same way.
\end{proof}

\begin{remark}\label{(3.6b)} 
The properties $Succ(x) \neq \emptyset$ and $Pred(y) \neq \emptyset$ 
imply $x \in cl \left( X_B \setminus X_{\max} \right)$ and
$y \in\ cl \left( X_B \setminus X_{\min} \right)$ respectively. 
In general, the opposite implications are not true, however, they are true
for a classic Bratteli diagram, because in this case the space $X_{B}$ is compact.
\end{remark}

We will say that an extension $\varphi$ of $\varphi_{B}$ 
(or $\psi$ of $\varphi_{B}^{-1}$) is \textit{$p$-continuous} if there is a
path $x \in X_{\max}$ (or $y \in X_{\min}$) such that
$\varphi(x)$ (or $\psi(y)$) is defined 
and $\varphi(x) \in Succ(x)$ (or $\psi(y) \in Pred(y)$). We will discuss examples of the sets $Succ(x)$ and $Pred(y)$ in Section~\ref{sect VershikMapPascal}.

\subsection{Classes of generalized Bratteli diagrams}

In this paper, we will also consider some particular classes 
of generalized Bratteli diagrams.

\begin{definition}\label{Def:irreducible_GBD} (1) 
Let $B = B(F_n)$ be
a generalized Bratteli diagram. If  $F_n = F$ 
and $V_n = V$ for every $n \in \N_0$, then the diagram $B$ 
is called \textit{stationary.} We 
will write $B = B(F)$ in this case.

(2) A generalized Bratteli diagram $B =(V,E)$, where 
all levels $V_i$ are identified with a set $V_0$ (e.g. 
$V_0 = \mathbb{N}$ or $\mathbb{Z}$),  is called  
\textit{irreducible} if 
for any vertices $i, j \in V_0$ and any level $V_n$ there exist 
$m 
> n$ and a finite path connecting $i \in V_n$ and $j \in V_m$. In 
other words, the $(j, i)$-entry of the matrix $F_{m-1} \cdots 
F_n$ 
is non-zero. Otherwise, the diagram is called \textit{reducible}.

\end{definition}

\begin{definition}\label{Def:BD_bdd_size} A generalized Bratteli 
diagram $B(F_n)$ is called of \textit{bounded size} if there exists 
a sequence of pairs of natural numbers $(t_n, L_n)_{n \in \N_0}$ 
such that, for all $n \in \mathbb{N}_0$ and all $v \in V_{n+1}$,
\begin{equation}\label{eq: Bndd size}
s(r^{-1}(v)) \in \{v - t_n, \ldots, v + t_n\} \quad \mbox{and} 
\quad \sum_{w \in V_{n}} f^{(n)}_{vw} = \sum_{w \in V_{n}} |E(w,v)| 
\leq L_n.
\end{equation} 
If the sequence $(t_n, L_n)_{n \in \N_0}$ is constant, i.e. $t_n = 
t$ and $L_n = L$ for all $n \in \N_0$, then we say that the 
diagram $B(F_n)$ is of \textit{uniformly bounded size}.
\end{definition} 

The following statement is taken from 
\cite{BezuglyiJorgensenKarpelSanadhya2023}.

\begin{lemma}\label{lemma_bdd_size_upper_cone}
Let $B=(V, E)$ be a generalized Bratteli diagram of bounded 
size. Let $n \in \mathbb{N}_0$, $v \in V_{n+1}$ and 
$E(V_0, v)$ be the set of all finite paths 
$\ov e = (e_0, \ldots, e_{n})$ such that $r(\ov e) = v$. 
Then 
$$
s(E(V_0, v)) \subset \left\{v - \sum_{i = 0}^n t_i, \ldots, 
v + \sum_{i = 0}^n t_i\right\}
$$
and
$$
H_v^{(n+1)} = |E(V_0, v)| \leq L_0 \cdots L_n,\ \ v \in 
V_{n+1}.
$$
\end{lemma}

\subsection{Subdiagrams and measure extension}
\label{ss subdiagrams} 
In this subsection, we give the basic definitions and include 
some results about subdiagrams of generalized Bratteli 
diagrams and the notion of measure extension. We use 
the approach developed in \cite{BezuglyiKarpelKwiatkowski2015} and \cite{AdamskaBKK2017} for standard Bratteli diagrams. Measure extensions from vertex subdiagrams for generalized Bratteli diagrams were considered in \cite{BezuglyiKarpelKwiatkowski2024}.

Let $B = (V, E)$ be a generalized Bratteli diagram. A \textit{subdiagram} $\ov B$ of $B$ is a (standard or generalized) Bratteli diagram $\ol B = (\ol V,\ol E)$, where $\ol V\subset V$ and $\ol E \subset E$ such that 
$\ol V = \bigcup_n \ol V_n$ and $\ol E = \bigcup_n \ol E_n$, where $\ol V_n \subset V_n$ and $\ol E_n \subset E_n$. In particular, we have $\ol V = s(\ol E)$ and $s(\ol E)  = r(\ol E) \cup \ol V_0$.

Let $\overline{B}$ be a subdiagram of a Bratteli diagram $B$. Then we have the sequence of incidence matrices 
$\{\overline{F}'_n\}_{n=0}^{\infty}$ of $\overline{B}$. There are two principal cases of subdiagrams, {\em edge subdiagrams} and {\em vertex subdiagrams}. By definition, an {\em edge} subdiagram is obtained from the diagram $B$ by ``removing'' some edges and leaving all vertices of $B$ unchanged. 
If $F'_n$ denotes the $n$-th incidence matrix of $B$, then we have  $\overline{F}'_n \leq F'_n$ for every $n \in \mathbb{N}_0$. 
We denote $\widetilde{F}'_n = F'_n - \overline{F}'_n$, 
i.e., $\widetilde{F}'_n$ is the matrix which shows the number of removed edges. 
Without loss of generality, we assume that  
$\overline{F}'_n < F'_n$ for infinitely many $n$.
In general, we do not require that the condition 
${f'}_{vw}^{(n)} > 0$ implies that $\ol {f'}_{vw}^{(n)} > 0$. But we implicitly assume that the path space 
$X_{\ol B}$ of an edge subdiagram is not trivial. Note that an edge subdiagram of a generalized Bratteli diagram is always a generalized Bratteli diagram.

A {\em vertex} subdiagram $\ov B = (\ol W, \ol E)$ of $B$ is a standard or generalized Bratteli diagram defined by a sequence $\overline W = \{W_n\}_{n\geq 0}$ of nonempty proper subsets $W_n \subset V_n$ and by the set of edges $\ol E_n \subset E_n$ whose source and range are in $W_{n}$ and $W_{n+1}$, respectively (only for $n = 0$, in the case of standard Bratteli diagrams we keep $W_0 = V_0 = \{v_0\}$). Thus, the incidence  matrix $\ol F'_n$ of $\ol B$ has the size $|W_{n+1}| \times |W_n|$, and it is represented by a block of $F'_n$ corresponding to the vertices from $W_{n}$ and $W_{n+1}$. We say that, in this case, $\ol W = (W_n)$ is the \textit{support} of $\ol B$. Set $W'_n = V_n \setminus W_n \neq \emptyset$ for all $n$.

It is easy to see that the path space $X_{\ol B}$ of a  subdiagram $\ol B$ of $B$ is 
a closed subset of $X_B$. 
On the other hand, there are closed subsets of $X_B$ which are not obtained as the path space of a Bratteli subdiagram. A closed subset $Z \subset X_B$ is the path space of a subdiagram if and only if $\mathcal R|_{Z \times Z}$ is an etal\'{e} equivalence relation (see \cite{GiordanoPutnamSkau2004} for details).

Let $\widehat X_{\ol B}:= \{y \in X_B : \exists x \in X_{\ov B} \mbox{ such that } x\mathcal{R}y\}$ be the subset of all paths in $X_B$ that are tail equivalent to paths from $X_{\ol B}$. In other words, $\widehat X_{\ol B}$ is the
smallest $\mathcal R$-invariant subset of $X_B$ containing 
$X_{\ol B}$, or an $\mathcal{R}$-saturation of $X_{\ov B}$ (see \cite{Kechris2024}). 
Let $\ov \mu$ be an ergodic tail invariant probability measure on $X_{\ol B}$. Then $\ol \mu$ can be canonically extended to the ergodic measure $\widehat{\ov \mu}$ on the space $\widehat X_{\ol B}$ by tail invariance, see \cite{BezuglyiKarpelKwiatkowski2015}, 
\cite{AdamskaBKK2017}, \cite{BezuglyiKarpelKwiatkowski2024}. More specifically, let the measure 
$\ol \mu$ be defined by a sequence of positive vectors 
$\{\ol p^{(n)} : n \in \N_0\}$ satisfying Theorem \ref{BKMS_measures=invlimits}, that is 
$(\ol F'_n)^T(\ol p^{(n+1)}) = 
\ol p^{(n)}, n \in \N_0$, where $\ol F'_n$ is the incidence
matrix for the subdiagram $\ol B$. Then, for every cylinder 
set $[\ol e] \subset X_B$ with $r(\ol e) = v \in \ol V_n$, we
set $\wh {\ol \mu}([\ol e]) = \ol p^{(n)}_v$. Then 
$\wh {\ol \mu}$ is defined on all clopen sets, and it
can be finally extended to an ergodic Borel measure on  $X_B$ by setting $\wh {\ol \mu} (X_B \setminus \wh X_{\ol B}) = 0$.

Let $\ol B$ be a vertex subdiagram of a generalized Bratteli
diagram $B$ defined by a sequence of subsets $(W_i)$. 
Denote by $\wh X_{\ol B}^{(n)}$ the set of all paths $x = (x_i)_{i = 0}^{\infty}$ from $X_B$ such that the finite path  $(x_0, \ldots, x_n)$ ends at a vertex $v$ of $\ov B$, and the tail  $(x_{n+1},x_{n+2},\ldots)$ belongs to $\overline{B}$,
i.e., 
\begin{equation}\label{n-th level}
\wh X_{\ol B}^{(n)} = \{x = (x_i)\in \wh X_{\ol B} : r(x_i) \in W_i, \ \forall i \geq n\}.
\end{equation}
It is obvious that $\wh X_{\ol B}^{(n)} \subset \wh X_{\ol B}^{(n+1)}$, $\wh X_{\ol B} = \bigcup_n \wh X_{\ol B}^{(n)}$,
and 
\begin{equation}\label{extension_method}
\widehat{\ov \mu}(\wh X_{\ol B}) = \lim_{n\to\infty} \widehat{\ov \mu}(\wh X_{\ol B}^{(n)}) = 
\lim_{n\to\infty}\sum_{w\in W_n}  H^{(n)}_w \ov p^{(n)}_w.
\end{equation}
This limit can be finite or infinite. If it is finite, then we 
say that $\ol\mu$ admits a finite measure extension 
$\widehat{\ov \mu}(\wh X_{\ol B}) < \infty$.

\begin{theorem}\label{TheoremIV_1} Let $\ol B$ be a vertex 
subdiagram of a generalized Bratteli diagram $B =(V, E)$ with
incidence matrices $(F'_n)$. Suppose that $\ol B$ is determined
by a sequence $(W_n)$ of nonempty proper subsets of $V_n$, $n \in \N_{0}$.
Let $\ol \mu $ be a probability tail invariant measure on the path space
$X_{\ol B}$ of $\ol B$ defined by its values $\ol p^{(n)}_w$ on cylinder sets. 
Then the following statements are equivalent:

\noindent 
(i) 
$\displaystyle
\widehat{\ol{\mu}}(\wh{X}_{\overline{B}}) < \infty;$ 

\noindent
(ii) 
$\displaystyle
\sum_{n = 0}^{\infty}\sum_{v \in W_{n + 1}}
\sum_{w\in {W}'_{n}} {f'}_{vw}^{(n)}  H_{w}^{(n)} 
\ol p_v^{(n+1)} < \infty;$ 

\noindent
(iii)
$\displaystyle
\sum_{n = 0}^{\infty}{\sum_{v \in W_{n + 1}}^{}{\wh{\overline{\mu}}\left({X}_{v}^{(n + 1)}\right)}} \sum_{w\in W'_{n}} f_{vw}^{(n)} < \infty,
$
where $f_{vw}^{(n)}$ are the entries of the stochastic matrix
$F_n$ and  $W_{n}^{'} = V_{n} \setminus W_{n},\ n = 1, 2, \ldots$

\end{theorem}

The proof of Theorem~\ref{TheoremIV_1} can be found in \cite{AdamskaBKK2017} for the classic
Bratteli diagrams and in \cite{BezuglyiKarpelKwiatkowski2024} for the generalized Bratteli diagrams. The proof of the following result can be also found in \cite{AdamskaBKK2017} for standard
Bratteli diagrams, the same reasoning works for generalized diagrams.

\begin{theorem}\label{Theorem_IV_2}
Let
$\overline{B}  = \left(\overline{W}, \overline{E} \right)$ 
be a vertex subdiagram of a generalized Bratteli diagram  
$B = (V, E)$.  Suppose that 
$$
\sum_{n = 0}^{\infty} \sup_{v \in W_{n + 1}}\left(\sum_{w \in W_{n}^{'}} f_{vw}^{(n)}\right) < \infty.
$$ 
Then, for any probability measure $\overline{\mu}$ on
$\overline{B}$, the measure extension 
$\wh{\overline{\mu}}(\wh{X}_{\overline{B}})$ is finite. 
\end{theorem}

Now we consider an edge subdiagram $\overline{B}$ of a 
generalized Bratteli diagram $B$ which is 
defined by a sequence of incidence matrices
$\overline{F}'_{n}$ (the entries of ${\overline{F}}_{n}'$
show the number of remaining edges in $\ol B$
after removing some of them). 
The path space $X_{\ol B}$ consists of infinite paths 
$x = (x_n)$ where every $x_n$ is an edge
in $\ol B$. 

Let $\overline{\mu}$ be a tail invariant measure on  $\overline{B}$. Every such measure can be extended to a 
(finite or infinite) measure $\widehat{\overline{\mu}}$ 
on $B$ by tail invariance. 
It is supported by the set ${\widehat{X}}_{\overline{B}}$,  and, as in (\ref{extension_method}), we find that 
\begin{equation}
\widehat{\overline{\mu}}({\widehat{X}}_{\overline{B}}) =  \lim_{n \to \infty}{\sum_{w \in  V_{n}} H_{w}^{(n)}}  {\overline{p}}_{w}^{(n)}
\end{equation}
where $H_{w}^{(n)}$ is the number of finite paths $\ol e$ in $X_B$
terminating at $w \in V_n$ and $\ol p^{(n)}_w$ is the value
of the measure $\ol\mu$ on $[\ol e]$.

\begin{prop}\label{prop-edge meas ext}
For an edge subdiagram $\overline{B}$ of a (classic or generalized) Bratteli diagram $B$, we have
\begin{eqnarray*}
\widehat{\overline{\mu}}(\widehat{X}_{\overline{B}})  &=& \widehat{\overline{\mu}}(\widehat{X}_{\ol B}^{(1)}) + \sum_{n=1}^\infty \sum_{v\in V_{n+1}} \sum_{w\in V_{n}} \widetilde{f'}_{v,w}^{(n)} H_w^{(n)} \overline{p}_v^{(n+1)},
\end{eqnarray*} 
where $\widetilde{f'}_{vw}^{(n)} = {f'}_{vw}^{(n)} - 
\ol {f'}_{vw}^{(n)}$.
\end{prop}

This proposition is proved exactly in the same way as the corresponding
result in \cite{AdamskaBKK2017}. Thus, if $ \widehat{\overline{\mu}}(\widehat{X}_{\ol B}^{(1)}) < \infty$, we have
\begin{equation}\label{add3_8}
\widehat{\overline{\mu}}( {\widehat{X}}_{\overline{B}}) < \infty \ \Longleftrightarrow \ 
\sum\limits_{n = 1}^{\infty}\sum\limits_{v \in V_{n + 1}}
{\sum\limits_{w \in V_{n}}}{\widetilde{f'}}_{vw}^{(n)} H_{w}^{(n)}  {\overline{p}}_{v}^{(n + 1)}
< \infty.
\end{equation}
Note that for a standard Bratteli diagram $B$ we always have $\widehat{\overline{\mu}}(\widehat{X}_{\ol B}^{(1)}) < \infty$.

We can give also a direct formula for the value of the measure
$\widehat{\overline{\mu}}$ on the cylinder sets. 
Namely, for  $n \in \N$ and  $w \in  V_{n}$, we have
\begin{equation}\label{eq: value of ext neas on e}
\widehat{\overline{\mu}}([\ol e]) =
\lim_{m \to \infty} \left\lbrack {\sum\limits_{v \in  
V_{n + m}}^{}{{g'}_{vw}^{(n,m)}} 
{\overline{p}}_{v}^{(n + m)} } \right\rbrack,\quad 
w = r(\ol e).
\end{equation}
The formula in \eqref{eq: value of ext neas on e} 
is valid for both edge and vertex subdiagrams.

%%%%%Section 4
\section{Tail invariant measures and inverse limits}
\label{sect_inverse limits}

This section is devoted to a detailed description of tail 
invariant measures on the path space of a generalized 
Bratteli diagram in terms of inverse limits of convex
closed sets. We will use the notation from Section
\ref{sect_GBD}.

\subsection{Inverse limits define tail invariant measures}
\label{ssect inverse limits}
Let $\mu \in M_1(\mc R)$ be a probability tail invariant 
measure on the path space $X_B$ of a generalized Bratteli 
diagram $B =(V, E)$. By Theorem \ref{BKMS_measures=invlimits},  
 $\mu$ is completely determined by 
a sequence of infinite vectors $\ol p^{(n)}, n \in \N_0$,
where the entries of $\ol p^{(n)}$ are the measure of the
finite paths $\ol e$ from $E(V_0, w)$, 
$p_{w}^{(n)}$= $\mu ([\ol e])$,  
such that
$$
p_{w}^{(n)} = \sum_{v \in V_{n + 1}}{f'}_{vw}^{(n)} p_{v}^{(n + 1)}. 
$$
The measure $\mu$ is also completely determined by 
a sequence of infinite vectors $\ol q^{(n)}, n \in \N_0$, where $q^{(n)}_w= \mu(X_w^{(n)})$ is a measure of a tower corresponding to a vertex $w \in V_n$
and
$$
q_{w}^{(n)} = \sum_{v \in V_{n + 1}}f_{vw}^{(n)}
q_{v}^{(n + 1)}.
$$
We have $\mu (X^{(n)}_w) = H^{(n)}_w p_{w}^{(n)}$. Since for every level $n$, the towers $X_w^{(n)}$ form a partition of $X_B$, we also have
\begin{equation}\label{eq_prob vector q e8}
  \sum_{w \in V_{n}}\mu (X^{(n)}_w) = \sum_{w \in V_{n}}{H_{w}^{(n)}  p_{w}^{(n)}} = 1.  
\end{equation}
Therefore, every vector $\ol q^{(n)}$ is probability, 
see \eqref{eq_prob vector q e8}. Moreover, we will make our 
notation more precise and write that the vectors
$\ol q^{(n)}$, $n \in \N_0,$ belong to $\De_1^{(n)} : =
\{\overline{x} = \langle x_{w} \rangle : \ w \in V_{n},\
\sum_{w \in V_{n}}x_{w} = 1,\ x_{w} \geq 0  \}$. Clearly, 
$\Delta_1^{(n)}$ is isomorphic to  $\Delta_1$. The index
$n$ shows that this set is related to the $n$-level of the
Bratteli diagram. Similarly,
the vectors $(\ol p^{(n)})$ (defined above) are considered 
in $\De^{(n)} = \{\ \ol x = \langle x_v\rangle : \sum_{v\in 
V_n} x_v \leq 1, x_v \geq 0 \}$, where the set of indices is 
$V_n$.

In what follows we will consider the maps defined by  
stochastic incidence 
matrices $(F_n)$ and describe the set $M_1(\mc R)$ of
probability tail invariant measures in terms of inverse 
limits. 
\\

(A) It follows from Theorem \ref{BKMS_measures=invlimits} that there 
exists a sequence of maps
\begin{equation}\label{eq_inverse limit for p e11}
\Delta_{1}^{(0)}  \stackrel{F_{0}^{T}} \longleftarrow   
  \Delta_{1}^{(1)} \stackrel{F_{1}^{T}} \longleftarrow  
  \Delta_{1}^{(2)}\stackrel{F_{2}^{T}} \longleftarrow \  
  \cdots
 \end{equation} 

The following result is a straightforward corollary of Theorem~\ref{BKMS_measures=invlimits}.
\begin{corol}
The set $M_1(\mc R)$ of all probability tail invariant 
measures on $X_B$ is identified with the inverse limit
of the sets $(\De_1^{(n)}, F_n^T)$:
$$
M_1(\mc R) = \varprojlim_{n\to \infty} (\De_1^{(n)}, F_n^T).
$$
\end{corol}

Indeed, this result holds because every $\mu \in M_1(\mc R)$
is uniquely determined by a sequence of vectors 
$\{\ol q^{(n)}\}$ satisfying \eqref{eq_for p and q(e9)}. 
Recall also that the sets $\Delta_{1}^{(n)}$ are convex 
subsets of  $\Delta^{(n)}\ \subset I^{V_{n}}$.
These sets are, in general, not closed.
\\

(B) We consider another sequence of maps that 
determines elements of the set $M_1(\mc R)$. We observe that the map defined by $F_{n}^{T}$ maps 
$\Delta^{(n + 1)}$ into $\Delta^{(n)}$. Indeed, if 
 $\overline{x} =\langle x_{v}  :  v \in V_{n + 1} 
 \rangle \in  \Delta^{(n + 1)}$, then  
 $\sum_{v \in V_{n + 1}}x_{v} \leq 1$. Hence,  
$F_{n}^{T} \overline{x} = \ol y = \langle y_{w} : 
w \in V_n \rangle$ and 
$$
\sum_{w\in V_n} y_w = \sum_{w\in V_n} 
\sum_{v \in V_{n + 1}}{f_{vw}^{(n)} x_{v}} 
= \sum_{v \in V_{n + 1}}x_{v} \sum_{w \in V_{n}}f_{vw}^{(n)} 
\leq 1.
$$

This means that we have also the following sequence of maps
of compact convex sets $\De^{(n)}$:
\begin{equation}\label{eq_inverse limit 2 for p e12}
\Delta^{(0)}  \stackrel{F_{0}^{T}} \longleftarrow   
  \Delta^{(1)} \stackrel{F_{1}^{T}} \longleftarrow  
  \Delta^{(2)}\stackrel{F_{2}^{T}} \longleftarrow \  
  \cdots
 \end{equation} 

We formulate the result. 
\begin{lemma} Let $M_1(\mc R)$, $\De^{(n)}$, $F_n^T$ be as above. Then
$$
M_1(\mc R) \subset \varprojlim_{n\to \infty} 
(\De^{(n)}, F_n^T).
$$
\end{lemma}

The lemma follows immediately from the fact that a 
sequence of probability non-negative vectors $\{\ol q^{(n)}\}$ defines a measure $\mu 
\in M_1(\mc R)$ if and only if it satisfies 
\eqref{eq_for p and q(e9)}.
Note that a sequence 
$\{ \overline{q}^{(n)}\} \in \varprojlim 
\left( \Delta^{(n)}, F_{n}^{T} \right)$
determines a probability tail invariant measure on 
$B = (V, E)$ if and only if
${\overline{q}}^{(n)\ } \in \Delta_{1}^{(n)}$ 
for $n = 0, 1,  \ldots$.
\\

(C)  For every $n \in \N_0$ and $m \in \N$, we define the set 
$$
\Delta^{(n,m)} := F_{n}^{T}\cdot  \ldots \cdot  
F_{n + m - 1}^{T}(\Delta^{(n + m)}).
$$
Clearly,  $\Delta^{(n,m)}$ is a subset of
$\Delta^{(n)}$ for every $m$ and
\begin{equation}\label{eq_Delta^(n,i) e13}
\Delta^{(n,m)} = \ F_{n}^{T}(\Delta^{(n + 1,m - 1)}),
\quad
\Delta^{(n,1)} \supset \Delta^{(n,2)} \supset \ldots \supset 
\Delta^{(n,m)} \supset \ldots.
\end{equation}
Hence, we can define
\begin{equation}\label{eq_intersection e14}
\Delta^{(n,\infty)} =\bigcap_{m = 1}^{\infty}\Delta^{(n,m)}.
\end{equation}
Then \eqref{eq_Delta^(n,i) e13} implies that 
\begin{equation}\label{eq_Delta^(n,infty) e15}
\Delta^{(n,\infty)} \supset  F_{n}^{T}(\Delta^{(n + 1,\infty)}),\ \ n = 1, 2, \ldots.
\end{equation}
Relations \eqref{eq_Delta^(n,infty) e15} define the following sequence of maps
\begin{equation}\label{eq_inverse limit e16}
\Delta^{(0,\infty)}  \stackrel{F_{0}^{T}} \longleftarrow   
  \Delta^{(1,\infty)} \stackrel{F_{1}^{T}} \longleftarrow  
  \Delta^{(2,\infty)} \stackrel{F_{2}^{T}} \longleftarrow \  
  \cdots
 \end{equation} 

Thus, we obtain the following lemma.
\begin{lemma}\label{lem_Delta inf meas}
Let $M_1(\mc R)$, $F_{n}^{T}$, $\Delta^{(n)}$, $\Delta^{(n,\infty)}$ be as above. Then
\begin{equation}\label{e17}
\varprojlim_{n\to \infty}  \left( \Delta^{(n)},\ F_{n}^{T} 
\right) = \varprojlim_{n\to \infty}  
\left(\Delta^{(n,\infty)}, F_{n}^{T} \right).
\end{equation}
The set $M_1(\mc R)$ is a subset of $\varprojlim  \left(\Delta^{(n,\infty)}, F_{n}^{T} \right).$
A sequence $\{{\overline{q}}^{(n)}\} \in 
\varprojlim \left(\Delta^{(n,\infty)}, F_{n}^{T} \right)$ 
determines a probability tail invariant measure on 
$B = (V, E)$ if and only if
${\overline{q}}^{(n)\ } \in \Delta_{1}^{(n)}$. 
\end{lemma}

\begin{remark}\label{r4}
For every $n \geq 0$, the sets $\Delta^{(n,m)}$, 
$m = 1, 2, \ldots$  and 
$\Delta^{(n,\infty)}$ are convex. 
In general, they are not closed. The sets
$\Delta^{(n,m)}$ are not empty while the sets 
$\Delta^{(n,\infty)}$ may be empty.

If all maps $F_{n}^{T}, n \in \N_0,$ are continuous,
then all the
above sets are not empty and closed. In this case, we are in 
the setting of the Krein-Milmann theorem. 
To apply this theorem in a general case, we need 
to use another sequence of maps which is considered in (D) 
below. 
\end{remark}

(D) 
Let $cl(\Delta^{(n,m)})$ denote the closure of the set 
$\Delta^{(n,m)}$ in the compact set $\Delta^{(n)}$. 
The sets
$cl(\Delta^{(n,m)})$ form a nested sequence, i.e., 
$cl(\Delta^{(n,m)}) \supset cl(\Delta^{(n,m+1)})$, 
and we can define
\begin{equation}\label{e18}
\Delta^{(n,\infty,cl)} := 
\bigcap_{m = 1}^{\infty}cl(\Delta^{(n,m)}),\ \ n = 0, 1, 2, \ldots .
\end{equation}
The sets $cl(\Delta^{(n,m)})$ and 
$\Delta^{(n,\infty,cl)}$ are closed and not empty. 
The difference with the previous cases is that we cannot 
claim that the sequence 
$$
\{ (\Delta^{(n,\infty,cl)}, F_{n}^{T}) \}
$$
forms the ``classical'' inverse limit, in general. The 
reason is that  $F_{n}^{T}
(\Delta^{(n+1,\infty,cl)})$ is not a subset of 
$\Delta^{(n,\infty,cl)}$.  
However, we again can consider the inverse limit of 
$\{ (\Delta^{(n,\infty,cl)}, F_{n}^{T}) \}$ as the set
of all sequences of probability vectors 
$\{\ol q^{(n)}\}$ such that $F_n^T (\ol q^{(n+1)}) = \ol q^{(n)}$ for all $n$. We remark that if $\{\ol q^{(n)}\}$ 
determines a probability measure $\mu$ then 
$\{\ol q^{(n)}\}$ 
satisfies the above condition.
  
\begin{remark}\label{rem_r6} 
If all maps $F_{n}^{T}$ are continuous, then 
$cl(\Delta^{(n,m)}) = \Delta^{(n,m)}$ and 
$\Delta^{(n,\infty)} = \Delta^{(n,\infty,cl)}$. This means 
that the inverse limit of the sets $\{ 
(\Delta^{(n,\infty,cl)}, F_{n}^{T}) \}$ exists and 
$$
\varprojlim_{n\to \infty}  
\left(\Delta^{(n,\infty, cl)}, F_{n}^{T} \right) = 
\varprojlim_{n\to \infty}  
\left(\Delta^{(n,\infty)}, F_{n}^{T} \right). 
$$
Then every sequence $\{\ol q^{(n)}\} \in \varprojlim_{n\to 
\infty}  \left(\Delta^{(n,\infty, cl)}, F_{n}^{T} \right)$
determines uniquely a probability measure $\mu \in 
M_1(\mc R)$ if and only if ${\overline{q}}^{(n)} \in \Delta_{1}^{(n)}$. 
\end{remark}

\subsection{Finite products of matrices 
\texorpdfstring{$F_n$}{F_n}}
Let $B = (V, E)$ be a generalized Bratteli diagram 
and suppose that the sequences of infinite matrices 
$(F'_n)$ and $(F_n)$ are defined as in Subsection 
\ref{ssect inverse limits}, see 
\eqref{eq_stoch matrix F_n}. Then we set 
$$
G'^{(n,m)} := {F'}_{n + m - 1} \cdots   F'_{n},
$$
$$
{G}^{(n,m)} := F_{n + m - 1} \cdots   F_{n}. 
$$
The entries of $G'^{(n,m)}$ and ${G}^{(n,m)}$ are denoted 
by ${g'}_{vw}^{(n,m)}$ and $g_{vw}^{(n,m)}$, respectively,
where $ v \in V_{n + m},\ w \in V_{n}$. 
It can be easily checked that, for $v \in V_{n+m}$, 
$$
H_v^{(n+m)} = \sum_{u \in V_n} {g'}_{vu}^{(n,m)} H_u^{(n)}
$$
and therefore 
\begin{equation}\label{eq_g(n,m) and g'(n,m) e21}
g_{vw}^{(n,m)} = {g'}_{vw}^{(n,m)} \frac{H_{w}^{(n)}}
{H_{v}^{(n + m)}}.
\end{equation}
The rows of the matrix $G^{(n,m)}$  can be written as 
follows:
$${\overline{g}}_{v}^{(n,m)} =
\frac{1}{\sum_{s \in V_{n}}{g'}_{vs}^{(n,m)} H_{s}^{(n)}} \cdot \langle {g'}_{vw}^{(n,m)} H_{w}^{(n)} :\ w \in V_{n} \rangle.
$$
Using the above notation, we can write
$\Delta^{(n,m)} = G^{(n,m)T}(\Delta^{(n + m)})$. 

If a Bratteli diagram $B = (V, E)$ is \textit{stationary}, 
i.e., $V_{n} = V$ and $F_{n} =  F$, $n \in \N_0$,  
then ${G'}^{(n,m)} = F^{m},\ m = 1, 2, \ldots$, and
the equality \eqref{eq_g(n,m) and g'(n,m) e21} 
has the form
\begin{equation}\label{e22b}
g_{vw}^{(n,m)} = \frac{f_{vw}^{(m)}  H_{w}^{(n)}}
{\sum_{s \in V_{n}}^{}f_{vs}^{(m)}  H_{s}^{(n)}}.
\end{equation}
Assuming $H_{w}^{(0)} = 1$ for every $w \in V_{0}$, we see 
that $H_{w}^{(n)} = \sum_{s \in V_{0}}f_{ws}^{(n)}$, 
$w \in V_{n}$, $n \geq 1$.

Let $\overline{g}_{v} = \langle g_{vw}^{(n,m)} : 
\ w \in V_{n} \rangle$ be the $v$-th row of $G^{(n,m)}$, 
$v \in V_{n + m}$. Then $\overline{g}_v$ can be 
viewed as 
vectors from the set $\Delta^{(n)}$. If $\{\ol e_v : v 
\in V_{n+m}\}$ is 
the standard basis in $\Delta^{(n + m)}$, then  
\be\label{eq-rows g_v}
\overline{g}_{v} = G^{(n,m)T}(\ol e_{v}^{(n + m)}).
\ee 

Let $\Delta^{(n)}_m(\{\overline{g} _{v}\})$ be the closed 
convex 
hull generated by the vectors $\overline{g}_{v}, v \in 
V_{n + m}$, where $m \in \N$ is fixed. 
It follows from the above arguments that the following
lemma holds.

\begin{lemma}\label{lem_closure} Let $\Delta^{(n)}_m(\{\overline{g} _{v}\})$ and $\Delta^{(n,m)}$ be as above. Then
$$
\Delta^{(n)}_m(\{\overline{g} _{v}\}) = cl(\Delta^{(n,m)}). 
$$
\end{lemma}

Recall that we can apply Theorem \ref{thm t1} to find a 
representation of vectors from the set 
$\Delta^{(n)}_m(\{\overline{g} _{v}\})$, where $n$ and $m$ 
are fixed:
$$
\Delta^{(n)}_m(\{\overline{g} _{v}\}) =
\left\{\int_{cl(\{\overline{g}_v\})} \ol z\; 
d\mu(\ol z) : \mu \in M_1(cl(\{\overline{g}_v\})) 
\right\}.
$$
In the above formula, we take the closure in $\Delta^{(n)}$
of the vectors that came from the level $V_{n+m}$.

Fix $n$ and consider the set of all sequences 
$\{\ol g_{v_m}\}$, where $v_m$ is any vertex from $V_{n+m}$. 
We denote by  
$L^{(n)}(\{ \overline{g}_{v}\})$ 
the set of all limit points of all convergent sequences, 
i.e., 
$$ 
L^{(n)}(\{ \overline{g}_{v}\}) = \{ \ol x \in \Delta^{(n)} : 
\ol x = \lim_{m\to \infty}
\ol g_{v_{m}},\ v_{m} \in V_{n + m} \}.
$$ 

It turns out that the set $L^{(n)}(\{ \overline{g}_{v}\})$ 
can be used to 
describe the vectors from $\Delta^{(n,\infty,cl)}$.

\begin{theorem}\label{thm int over L t3}
Let $\Delta^{(n,\infty,cl)}$ and $L^{(n)}(\{ \overline{g}_{v}\}$ be as above. Then
\begin{equation}\label{eq_int over L e22}
\Delta^{(n,\infty,cl)} = 
\left\{ \int_{L^{(n)}(\{ \overline{g}_{v}\})} \overline{z}
\; d\mu(\ol z)\ :\ \mu \in 
M_1(L^{(n)}(\{ \overline{g}_{v}\})) \right\}. 
\end{equation}

\end{theorem}

\begin{proof}
We first observe that, for every $m \geq 1$, the closure 
$cl(\Delta^{(n,m)})$ can be characterized as the set of all 
limit points of sequences $\{\overline{y}^{(m)}(l) \}$ (as
$l \rightarrow \infty$), where ${\overline{y}}^{(m)}(l)$ 
is a finite convex combination of vectors 
$\overline{g}_{v},\ v \in V_{n + m}$, see Lemma 
\ref{lem_closure}.

Recall that for every $k = 1, 2, \ldots$
$$\Delta^{(n,\infty,cl)} = \bigcap_{m = k}^{\infty} (cl( 
\Delta^{(n,m)}))$$
Hence, for each 
$\overline{y} \in \Delta^{(n,\infty,cl)}$, we can select a 
sequence ${ \{ \overline{y}}^{(m)} \},\ m \geq k $, 
such that ${\overline{y}}^{(m)}$ 
is a finite convex combination of the vectors
${\overline{g}}_{v},\ v \in V_{n + m}$,  
and ${\overline{y}}^{(m)} \rightarrow \overline{y}$ 
in $\Delta^{(n)}$ as $m \to \infty$. 

Denoting the closure of the vectors $\{\ol g_v : v \in 
V_{n+m}, m\geq k\}$ by $Z_k$ and applying Theorem 
\ref{thm t1} again, we have for every $\overline{y} \in 
\Delta^{(n,\infty,cl)}$
\be\label{eq y as integral}
\overline{y} = \int_{Z_k} \overline{z} \; d\mu_{k} 
(\overline{z}),
\ee
where  $\mu_{k}$ is a Borel probability measure on the set
$Z_k$. Since $Z_ 1 \supset Z_2 \supset \cdots $, we can 
assume that each measure $\mu_k$ can be considered on 
the set $Z_1$, $k \geq 1$.

Let $\mu$ be a limit measure defined by a subsequence 
of $\{\mu_{k} : k \geq 1\}$. Then $\mu(Z_k) =1$ for all 
$k \in \N$ which implies that 
$$
\mu(\bigcap_{k \geq 1} Z_k) =1.
$$
Moreover, relation \eqref{eq y as integral} holds for
the measure $\mu$.

It is easy to see from the definition of the sets
$Z_k$ and $L^{(n)}(\{ \overline{g}_{v}\})$ that
$$
\bigcap_{k=1}^{\infty}Z_k = L^{(n)}(\{ \overline{g}_{v}\}).
$$
Therefore, we can conclude that for every $\ol y$ from
$\Delta^{(n,\infty,cl)}$
$$
\ba 
\overline{y} = & \int_{Z_k} \ol z\; d\mu (\overline{z})\\
= &  \int_{\bigcap_k Z_k} \ol z\; d\mu (\overline{z})\\
= & \int_{L^{(n)}(\{ \overline{g}_{v}\})} \ol z\; d\mu 
(\overline{z}). \qedhere \\
\ea
$$
\end{proof}

\begin{remark}\label{rem_r7r8}
We observe that the extreme points of 
${\Delta}^{(n,\infty,cl)}$ are contained in 
$L^{(n)}(\{ \overline{g}_{v}\})$.

To determine the set $L^{(n)}(\{ \overline{g}_{v}\})$, we 
must find the set of all limit vectors of the sequences 
of the form 
$\{ {\overline{g}}_{v_{m}},\ v_{m} \in V_{n + m}\}$. 
To do this, we can first find the set of all limit vectors 
of sequences of the so-called ``normalized'' vectors 
$\{ \overline{y}_{v_{m}},\ v_{m} \in V_{n + m} \}$.
This approach is discussed below.
\end{remark}

We recall that the entries of finite products $G'^{(n,m)}$ 
of incidence matrices $F'_n$ are denoted by 
${g'}_{v,w}^{(n,m)}$, where $w \in V_n$ and $v \in V_{n+m}$.
For $v \in V_{n+m}$, define the vector $\ol y_v^{(n,m)}$ by 
setting 

\begin{equation}\label{e23}
\ol y_{v}^{(n,m)} = \frac{1}{\sum_{w \in V_{n}}
{g'}_{vw}^{(n,m)}}  {\overline{g}'}_{v}^{(n,m)}, 
\ v \in V_{n + m},
\end{equation}
where ${\overline{g}'}_{v}^{(n,m)}$ is the $v$-th row of
$G'^{(n,m)}$ and ${g'}_{vw}^{(m,n)}$ is the $w$-th entry 
of this vector. Clearly, every vector 
$\overline{y}_{v}^{(n,m)}$ is probability. 

The following statement clarifies the meaning of this
normalization.

Let 
\be\label{eq-correction}
\ol y^{(n, \infty)} = \lim_{m \to \infty}\ol 
y_{v}^{(n,m)}
\ee 
where $v = v_m \in V_{n+m}$, see \eqref{e23}. 
Denote by $P^{(n)}$ the set of vectors $\ol y^{(n, \infty)}$
satisfying the two conditions:
\be\label{eq-conditions P(n)}
 \sum_{w \in V_{n}}{y_{w}^{(n, \infty)} H_{w}^{(n)}} < 
 \infty
\ee
and
\be\label{eq limit vectors p(infty)}
\lim_{m \to \infty} \sum_{w \in  V_{n}}
\left[ \frac{{g^{'}}_{vw}^{(n, m)}}{\sum_{u \in  V_{n}}
{g^{'}}_{vu}^{(n, m)}} \right]  H_{w}^{(n)} 
= \sum_{w \in V_{n}}{y_{w}^{(n, \infty)} H_{w}^{(n)}}.
\ee

\begin{theorem}\label{thm 4.11 edited} 
Let $\overline{g}_{v} = G^{(n,m)T}(\ol e_{v}^{(n + m)})$, where 
$v \in V_{n+m}$ and $m \geq 1$. Then 
$L^{(n)}(\{ \overline{g}_{v}\})$ consists of all 
vectors $\overline{q}^{(n,\infty)} = \langle  
q^{(n. \infty)}_w : \ w \in V_n \rangle$ such that 
\begin{equation}\label{e24}
{q}^{(n,\infty)}_w = \frac{1}{\sum_{u \in V_{n}}
{y_{u}^{(n, \infty)}  H_{u}^{(n)}}} \cdot 
 y_{w}^{(n, \infty)}  H_{w}^{(n)}
\end{equation}
where $\ol y^{(n, \infty)} = \langle  y^{(n,\infty)}_w : \
w \in V_n\rangle \in P^{(n)}$.
\end{theorem}

\begin{proof}
Let $\ol y^{(n, \infty)} \in P_n$. It follows from \eqref{e23} 
that 
$$
\lim_{m \to \infty} \frac{{g'}_{vw}^{(n, m)}}
{\sum_{u\in V_{n}}^{}{g'}_{vu}^{(n, m)}}
= y_{w}^{(n,\infty)}, \ \ \ w \in V_{n}.
$$ 
Then, by the definition of the set $P^{(n)}$ and 
\eqref{eq limit vectors p(infty)}, we have   
$$
\lim_{m \to \infty}{\frac{\sum_{u \in V_{n}}{g'}_{vu}^{(n, m)}}{\sum_{u \in V_{n}}{{g'}_{vu}^{(m, n)}
\cdot H_{u}^{(n)}}}} = 
\frac{1}{\sum_{u \in V_{n}}{y_{u}^{(n, \infty)} H_{u}^{(n)}}}.
$$

Furthermore, we use the fact that ${G'}^{(n,m)} H^{(n)} =
H^{(n+m)} $ and relation \eqref{eq_g(n,m) and g'(n,m) e21}
to compute 
\begin{equation}\label{e26-1}
\ba
{g}_{vw}^{(n, m)} =\  & {g'}_{vw}^{(n, m)}
\frac{H^{(n)}_w}{H_v^{(n+m)}} \\
=\ & {g'}_{vw}^{(n, m)}
\frac{H^{(n)}_w}{\sum_{u\in V_n}  {g'}_{vu}^{(n, m)}
H_u^{(n)}}\\
=\  & \frac{\sum_{u\in V_n}  {g'}_{vu}^{(n, m)}}
{\sum_{u\in V_n} {g'}_{vu}^{(n, m)}H_u^{(n)}} \cdot 
\frac{{g'}_{vw}^{(n, m)} H^{(n)}_w}
{\sum_{u\in V_n} {{g'}}_{vu}^{(n, m)}}, \ \ w \in V_n. 
\ea
\end{equation}

Finally, we obtain from \eqref{e26-1}
that $\ol q^{(n, \infty)} =
\lim_{m \to \infty} {\overline{g}}_{v}^{(n, m)} $
where 
$$
q^{(n, \infty)}_w = \frac{1}{\sum_{u \in V_{n}}
{y_{u}^{(n, \infty)} H_{u}^{(n)}}}
y_{w}^{(n, \infty)}  H_{w}^{(n)}.
$$
Clearly, $q^{(n, \infty)}_w$ is nonzero if and only if
$y_{w}^{(n, \infty)}$ is nonzero. 

Now assume that
$$
\ol q^{(n, \infty)} = \lim_{m \to \infty} {\overline{g}}_{v_{m}}^{(n, m)}.
$$
Taking again a subsequence of $\{m\}$ we can show that
the limit 
$$
\lim_{m \to \infty} \frac{{g^{'}}_{vw}^{(n, m)}} 
{\sum_{u \in V_{n}}{g'}_{vu}^{(n, m)}}
= y_{w}^{(n, \infty)}
$$
exists for every $w \in  V_{n}$. 
Then, repeating the above computations, we prove that 
the vector $\ol q^{(n, \infty)}$ 
is determined by \eqref{e24}.
\end{proof}

\begin{corol}\label{cor Th 4.11}
(1) Suppose that a generalized 
Bratteli diagram $B = (V,E)$ has the property: 
the set $\{H^{(n)}_w : \ 
w \in V_n\}$ is bounded for every $n \in \N$. 
Then every vector $\ol y^{(n, \infty)}$, defined by 
\eqref{eq-correction} satisfies 
\eqref{eq-conditions P(n)} and \eqref{eq limit vectors p(infty)} and therefore Theorem 
\ref{thm 4.11 edited} holds.

(2) Suppose a generalized Bratteli diagram is 
of bounded size. Then Theorem \ref{thm 4.11 edited} holds. 

(3) For the infinite Pascal-Bratteli diagram 
(see the definition in Section \ref{sect Bratteli-Pascal})
Theorem \ref{thm 4.11 edited} holds.
\end{corol}

The proof of Corollary \ref{cor Th 4.11} is 
straightforward. We note only that to show (2), we 
apply Lemma \ref{lemma_bdd_size_upper_cone}, we
see that $H_w^{(n)} \leq L_0 \cdots L_{n-1}$ where $B$
is determined by the sequence of parameters $(t_i, L_i)$.
We leave the details to the reader.

\begin{remark}
(1) Let $B =(V, E)$ be a generalized Bratteli diagram and 
suppose that $\ol B = (\ol W, \ol E)$ is a subdiagram such 
that $\ol W_n $ is a finite subset of $V_n$. 
Then Theorems \ref{TheoremIV_1} and \ref{Theorem_IV_2} can be used 
to determine tail invariant measures 
on $B$ which are extensions of probability measures from
$\ol B$. 

(2)
The above results are also valid for standard Bratteli 
diagrams, i.e., for diagrams with  finite set  $V_{n}$ 
for all $n$. In this case, the incidence matrices $F'_{n}$
and the stochastic matrices $F_{n}$ have finite sizes
and the maps $F_{n}^{T}$ and $G^{(n,m)}$  are 
continuous. Moreover, the sets $\Delta_{1}^{(n)}$
are closed. That is why we do not need the set 
$\Delta^{(n)}$ and can use $\Delta_{1}^{(n)}$ only.
We also have 
$\Delta^{(n,\infty,cl)} = \Delta^{(n,\infty)}$. 
The sets 
$L^{(n)}(\{ \overline{g}_{v}\})$ can be defined as above. 
Finally, we can show that the extreme points of the set
$\Delta^{(n,\infty)}$ are contained in 
$L^{(n)}(\{ \overline{g}_{v}\})$.
\end{remark}

%%%%
\section{Infinite Pascal-Bratteli diagram}
\label{sect Bratteli-Pascal}
%%%%%%%%%%%%%%%%%%%%%%%%%%%%%%%%%%%%%%%%%%%%%%%%%%%%%%
In this section, we define an infinite Pascal-Bratteli 
diagram (we will abbreviate the name to IPBD) which is 
a generalization of the classic Pascal-Bratteli diagram. 
They have been studied in several papers from various points 
of view, see e.g. \cite{Boca2008}, \cite{FrickOrmes2013}, 
\cite{FrickPetersen2010}, \cite{Fricketal2017},
\cite{MelaPetersen2005}, \cite{Mela2006},
\cite{Mundici2011}, \cite{Strasser2011}, \cite{Vershik2014}. 
We define and discuss Pascal-Bratteli diagrams for the
case of generalized diagrams assuming that the set 
of all vertices is infinite for every level. 
We present two versions of such diagrams: one is 
called $\mathbb{N}$-IPB diagram and the second is 
called $\mathbb{Z}$-IPB diagram. Here $\N$ stands for
natural numbers and $\Z$ stands for integers. 

We recall briefly the notion of the Pascal-Bratteli diagram
with finite levels. For fixed $n \in \N_0$, let $V_n$
be the set of all pairs $(x,y)$ of non-negative integers 
such that $x+y = n$. They are considered the vertices of 
the $n$-th level. The set of edges $E_n$ between the levels
$V_n$ and $V_{n+1}$ is defined by the following rule:
$|E((x_n, y_n), (x_{n+1}, y_{n+1})) | = 1$ if and only if 
either $x_{n+1} = x_n +1$ or $y_{n+1} = y_n +1$. In other
words, the set $r^{-1}(x_n, y_n) = \{ (x_n -1, y_n), (x_n,
y_n-1)\}$, where $(x_n, y_n) \in V_n$. 

In what follows we will extend this definition to the 
cases when a fixed number $n$ is decomposed into infinite sums of finitely many natural numbers and infinitely many zeroes and the sums will be indexed by natural numbers or integers.

\subsection{\texorpdfstring{$\mathbb{N}$}{\mathbb{N}}-infinite Pascal-Bratteli diagram}
\label{ssect N PB}

For $n\in \N$, define
$$
V_{n} = \{\overline{s}= (s_{1}, s_{2}, \ldots ) : 
\sum_{i = 1}^{\infty}s_{i} = n , \ s_i \in \N_0\}.
$$
It is convenient to represent $\ol s$ as an infinite vector 
indexed by $\N$ with nonzero entries $s_i$ where $i$
runs a finite set of indexes $I_{\ol s} = \{i_1, \dots, 
i_k\}$ where $k\leq n$.
In such a way, we numerate the levels $V_n$ by natural
numbers and every $V_n$ contains countably many vertices. 

We will use the standard basis $\{\ol e^{(i)}\}$ in 
$\R^\infty$ where $\ov e^{(i)} = (0,.., 0,1,0, ...)$ (the 
only nonzero entry is at the $i$-th place). Using this 
notation, we can represent every vertex $\ol s \in V_n$ as 
\be\label{eq_ol s}
\ol s = \sum_{j=1}^n\ol e^{(i_j)},
\ee
where $i_j \leq i_{j+1}$. Comparing with the formula
$\overline{s}= (s_{i_1}, \dots  , s_{i_k})$, we see that
$s_{i_j}$ shows how many times $\ol e^{(i_j)}$ is included
in \eqref{eq_ol s}. Note that the vertices of the first 
level $V_1$ can be identified with the vectors 
$\ol e^{(i)}$.

To define the set $E_n$ of edges between the vertices of 
levels $V_{n}$ and $V_{n+1}$, we use the following property:
for $\ol s \in V_n$, $\ol t \in V_{n+1}$, the set $E(\ol s, 
\ol t)$ consists of exactly one edge if
and only if $\ol t = \ol s + \ol e^{(i)}$ for some $i 
\in \N$. In other words, if $\overline{s} =(s_1, s_2, 
\cdots) \in V_{n}$ and 
$\overline{t} =( t_{1}, t_{2}, \ldots ) \in V_{n + 1}$, then 
there exists $i_0$ such that $t_{i_0} = s_{i_0} +1$ and 
$t_i = s_i$ if $i \neq i_0$. 
It is obvious that every $\ol s\in V_n$ is the source for
an edge and  every $\ol t \in 
V_{n+1}$ is the range for and edge from $E(V_n, V_{n+1})$.
It is also clear that $|s^{-1}(\ol s)| = \aleph_0$ since we can add $1$ in any of the coordinates of $\ov s$ and obtain a vertex $\ov t \in V_{n+1}$ which will be connected by an edge to $\ov s$.
In contrast, the set $r^{-1}(\ol t)$ is finite because 
$\ol s \in s(r^{-1}(\ol t))$ if and only if $\ol s = 
\ol t - 
\ol e^{(i_0)}$ where $i_0$ is an index such that $t_{i_0}$
is positive. Moreover, $|r^{-1}(\ol t)| \leq n$ for all 
$\ol t \in V_n$ and if $\ol t = (t_{i_1}, \cdots, t_{i_k})$,
then $s(r^{-1}(\ol t)) = \{\ol t - \ol e^{(i_1)}, \cdots, 
\ol t - \ol e^{(i_k)}\}$.

We set $V = \bigcup_n V_n$ and $E = \bigcup_n E_n$. 
The diagram $B = (V, E)$ 
is called the \textit{infinite Pascal-Bratteli diagram}
(or $\mathbb{N}$-IPB diagram).

We discuss the structure of the incidence matrices 
$F'_n = ({f'}^{(n)}_{\ol t \ol s})$ and the corresponding 
stochastic matrices for the 
$\mathbb{N}$-IPB diagram. The rows of $F'_n$ are determined
by vertices $\ol t = (t_{i_1}, \cdots, t_{i_k})$ of 
$V_{n+1}$:
\be\label{eq:F'_n for P-B} 
{f'}^{(n)}_{\ol t \ol s} = \begin{cases}
1, &  \ol s = \ol t - \ol e^{(i_j)},\ j = 1, \cdots, k\\ 
0, &   \mbox{otherwise} 
\end{cases}
\ee
Every row of $F'_n$ has finitely many nonzero entries,  
and every column of $F'_n$ has infinitely many 1's and infinitely many 0's. 

To describe the path space $X_B$ of the $\mathbb{N}$-IPB 
diagram, we note that every path $\ol x = (x_1, x_2, 
\ldots)$ 
(finite or infinite) can be identified with a sequence of
vertices $(\ol s^{(1)}, \ol s^{(2)}, \ldots)$ such that $\ol s^{(i)} = 
s(x_i) \in V_i,\; i \in \N$. Moreover, if $I_{\ol s}$ is the
set of nonzero entries of the vector $\ol s$, then, for 
every path $\ol x = (\ol s^{(1)}, \ol s^{(2)}, \ol s^{(3)},  \ldots)$, we have
$I_{\ol s^{(1)}} \subset I_{\ol s^{(2)}} \subset  I_{\ol s^{(3)}} 
\subset \ldots $ Moreover, $|I_{\ol s^{(i+1)}}|= 
|I_{\ol s^{(i)}}| +1 $ for $i \in \N$.

The stochastic matrix $F_n$ is determined by the entries
${f'}^{(n)}_{\ol t \ol s}$ and the heights 
$H^{(n)}_{\ol s}$ as above in \eqref{eq_stoch matrix F_n}. 
To work with the entries of stochastic matrices, we need 
a formula for $H^{(n)}_{\ol s}$ where $\ol s \in V_n$.
For this, we will simplify our notation and write 
$\ol s = (s_1, s_2, ...,s_k) \in V_n$ if $s_1+\ldots+ s_k =
n$. The difference is that this notation does not indicate  
what indexes correspond to positive entries in the infinite 
vector $\ol s$. Clearly, $k$ depends on $\ol s$ in this 
representation, however, we will write 
$\ol s = (s_1, \ldots, s_k)$ denoting a vertex from $V_n$. 

\begin{lemma}\label{lem_heights}
In the above settings and for $\ol s = (s_1, \cdots, s_k)$,
we have
\begin{equation}\label{e28}
H^{(n)}_{\overline{s}} = \frac{n!}{{s}_1!
s_{2}!\cdot \ldots s_k!},
\end{equation}

For $\ol t = (t_1, t_2, ..., t_k)$ and $\ol s = \ol t - \ol 
e^{(i)}$, the entry ${f}^{(n)}_{\ol t \ol s}$ 
of the stochastic matrix $F_n$ is determined 
by the relation:
\be\label{eq_entries of F_n}
{f}^{(n)}_{\ol t \ol s} = \frac{t_i}{n+1}
\ee
\end{lemma}

\begin{proof}
Relation \eqref{e28} can be easily proved 
by induction. Indeed, if $\ol t = (t_1, \dots, t_k)\in 
V_{n+1}$, then
$\ol t$ is connected by an edge with the vertices 
\be\label{eq: ol s vs ol t}
\ol s(i) = (t_1, \dots, t_i -1, \dots, t_k), \ \ 
i = 1,\dots, k.
\ee
Hence,
$$
\ba 
H_{\ol t}^{(n+1)} = & \ H_{\ol s(1)}^{(n)} + \cdots +
H_{\ol s(k)}^{(n)}\\ 
= & \ \frac{n!}{(t_1 -1)!t_2! \cdots  t_k! } + \cdots + 
\frac{n!}{t_1!t_2! \cdots  (t_k -1)! }\\
= & \ \frac{n!}{(t_1 -1)! \cdots (t_k -1)!}\cdot 
\frac{t_1 + \cdots + t_k}{t_1 \cdots t_k} \\
= & \ \frac{(n+1)!}{t_1 ! \cdots t_k !}.
\ea
$$

To compute the entries of the stochastic matrix $F_n$, we 
use \eqref{e28} and \eqref{eq: ol s vs ol t}:
\begin{equation}\label{e30}
f_{ \overline{t}\overline{s}}^{(n)} =
{f'}_{ \overline{t} \overline{s}}^{(n)} \cdot  \frac{H^{(n)}_{\overline{s}}}{H^{(n+1)}_{ \overline{t}}} = 
\frac{n!}{s_1! \cdots  s_k!}
\cdot  \frac{t_1! \cdots  t_k!} {(n + 1)!} 
= \frac{t_i}{n + 1}.
\end{equation}
\end{proof}

Next, we find the formulas for the entries 
$g_{\overline{t} \overline{s}}^{(n,m)}$ of matrices
$G^{(n,m)}$, the products of stochastic matrices 
$F_{n+ m-1}\cdots
F_n$ where $m\geq 1$ and $\overline{t} \in V_{n + m},\ 
\overline{s} \in V_{n}$.  We have seen in Section 
\ref{sect_inverse limits} that
these matrices play an important role in finding 
tail invariant measures on the path space of a Bratteli 
 diagram $X_B$.

\begin{remark}
 Let $\ol t = (t_1, t_2, ... )\in 
V_{n+m}$ and $\ol s = (s_1, s_2, ... )\in V_n$. Then
$E(\ol s, \ol t) \neq \emptyset$ (i.e., $\ol s$ and 
$\ol t$ 
are  connected by a finite path) if and only if $t_i \geq 
s_i$ for every $i=1, 2, ...$. In particular, $s_i > 0$ 
implies $t_i >0$.
\end{remark}

\begin{lemma}
Let $(F'_n)$ be the sequence of incidence matrices for 
the $\N$-IPB diagram defined by \eqref{eq:F'_n for P-B}. 
Then, for every $n \in \N_0$, the corresponding 
stochastic matrices $F_n$ do not satisfy 
Theorem \ref{t2}. This means that the mapping 
$F_n : \Delta_1 \to \De_1$ is discontinuous. 
\end{lemma}

\begin{proof}
To prove this, we fix some
$\overline{s}= (s_{1}, s_{2}, s_{3}, \ldots) \in V_{n}$ 
and let 
${\overline{t}}^{(i)} = \overline{s} + \ol e^{(i)}$, 
$i = 1, 2, \ldots$. Then 
${\overline{t}}^{(i)} \in V_{n + 1}$ 
and
$f_{{\overline{t}}^{(i)} \overline{s}}^{(n)} =
\frac{s_{i} + 1}{n + 1} \geq \frac{1}{n + 1}$. 

Let ${\overline{g}}_{\overline{t}}$ 
be the ${\overline{t}}$-th row of the matrix $F_{n}$. 
Then
$|{\overline{g}}_{\overline{t}^{(i)}}| \geq  \frac{1}{(n + 1) \  2^{a(\overline{s})}}$. This implies that
$\mathop{\overline{\lim}}\limits_{b(\overline{t}) \to \infty} | {\overline{g}}_{\overline{t}} | > 0$, 
where $b(\overline{t})$ is an enumeration of vertices 
of the set $V_{n + 1}$.
\end{proof}

%%%%%%%%%%%%%%%%%%%%%%%%%%%%%%%%%%%%%%%%%%%%%%%%%%%%%%%%%%

\begin{lemma}\label{lem entries of G(m,n)}
Let $\ol t = (t_1, ..., t_k )\in V_{n+m}$ and 
$\ol s = (s_1,..., s_l)\in V_n$ be two vertices connected 
by a finite path. Then 
\begin{equation}\label{e32}
g_{ \overline{t} \overline{s}}^{(n,m)} =
\left\lbrack\begin{pmatrix}
t_{1} \\
s_{1} \\
\end{pmatrix} 
\cdots 
\begin{pmatrix}
t_{k} \\
s_{k} \\
\end{pmatrix}
\right\rbrack \cdot 
\begin{pmatrix}
 n + m  \\
n \\
\end{pmatrix}^{-1}
\end{equation}
\end{lemma}

\begin{proof}
We first note that the entries ${g'}_{\overline{t} 
\overline{s}}^{(n,m)}$ of the matrices ${G'}^{(n,m)}
= F'_{n+m -1}\ \cdots \ F'_n$ can be found as follows: 
\be\label{eq:entries g(n,m)}
{g'}_{ \overline{t} \overline{s}}^{(n,m)} =  
\frac{m!}{\left( t_{1} - s_{1} \right)! \cdots 
\left( t_{k} - s_{k} \right)!}, \ \ 
\textrm{ if } s_{i} \leq t_{i} \textrm{ for all } 
i=1,..., k. 
\ee
We observe that, if the vertices $\ol t$ and $\ol s$ are 
connected by a path, then  $l \leq k$. To make
sense of all differences $t_i - s_i$ we can assume
(if necessary) that some 
$s_i$ are zeros. It does not affect the formula in 
\eqref{e32} because $\begin{pmatrix}
t_{i} \\
s_{i} \\
\end{pmatrix} = 1$ if $s_i =0$.
To complete the definition of entries
of ${G'}^{(n,m)}$, we set 
$$
{g'}_{ \overline{t} \overline{s}}^{(n,m)} = 0, \ \ 
\textrm{ if } s_{i} > t_{i} 
\textrm{ for at least one index } i \geq 1.
$$
The condition $s_{i} > t_{i}$ means that there is no path
between $\ol t$ and $\ol s$. 

We leave the proof of relation \eqref{eq:entries g(n,m)}
to the reader since it can be proved exactly in 
the same way as \eqref{e28}. 

Next, we use Lemma \ref{lem_heights} and compute 
$$
\ba
g_{ \overline{t} \overline{s}}^{(n,m)} =  & \ 
{g'}_{ \overline{t} \overline{s}}^{(n,m)} 
\frac{H^{(n)}_{\ol s}}{H^{(n +m)}_{\ol t}} \\
= &\ 
\frac{m!}{\left( t_{1} - s_{1} \right)! \cdots  
\left( t_{k} - s_{k} \right)! } 
\cdot  \frac{n!}{{s}_{1}! \cdots  s_{k}!} 
\cdot  \frac{{t}_{1}! \cdots  t_{k}!}{(n + m)!} \\
= &\ \left\lbrack\begin{pmatrix}
t_{1} \\
s_{1} \\
\end{pmatrix} 
\cdots 
\begin{pmatrix}
t_{k} \\
s_{k} \\
\end{pmatrix}
\right\rbrack \cdot 
\begin{pmatrix}
 n + m  \\
n \\
\end{pmatrix}^{-1}\\
\ea
$$
\end{proof}

\begin{remark}
There are only $l$ factors different from 1 in the product 
$\begin{pmatrix}
t_{1} \\
s_{1} \\
\end{pmatrix} 
\cdots 
\begin{pmatrix}
t_{k} \\
s_{k} \\
\end{pmatrix}$. So, with some abuse of notation, one 
can also write this product as 
$\begin{pmatrix}
t_{1} \\
s_{1} \\
\end{pmatrix} 
\cdots 
\begin{pmatrix}
t_{l} \\
s_{l} \\
\end{pmatrix}$. In this case, we consider only the 
entries of 
$\ol t$ and $\ol s$ that contribute to the product above.  
\end{remark}

%checked till here

Using the results of Section \ref{sect_inverse limits}, we 
describe the sets $\Delta^{(n,\infty,cl)},\ 
n\in \N, $ for the $\mathbb{N}$-IPB diagram. 
Fix $n, m \geq 1$. The rows of the matrix $G^{(n,m)}$ are 
represented by the vectors 
$$
\overline{g}_{\overline{t}}^{(n,m)} = 
\langle g_{\ \overline{t}\overline{s}}^{(n,m)} :\  
\overline{s} \in V_{n} \rangle,\ \ \ \ol t =(t_1, ..., t_k)
\in V_{n+m},
$$
where the entries $g_{\overline{t} \ol s}^{(n,m)}$
are found in Lemma \ref{lem entries of G(m,n)}. 
If necessary, we can write $\ol s = (s_1, \ldots, s_l) 
\in V_n$
as a collection of $k$ numbers $(s_i : i = 1,..., k)$, 
where $s_1 + \cdots 
+ s_k = n$, $s_i \leq t_i$,   assuming that some of them 
can be zero. 

\begin{remark} \label{rem representation of entries}
Let $\ol t = (t_1, ..., t_k )\in V_{n+m}$ and 
$\ol s = (s_1,..., s_l)\in V_n$ be two vertices connected 
by a finite path. 
One can easily see that relation (\ref{e32}) can be written 
then in the form
\begin{equation}\label{e33}
\ba
{g}_{ \overline{t} \overline{s}}^{(n,m)} = &\ 
\frac{m!}{\left( t_{1} - s_{1} \right)!\ \cdots \ \left( t_{k} - s_{k} \right)!} \cdot  \frac{n!}{{s}_{1}!\ \cdots 
\ s_{k}!} \cdot  \frac{{t}_{1}!\ \cdots \ 
t_{k}!}{(n + m)!}\\ 
= &\ \frac{n!}{{s}_{1}!\ \cdots \ s_{l}!}  \cdot 
\frac{\left\lbrack \left( t_{1} - s_{1} + 1 \right) \ \cdots\ t_{1} \right\rbrack\ \cdots \ \left\lbrack \left( t_{l} - s_{l} + 1 \right)\ \cdots\  t_{l}  \right\rbrack}
{(m + 1) \ \cdots \  (m + n)}.
\ea
\end{equation}
In other words, we should take into account only non-zero 
entries of $\ol s$ to compute ${g}_{ \overline{t} \overline{s}}^{(n,m)}$ in \eqref{e33}
\end{remark}

We recall that $L^{(n)} = 
L^{(n)}(\{ \ol g_{\ol t} \})$ 
denotes the set of all limit points of convergent sequences 
$\{\ol g^{(n,m)}_{\ol t_m}  :\ \ol t_{m} \in V_{n + m}\}$.

\begin{theorem}\label{thm L(n) via d's}
%05/10/2023
Let $\ol d = \langle d_i : i \in \N_0\rangle$ be a 
vector from $\Delta$, i.e. $\sum_{i=1}^{\infty} d_{i} \leq 
1$.  The set $L^{(n)}(\{ \ol g_{\ol t} \})$ 
consists of the vectors $\ol q^{(n)}(\ol d)
= \langle q^{(n)}_{\ol s}(\ol d) : \ \ol s \in 
V_n\rangle$ where 
$$
 q^{(n)}_{\ol s}(\ol d) =\frac{n!}{s_{1}!\ \cdots \ 
s_{k}!} \cdot   d_{1}^{s_{1}}\ \cdots \ d_{k}^{s_{k}}. 
$$
\end{theorem}

We will need the following lemma.

\begin{lemma}\label{lem_d_i}
Let $\{\ol t^{(m)}\}$ be a sequence of vertices such that
$$\overline{t}^{(m)} = (t_{1}^{(m)}, t_{2}^{(m)}, \ldots) 
\in V_{n + m}.
$$ 
Then the sequence 
$\overline{g}_{\overline{t}^{(m)}}$ converges if and only 
if the limit 
\be\label{eq_d_i limit}
d_i = \lim_{m \rightarrow \infty}\frac{t_{i}^{(m)}}{m}
\ee
exists for all $i = 1, 2, ...$
\end{lemma}

\begin{proof}
We first assume that the limit $d_i$ in 
\eqref{eq_d_i limit} exists,  $i = 1, 2, \ldots$. %%%%%%%%%%%%%%%%%%%%%%%%%%%%%%%%%%%%%%%%%%%%%%%%%%%%
%05/10/2023
Then, for every $k=1,2, ...$,
$$
\sum_{i = 1}^{k}d_{i} =  
\lim_{m \rightarrow \infty}\frac{1}{m} \cdot  
\sum_{i = 1}^{k}t_{i}^{(m)} \leq 1
$$
which implies that 
$$
\sum_{i = 1}^{\infty}d_{i} \leq 1
$$
%%%%%%%%%%%%%%%%%%%%%%%%%%%%%%%%%%%%%%%%%%%%%%% 
It follows immediately 
from \eqref{e33} that for every 
$\overline{s} = (s_1, ..., s_l) \in V_{n}$
 \be\label{eq_limit}
 \ba 
 \lim_{m \to \infty}
 g_{\overline{t}^{(m)}\overline{s}}^{(n,m)} = &\ 
 \frac{n!}{{s}_{1}! \ \cdots\  s_{l}!} \cdot 
 \lim_{m\to \infty} \frac{\left\lbrack \left( t_{1}^{(m)} - s_{1} + 1 \right) \ \cdots\ t_{1}^{(m)}\right\rbrack} {m^{s_{1}}} \ \cdots \ 
 \frac{\left\lbrack \left( t_{l}^{(m)} - s_{l} + 1 \right) \ \cdots\  t_{l}^{(m)} \right\rbrack}{m^{s_{l}}} \times \\
& \qquad \times   \frac{m^{s_{1}}\ \cdots\ m^{s_{l}}}
 {(m + 1)\ \cdots\   (m + n)}  \\ 
  = & \ \frac{n!}{{s}_{1}! \ \cdots\  s_{l}!} \cdot 
d_{1}^{s_{1}} \ \cdots \ d_{l}^{s_{l}}\\
 \ea
 \ee
 We used here that $s_1 + \cdots + s_l = n$ and 
 \eqref{eq_d_i limit}.

To prove the converse, we assume that the limit 
$$
\lim_{m_{j \to \infty}} g_{{\overline{t}}^{(m_{j})}}^{(n, m_{j})} \quad \mbox{exists}.
$$ 
This implies that 
$$
\lim_{m_j \to \infty}g_{{\overline{t}}^{( m_{j} )} 
\overline{s}}^{( n, m_{j})} 
\quad \mbox{exists\ for\ every} \ \overline{s} \in V_{n}.
$$
For infinitely many vertices $\ol s \in V_n$, this limit is zero. For other vertices, we use \eqref{eq_limit} to see
that the limit 
$$
\lim_{m_j\to \infty} \left\lbrack \frac{\left( t_{1}^{(m_j)} - s_{1} + 1 \right)}{m_j} \ \cdots\ \frac{t_{1}^{(m_j)}}{m_j}
\right\rbrack\ \cdots\ \left\lbrack \frac{\left( t_{l}^{(m_j)} - s_{l} + 1 \right)}{m_j} \ \cdots\ \frac{t_{l}^{(m_j)}}{m_j}
\right\rbrack
$$
exists. 
%%%%%%%%%%%%%%%%%%%%%%%%%%%%%%%%%%%%%%%%%%%%%%%%%%%%%
%06/10/2023
Indeed, for some fixed $i \geq 1$, take $\ol s \in V_n$ 
such that $s_i = n$ and $s_j = 0$ for $j \neq i$.
Then
$$
\lim_{m_j \to \infty}\frac
{\left( t_{i}^{( m_{j})} - n + 1 \right)\ \cdots \  
t_{i}^{(m_{j})} }{m_{j}^n}
$$
exists. We can find a subsequence $\{ m_{j}^{'} \}$ 
(depending on $i$) such that
$\lim_{m_j^{'} \to \infty}\frac{t_{i}^{( m_{j}^{'})}}
{m_{j}^{'}}$ exists for every $i =1, 2, \ldots$.
Now, applying the standard argument, we can select a 
subsequence $\{ m^{''}_{j} \}$ (being a subsequence of each 
sequence $\{ m_{j}^{'}\}$ depending on $i$ for sufficiently large $j$) such that
\be\label{eq:def of d_n,i}
\lim_{m_j^{''} \to \infty}\frac{t_{i}^{( m_{j}^{''})}}{m_{j}^{''}} 
=  d_{n,i}, \  
i = 1, 2, \ldots 
\ee
According to the above considerations, we have 
$\sum_{i = 1}^{\infty}{d_{n,i}} \leq 1$.

%%%%%%%%%%%%%%%%%%%%%%%%%%%%%%%%%%%%%%%%%%%%%
\end{proof}

\begin{proof} (\textit{Proof of Theorem \ref{thm L(n) via d's}}).    
We use Lemma \ref{lem_d_i} and relations \eqref{eq_limit},
\eqref{eq:def of d_n,i} to find the limit 
points of the set $L^{(n)}(\{\ol g_t\})$:
$$
\lim_{m_j \to \infty}g_{{\overline{t}}^{( m_{j}^{'})},\ol s} ^{( n,m_{j}^{'})} = 
\frac{n!}{s_{1}!\ \cdots\  s_{l}!}\cdot 
d_{n,1}^{s_{1}}\ \cdots \  d_{n,l}^{s_{l}}. 
$$
This proves the theorem.
\end{proof}

\begin{prop} \label{prop inv meas PB}
(1) Let $B$ be the $\N$-infinite Pascal-Bratteli diagram,
and $\ol d \in \Delta_1$ a probability vector. Define 
the sequence of vectors $\ol p^{(n)} = \langle 
p^{(n)}_{\ol s} : \ \ol s \in V_n\rangle$ ($n \in \N$)  
by setting
\be\label{eq_vectors p(n)}
p^{(n)}_{\ol s} = d_1^{s_1} \ \cdots\ d_l^{s_l}
\ee
where $\ol s = (s_1, \ldots, s_l)$ is any vertex from $V_n$.
Then $\ol d$ determines uniquely a probability 
tail invariant measure $\mu_{\ol d}$ on the path space 
$X_B$ of the diagram $B$.

(2) Suppose that 
$\ol q^{(n)}(\ol d) = \langle q^{(n)}_{\ol s}(\ol d^{(n)}) : 
\ \ol s \in V_n\rangle$ is a vector from $L^{(n)}(\{
\ol g_{\ol t}\})$ where 
$$
q^{(n)}_{\ol s}(\ol d^{(n)}) =\frac{n!}{s_{1}!\ \cdots \ 
s_{l}!} \cdot   d_{n,1}^{s_{1}}\ \cdots \ d_{n,l}^{s_{l}},
$$
and 
$ {\overline{d}}^{(n)}  = \langle d_{n,i} :\ i=1, 2, \ldots 
\rangle$ is a probability vector.
Then the sequence of 
${\overline{q}}^{(n)}({\overline{d}}^{(n)})$
determine an invariant measure on $\mathbb{N}$-IPB 
diagram if and only if 
${\overline{d}}^{(n)} = \overline{d}$  for every 
 $n \in \N$.

\end{prop} 

\begin{proof}
To prove (1), we use Theorem \ref{BKMS_measures=invlimits}
and show that ${F'}_n^T \ol p^{(n+1)} = \ol p^{(n)}$ for
all $n\geq 1$. Indeed, if $\ol s = \langle s_j : j \in \N
\rangle$ is 
a vertex in $V_n$ ($\ol s$ has $l$ nonzero 
entries) connected with a vertex  $\ol t \in 
V_{n+1}$, then $\ol t = \ol s + \ol e^i$ for some $i \in 
\N$. Hence, 
$$
\ba 
\sum_{\overline{t} \in V_{n + 1}}{f'}_{\overline{t}\overline{s}}^{(n)}
p_{\overline{t}}^{(n + 1)} = & 
\sum_{i = 1}^{\infty} (d_{1}^{s_{1}} \ 
\cdots\  d_{i}^{s_{i} + 1}\cdots) \\
= & {\lbrack d}_{1}^{s_{1}}\cdot \ d_{2}^{s_{2\ }}\cdot \ldots\rbrack\cdot \sum_{i = 1}^{\infty}d_{i} \\
= &  d_{1}^{s_{1}}\ \cdots \ d_{l}^{s_{l}}\\
=&  p_{\overline{s}}^{(n)}.
\ea 
$$

Using \eqref{eq_vectors p(n)}, we assign the value 
$p_{\ol s}^{(n)}$ to any cylinder set connecting 
a vertex in $V_1$ with $\ol s$. It defines $\mu_{\ol d}$
on all clopen sets. The proved relation means that 
this definition satisfies Theorem \ref{BKMS_measures=invlimits}. Hence, we can extend $\mu_{\ol d}$ to all Borel sets and obtain a tail invariant probability measure.

(2)  We use the computation similar to that from (1) to see
that ${F'}^T_{\ol t \ol s} \ol p^{(n+1)} = \ol p^{(n)}$ 
where $\ol p^{(n)}$ is determined by $\ol d^{(n)}$ as in 
(1). To see that $\ol d^{(n)} = \ol d^{(n+1)}$ it suffices
to take $\overline{s} = (0,  \ldots, n, 0, \ldots )$ where 
$n$ is at $k$-th position.

\end{proof}

\begin{theorem}\label{t5}
The  measures $\mu_{\overline{d}}$, defined in Proposition
\ref{prop inv meas PB} by probability vectors 
$\overline{d} = \langle d_{i} :\ i = 1, 2, \ldots 
\rangle$, form the set of all probability invariant ergodic 
measures on $\mathbb N$-IPB diagram.
\end{theorem}

\begin{proof} Let $\mu$ be an ergodic tail invariant 
probability measure on the path space $X_B$ of the 
$\N$-IPB diagram.  
By Lemma \ref{lem_Delta inf meas}, there exists a 
sequence of vectors $\{ \overline{q}^{(n)}\}$ such that
${\overline{q}}^{(n)} \in \Delta^{(n,\infty,cl)}$
and this sequence determines the measure $\mu$. We show 
that 
\be \label{eq mu = m_d e36}
\mu = \mu_{\ol d}
\ee
for some probability vector $\ol d$. 
The vectors $\overline{q}^{(n)}$'s are extreme points of
$\Delta^{(n,\infty,cl)}$.
As observed in Remark \ref{rem_r7r8}, the extreme points 
of $\Delta^{(n, cl, \infty)}$ are contained in 
$L^{(n)}(\{\ol g_{\ol t}\})$ where $\ol t \in V_{n+m}, 
m\geq 1$. Hence, the vectors  
$\ol q^{(n)}$ are of the form 
${\overline{q}}^{(n)}({\overline{d}}^{(n)})$, where
$\overline{d}^{(n)}$ are probability vectors.
It follows from Proposition \ref{prop inv meas PB} (2)
that $\overline{d}^{(n)} = \ol d$ for all $n$. This proves
\eqref{eq mu = m_d e36}.

To prove that each measure $\mu_{\overline{d}}$ is ergodic, 
we will show that the measures $\mu_{\overline{d}}$, when 
$\overline{d}$ runs over all probability vectors, 
are mutually singular. 

Fix a probability vector $\overline{d} = \langle d_i : 
i \in \N\rangle$ and define the set 
$X(\overline{d})$ consisting of all infinite paths
$\overline{e} = ( {\overline{s}}^{(1)},
{\overline{s}}^{(2)}, {\overline{s}}^{(3)}, \ldots) 
\in X_{B}$ such that 
\be\label{eq:limit}
\lim_{n \rightarrow \infty}\frac{s_{i}^{(n)}}{n} = d_{i},
\ i = 1, 2, \ldots,
\ee
where 
${\overline{s}}^{(n)} =  ( s_{1}^{(n)}, s_{2}^{(n)}, 
\ldots) \in V_{n}$, $n \in \N$.
We will show that $\mu_{\overline{d}}( X( \overline{d} ) ) = 1$.

Recall that 
$ \ol e^{(j)} = (0, \ldots, 0, 1, 0, \ldots\ )$, $j \in \N$,
is the vector such that the symbol ``1'' appears at the 
$j$-th position. Clearly, $\ol e^{(j)} \in \Delta_{1}$. 

Let $\overline{y} \ : \ X_{B} \rightarrow \Delta$
be a random vector-valued function such that 
$\overline{y}$ takes the value  $\ol e^{(j)}$
with probability $d_{j}$, $j = 1, 2, \ldots$.

Next, we define a sequence of independent random vectors 
${\overline{Y}}_{1}, {\overline{Y}}_{2}, {\overline{Y}}_{3}, 
\ldots$ as follows: for an infinite path
$\overline{e} =({\overline{s}}^{(1)},
{\overline{s}}^{(2)}, {\overline{s}}^{(3)}, \ldots )$, we  
set $\overline{Y}_{n}(\overline{e})= \overline{e}^{(j)}$, 
where $j = j_{n}$ is a natural number such that 
$\overline{s}^{(n)} = {\overline{s}}^{(n - 1)} + 
\ol e^{(j)}$ ($\ol s^{(0)} = \ol 0$). 
Clearly, $\overline{Y}_{n}$ has the same distribution as 
$\overline{y}$. In particular, 
$\mathbb E( {\overline{Y}}_{n} ) = \mathbb E(\overline{y}) = 
\overline{d} = \langle  d_{i}, i = 1, 2, \ldots \rangle$.

It follows from the definition of $\ol Y_n$'s that 
$({\overline{Y}}_{1} + {\overline{Y}}_{2} + 
+ \ldots + {\overline{Y}}_{n} ) (\overline{e})= 
{\overline{s}}^{(n)}$.
This implies that 
$$\frac{1}{n}  \cdot  ({\overline{Y}}_{1} + 
{\overline{Y}}_{2} + \ldots + \overline{Y}_{n})(\overline{e}) \in \Delta_{1}.
$$
Applying the limit theorem for the sequence $\{\ol Y_n\}$, we get 
$$\frac{1}{n}  \cdot  ({\overline{Y}}_{1} + {\overline{Y}}_{2} +  \ldots + {\overline{Y}}_{n} )(\overline{e}) \rightarrow  \mathbb E(\overline{y}) = 
\overline{d}  =\langle d_{i} \: \ i \in \N \rangle 
$$ 
for $\mu_{\overline{d}}$-almost all $\overline{e} 
\in X_{B}$. 
It follows from \eqref{eq:limit} that 
$\mu_{\overline{d}} ( X( \overline{d} ) ) = 1$.

Since $X(\overline{d}) \cap X( \overline{c} ) = \emptyset$ 
for different probability vectors $\overline{d}$ and 
$\overline{c}$, we obtain that the measures 
$\mu_{\overline{d}}$ and $\mu_{c\overline{}}$ are mutually 
singular. 
We can finish the proof of Theorem \ref{t5} as follows. 
It was proved that each ergodic invariant measure is of the
form $\mu_{\overline{d}}$. Then every invariant probability 
measure
$\mu$ on the $\N$-IPB diagram is an integral over the set 
of ergodic measures. Thus, $\mu$ cannot be singular to each 
ergodic measure.  Therefore every measure 
$\mu_{\overline{d}}$ is ergodic.
\end{proof}

%%%%%%%%%%%%%%%%%%%%%%%%%%%%%%%%%%%%%%%%%%%%%%%%%%%%%%%%%%%%%%%%%%%%%%%%%%%%%%%%%%%%%%%%%%%%%%%%%%%%

\subsection{\texorpdfstring{$\mathbb{Z}$}{\mathbb{Z}}-Infinite Pascal-Bratteli diagram}

In this subsection, we consider another realization 
of the infinite Pascal-Bratteli diagram with vertices whose 
entries are indexed by $\Z$. For the reader's convenience,
we remind the notation from Subsection \ref{ssect N PB} 
adapted to this case.

For $n \in \N$ define 
$$
V_{n} = \{\overline{s}=(\ldots s_{-1}, s_{0}, s_{1}, 
\ldots)  \ :\ s_{i} \in \N_0, 
\sum_{i = - \infty}^{\infty} s_{i} =n \}.
$$
Every $V_n$ is a countable set. To define the set of all 
edges $ E = \bigcup_{n\geq 1} E_n$, we say that for 
 $\overline{s} = (\ldots, s_{- 1}, s_{0}, s_{1}, \ldots) 
\in V_{n}$ and
$\overline{t} = (\ldots, t_{- 1}, t_{0}, t_{1}, \ldots) \in 
V_{n + 1}$ the set $E(\ol s, \ol t)$ consists of exactly
one edge if and only if $\ol t = \ol s + \ol e^{(i)}$ for 
some $i \in \Z$. Here $e^{(i)} = (\ldots,  0, 1, 0, 
\ldots) $ is the vector whose only non-zero entry is 1 
at the $i$th place. That is if $\overline{s}$ 
and $\overline{t}$ as above, then 
there exists $i$ such that $t_{i} = s_{i} +1$ and 
$t_i = s_i$ if $i \neq i_0$. In particular, $s_{i} = 0$ if
$t_{i} = 0$. Then every $\ol s\in V_n$ is the source 
for an edge and  every $\ol t \in 
V_{n+1}$ is the range for and edge from $E(V_n, V_{n+1})$.
Hence $|s^{-1}(\ol s)| = \aleph_0$ and the set 
$r^{-1}(\ol t)$ is finite. 

The diagram $B = (V, E)$, where $V$ and $E$ are defined 
above, is called the $\mathbb{Z}$-\textit{infinite 
Pascal-Bratteli diagram} (or $\mathbb{Z}$-IPB diagram).

We use the same notation $H^{(n)}_{\overline{s}}$, 
${F'}_{n} =({f'}_{ \overline{t} \overline{s}}^{(n)})$, and 
$F_{n} = (f_{ \overline{t} \overline{s}}^{(n)}) \ 
(\overline{t} \in V_{n + 1}$, $\overline{s} \in V_{n}, 
n \in \N)$ for the height of the towers, incidence, and 
stochastic matrices of the $\mathbb{Z}$-IPB diagram, 
respectively.

Let  $i \in  \mathbb{Z}$ is  such  that $\overline{t} = \overline{s} + \ol e^{(i)}$, $i\in \Z$. 
Then we have the following formulas
\begin{equation}\label{e1'}
H^{(n)}_{\overline{s}} = \frac{n!}{{(\ldots s_{ -1}! \cdot s}_{0}! \cdot s_{1}! \cdot \ldots)},
\end{equation}

\begin{equation}\label{e2'}
{f'}_{ \overline{t} \overline{s}}^{(n)} = 1 \textrm{ if } \overline{t} = \overline{s} + \ol e^{(i)}, 
\textrm{ and }
{f'}_{ \overline{t} \overline{s}}^{(n)} = 0 \textrm{ otherwise}.
\end{equation}

\begin{equation}\label{e3'}
f_{\ \overline{t} \overline{s}}^{(n)} = {f'}_{ \overline{t} \overline{s}}^{(n)} \cdot \frac{H^{(n)}_{\overline{s}}}{H^{(n+1)}_{\overline{t}}}
= \frac{n!}{(\ldots s_{-1}! \cdot s_{0}! \cdot s_{1}! \cdot \ldots)} \cdot \frac{(\ldots \cdot t_{-1}! \cdot t_{0}! \cdot t_{1}! \cdot \ldots)}{(n + 1)!} = \frac{t_{i}}{(n + 1)}, 
\end{equation}

Using similar arguments as for the $\mathbb{N}$-IPB 
diagram, we can find the formulas for the entries 
${g'}_{ \overline{t} \overline{s}}^{(n, m)}$ and
$g_{\overline{t} \overline{s}}^{(n, m)}$ of the matrices
${G'}^{(n, m)}$ and $G^{(n, m)}$,
$\overline{t} \in V_{n + m}$, $\overline{s} \in V_{n}$. 
We obviously have 
${g'}_{ \overline{t} \overline{s}}^{(n, m)} = 0$  if 
$s_{i} > t_{i}$  for at least one index $i \in  \Z$.
If  $s_{i} \leq t_{i}$ for all $i \in \Z$, then 
\begin{equation}\label{e4'}
{g'}_{\ \overline{t} \overline{s}}^{(n, m)} = \frac{m!}{\ldots \cdot ( t_{-1} - s_{-1})! \cdot \left( t_{0} - s_{0} \right)! \cdot \left( t_{1} - s_{1} \right)! \cdot \ldots} 
\end{equation}
Further, we calculate as in \eqref{e33}
\begin{equation}\label{e5'}
\ba 
{g}_{\overline{t}\overline{s}}^{(n, m)}= &\ 
\frac{m!}{\ldots \cdot ( t_{-1} -  s_{-1})! \cdot \left( t_{0} - s_{0} \right)! \cdot \left( t_{1} - s_{1} \right)! \cdot \ldots\ } \cdot 
\frac{n!}{{ \ldots \cdot s_{-1}!\cdot s}_{0\ }! \cdot s_{1}! \cdot \ldots} \times \\
&\ \qquad 
\times  \frac{{\ldots \cdot t_{-1}t}_{0}! \cdot t_{1}! \cdot \ldots}{(n + m)!} \\ 
= & \
\left\lbrack \ldots \cdot \begin{pmatrix}
t_{-1} \\
s_{-1} \\
\end{pmatrix}
\cdot 
\begin{pmatrix}
t_{0} \\
s_{0} \\
\end{pmatrix} 
\cdot 
\begin{pmatrix}
t_{1} \\
s_{1} \\
\end{pmatrix}
\cdot  \ldots{} \right\rbrack 
\begin{pmatrix}
 n + m  \\
n \\
\end{pmatrix}
\ea 
\end{equation}
whenever $s_{i} \leq t_{i}$ for all $i$ 
and we set 
$g_{\overline{t} \overline{s}}^{(n, m)} = 0$ if
$s_{i} > t_{i}$ for at least one $i \in \mathbb{Z}$ .

Next, we describe the sets 
$\Delta^{(n,\infty,cl)}, \ n = 1, 2, \ldots$ 
for $\mathbb{Z}$-IPB diagram. Fix $n, m \geq 1$ using 
\eqref{e5'}. The rows of the matrix
$G^{(n, m)}$ are the vectors of the form
$$
{\overline{g}}_{\overline{t}}^{(n,m)} = 
\langle g_{ \overline{t} \overline{s}}^{(n, m)} \: \  \overline{s} \in V_{n} \rangle. 
$$
We have from (\ref{e5'})
\begin{equation}\label{e6'}
\ba
{g}_{ \overline{t} \overline{s}}^{(n, m)}= &\ 
\frac{m!}{ \ldots ( t_{-1} - s_{-1})! \cdot \left( t_{0} - s_{0} \right)! \cdot \left( t_{1} - s_{1} \right)! \cdot \ldots} \cdot  
\frac{n!}{{\ldots s_{-1} \cdot s_{0}! \cdot s_{1}! \cdot \ldots}} \times  \\
 & \  \qquad 
\times \frac{{\ldots  t_{-1} ! \cdot t}_{0}! \cdot t_{1}! \cdot \ldots}{(n + m)!}\\
= & \ \frac{n!}{{ \ldots s_{-1}! \cdot s}_{0}! \cdot s_{1}! \cdot \ldots} \cdot 
\frac{ \left\lbrack  \ldots ( t_{-1} - s_{-1} + 1 ) \ldots t_{-1} \right\rbrack \cdot \left\lbrack \left( t_{0} - s_{0} + 1 \right) \ldots t_{0} \right\rbrack \cdot \ldots}{(m + 1)  (m + 2)  \ldots  (m + n)}.
\ea
\end{equation}

Let
$m \rightarrow \infty$ and take ${\overline{t}}^{(m)} = (\ldots(t_{-1}^{(m)}, t_{0}^{(m)}, t_{1}^{(m)}, \ldots) 
\in  V_{n + m}$. 
Using relation (\ref{e6'}) and the same arguments as in 
the case of the $\mathbb{N}$-IPB diagram, we prove that
$$\lim_{m \rightarrow \infty}{\overline{g}}_{{\overline{t}}^{(m)}} \ \ 
\mbox{exists} \ \Longleftrightarrow \ 
\lim_{m \rightarrow \infty}\frac{t_{i}^{(m)}}{m} \ \ 
\mbox{exist\ for\ all}\ i \in \Z.
$$

Then we can repeat the proof of Theorem 
\ref{thm 4.11 edited} and show that 
\begin{equation}\label{e7'}
L_{n} (\{ {\overline{g}}_{\overline{t}} :\  \overline{t} 
\in \ V_{n + m},  m \geq 1 \}) =
\{ {\overline{q}}^{(n)}(\overline{d}) \}
\end{equation}
where 
$$
\overline{q}^{(n)}(\overline{d}) = \langle q_{\overline{s}}^{(n)} ( \overline{d} ) : \ \overline{s} \in V_{n} \rangle =
\left\langle \frac{n!}{\ldots s_{-1}! \cdot s_{0}! \cdot s_{1}! \cdot \ldots } 
\cdot [\ldots \cdot d_{-1}^{s_{-1}} \cdot d_{0}^{s_{0}} \cdot d_{1}^{s_{1}} \cdot \ldots] \ : \  
\overline{s} \in V_{n} \right\rangle,
$$
and
$\overline{d} = \langle d_{i} : \  i \in \Z \rangle$
 is a non-negative vector such that 
 $\sum_{i = -\infty}^{\infty} d_{i}  \leq 1$.

\begin{lemma}\label{r1'}
Let $\overline{d} = \langle d_{i} :\ i \in \mathbb{Z} \rangle$ be a probability vector.
Then $\overline{d}$ determines a tail invariant probability 
measure $\mu_{\overline{d}}$ on the 
$\mathbb{Z}$-IPB diagram such that
$$
p_{\overline{s}}^{(n)} = \mu_{\overline{d}}
([\ol e_{\ol s}]) 
= \frac{q_{\overline{s}}^{(n)}}{H_{\overline{s}}^{(n)}} 
= \ldots \cdot d_{-1}^{s_{-1}}d_{0}^{s_{0}} d_{1}^{s_{1}} \cdot \ldots,\quad
\overline{s} \in V_{n},\ n \in \N,
$$
where $\ol e_{\ol s}$ is a finite path with the range 
$\ol s \in V_n$.
\end{lemma}

\begin{theorem}\label{t1'}
The  measures $\mu_{\overline{d}}$, defined in Lemma 
\ref{r1'} by probability vectors 
$\overline{d} = \langle d_{i} :\ i = 1, 2, \ldots 
\rangle$, form the set of all probability invariant ergodic 
measures on $\mathbb Z$-IPB diagram.

\end{theorem}
%%%%%%%%%%%%%%%%%%%%%%%%%%%%%%%%%%%%%%%%%%%%%%%%%%%%%%%%%%%%%%%%%%%%%%%%%%%%%%%%%%%%%%%%%%%%%%%%%%%%%%%%%%%%

\subsection{Multi-dimensional finite Pascal-Bratteli diagrams}

We discuss in this section multi-dimensional finite 
Pascal-Bratteli diagrams. 
 These diagrams and their tail invariant probability measures  were considered  
by K. Petersen and  the other authors  in \cite{MelaPetersen2005}  and in 
\cite{FrickPetersen2010}.
Our goal is to show that these measures can be obtained 
using the ``geometrical'' method, which was presented
in Section \ref{sect_inverse limits}.

To define an MFPB diagram $B = (V, E)$, 
we fix first a natural number 
$k \geq 2$. For $n= 0, 1, 2, \ldots $, we set
$$V_{n} = \{\overline{s} = (s_{1}, s_{2}, \ldots, s_{k} ) 
: \  s_{i} \in \N_0,\ \  \sum_{i = 1}^{k}s_{i} =n \}.
$$
Of course, $V_{0} = \{ {\overline{s}}_{0} = 
(0, \ldots, 0 )\}$ and $|V_{n}| < \infty$ for all $n \in 
\N.$

We define the set of edges $E$ as follows.
Let $\overline{s} = (s_1, \ldots, s_k) \in V_{n}$ and  
$\overline{t} = (t_{1},  \ldots, t_{k}) \in V_{n + 1}$. 
Then the edge 
$e(\overline{s}, \overline{t})$
exists if and only if $\ol s + \ol e^{(i_0)} =
\ol t$, i.e.,  
$s_{i_{0}} + 1 = t_{i_{0}}$ for some $i_{0}$ and
$s_{i} = t_{i}$ for $i \neq i_{0}$. 
Then, for each 
$\overline{s}  \in V = \bigcup_{n = 0}^{\infty}{(V}_{n})$, 
we have $|s^{(-1)} ( \overline{s})| = k $ 
and $|r^{( - 1)} ( \overline{s})| = k$. 
The defined diagram is called $k$-\textit{dimensional 
Pascal-Bratteli diagram}. 

Applying the methods used in Subsection 
\ref{ssect N PB}, we can find the formulas for the heights
$H^{(n)}_{\ol s}$ of towers and entries of matrices ${G'}^{(n,m)}$ and $G^{(n,m)}$.  

The number of finite paths with range $\ol s =
(s_1, \ldots, s_k) \in V_m$ is
$$
H_{\overline{s}}^{(n)} = \frac{n!}{s_{1}!\cdot \ldots 
\cdot s_{k}!}.
$$ 

For $\ol s = (s_1, \ldots, s_k) \in V_m$ and 
$\overline{t} = (t_1, \ldots, t_k) \in  V_{n + 1}$, 
the entries ${f^{'}}_{\overline{t}\overline{s}}^{(n)}$ of 
the incidence matrix $F'_{n}$ are determined by the 
rule:
$$
{f^{'}}_{\overline{t}\overline{s}}^{(n)} =
\begin{cases}
    1, \ & \mbox{if} \ e(\overline{s}, \overline{t}) \neq
    \emptyset\\
    0, \ & \mbox{if} \ e(\overline{s}, \overline{t}) = 
    \emptyset
\end{cases}
$$

Then the entries $\{ f_{\overline{t} \overline{s}}^{(n)}\}$ 
of the stochastic matrix $F_{n}$ are given by the following formulas:
$$
f_{\overline{t} \overline{s}}^{(n)} = 
{f^{'}}_{\overline{t} \overline{s}}^{(n)} \cdot  \frac{H_{\overline{s}}^{(n)}}{H_{\overline{t}}^{n + 1}}
= \frac{t_{i}}{(n + 1)},\ n \in \N, 
$$ 
if $t_{i} = \ s_{i} + 1$ for some $i$,
and $f_{\overline{t}\overline{s}}^{(n)} = 0$ otherwise.

The entries ${g'}_{ \overline{t} \overline{s}}^{(n,m)}$ and
$g_{\overline{t} \overline{s}}^{(n,m)}$ of the matrices
${G'}^{(n,m)}$ and $G^{(n,m)}$, can be found similar to 
(\ref{e32}) and \eqref{eq:entries g(n,m)}. 
For $\ol s \in V_n$ and $\ol t \in V_{n+m}$, we obtain 
\begin{equation}\label{e31'}
    {g'}_{ \overline{t}\overline{s}}^{(n,m)} = \frac{m!}{\left( t_{1} - s_{1} \right)! \cdot  \left( t_{2} - s_{2} \right)!\cdot  \ldots \cdot  \left( t_{k} - s_{k} \right)!}
\end{equation}
if $s_{i} \leq t_{i}$ for $i = 1, 2, \ldots, k$, and
${g'}_{ \overline{t}\overline{s}}^{(n,m)} = 0$ 
if $s_{i} > t_{i}$ for at least one index $i \geq 1$. 
Similarly, we can write 
\begin{equation}\label{e32'}
\ba 
{g}_{\ \overline{t}\overline{s}}^{(n,m)} = &\ 
\frac{m!}{\left( t_{1} - s_{1} \right)!  \cdot  \ldots 
\cdot  \left( t_{k} - s_{k} \right)!}
\cdot  \frac{n!}{{s}_{1\ }! \cdot  \ldots \cdot {s}_{k}!}  
\cdot  \frac{{t}_{1}!  \cdot \ldots \cdot  t_{k}!}
{(n + m)!} \\
= &\  \left\lbrack\begin{pmatrix}
t_{1} \\
s_{1} \\
\end{pmatrix} 
\cdot  \ldots \cdot 
\begin{pmatrix}
t_{k} \\
s_{k} \\
\end{pmatrix} \right\rbrack \cdot 
\begin{pmatrix}
 n + m \\
n \\
\end{pmatrix}^{-1}
\ea 
\end{equation}
whenever $s_{i} \leq t_{i}$ for $i = 1, \ldots, k$, 
and $g_{\overline{t} \overline{s}}^{(n,m)} = 0$ if
$s_{i} > t_{i}$ for at least one $i$.

\begin{prop}
Let $\overline{d} = \langle d_{i} :\ i = 1,..., k \rangle$ be a probability vector.
Then $\overline{d}$ determines a tail invariant probability 
measure $\mu_{\overline{d}}$ on the path space of the 
MFPB diagram such that
$$
p_{\overline{s}}^{(n)} = \mu_{\overline{d}}
([\ol e_{\ol s}]) 
= \frac{q_{\overline{s}}^{(n)}}{H_{\overline{s}}^{(n)}} 
=   d_{1}^{s_{1}}d_{2}^{s_{2}} \ \ldots \ d_{k}^{s_{k}} \quad \overline{s} \in V_{n},\ n \in \N,
$$
where $\ol e_{\ol s}$ is a finite path with the range 
$\ol s \in V_n$.

The measures $\mu_{\overline{d}}$ form the set of all 
probability invariant ergodic measures on the MFPB diagram.
\end{prop}

\subsection{Subdiagrams of the Pascal-Bratteli diagrams}

Let $B = (V, E)$ be a $\mathbb{N}$-infinite Pascal-Bratteli diagram. Take a proper non-empty subset $K \subset  \N$, $K$ can be finite or infinite. Define the sets $W_n$  by
the following rule:  
$$
W_{n} := \{\overline{s} = (s_{i}) : \sum_{i = 1}^{\infty} 
s_{i} = n\ \mbox{such\ that}\ s_i = 0 \ \mbox{whenever} \ 
i \in \N\setminus K\}.
$$ 
Here $n = 1, 2, \ldots$ and the entries of $\ol s$ are non-negative integers. 

Of course, the terms of the sequence $\ol W = (W_{n})$ can be 
viewed as subsets of 
$V_{n}$, the levels of $\N$-infinite  
Pascal-Bratteli diagram. Thus, the set $K$ defines a 
subdiagram ${\overline{B}}_{K} = (\ol W, \overline{E})$,
where $\overline{E}$ is the set of all possible edges between $W_{n}$ and $W_{(n + 1)}$, $n \in \mathbb{N}$. 
It is clear that ${\overline{B}}_{K}$ is, in its turn,  
a Pascal-Bratteli diagram (finite if $|K| < \infty$ or infinite if $|K| = \infty$). 
\vskip 0.3cm
\textbf{Claim.} $\widehat{X}_{{\overline{B}}_{K}} = X_{{\overline{B}}_{K}}$. 
\vskip 0.2cm

This means that, by the definition of $\ol B_K$, the path space 
of the diagram is invariant for the tail equivalence 
relation. 
Moreover, every invariant probability measure $\overline{\mu}$ 
on
$\overline{B}_{K} = (\ol W, \ol E)$ is automatically its extension to the whole diagram $B = (V, E)$.

The same facts are true for $\Z$-infinite Pascal-Bratteli
diagram where $K$
is a proper subset of $\mathbb{Z}$. 

%%%%%%%%%%%%%%%%%%%% 
%%%%%%%%%%%%%%%%%%%%

\section{A class of Bratteli diagrams of bounded size}
\label{sect bounded size}

In this section, we consider generalized Bratteli diagrams 
with additional properties. In particular, we define
the diagrams of (uniformly) bounded size. 
This class of Bratteli diagrams was defined and considered in 
detail in 
\cite{BezuglyiJorgensenKarpelSanadhya2023}, see Definition 
\ref{Def:BD_bdd_size}. We first discuss the case of 
standard Bratteli diagrams.

\subsection{Bratteli diagram of uniformly bounded 
size}\label{ss finite uniformly bndd}

To define a standard Bratteli diagram $B_k = (V, E)$, we fix a 
number $k \in \N$ and take 
$V_{n} = \{- nk, ... , -1, 0, 1, ... , n k \}$, 
$n \in \N$. 
The set $E_{n}$ of edges between $V_{n}$ and $V_{n + 1}$
is defined as follows: for $w\in V_n, v \in V_{n+1}$,
the edge $e(w, v)$ exists if and only if 
$v \in \{ w-k, \ldots, w-1, w, w+1, \ldots, w+k\}$.
In Figure 1 this diagram is shown schematically, we do not draw
all finite paths. 

\begin{figure}[htb!]\label{fig bndd}
\begin{center}
\includegraphics[scale=0.7]{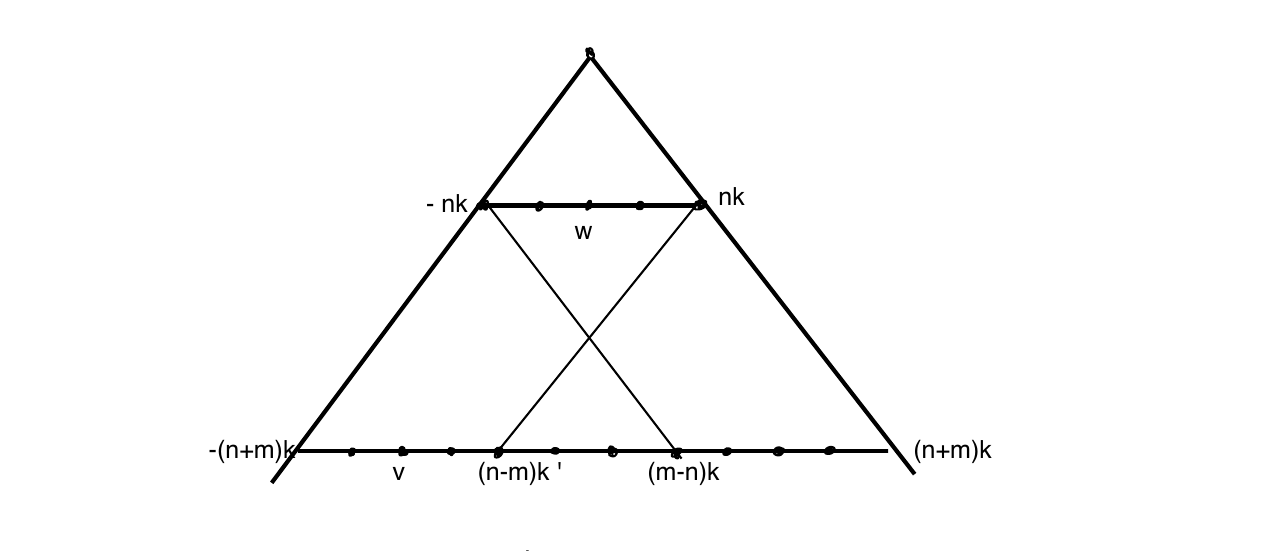}
\caption{The Bratteli diagram $B_k$.}
\end{center}
\end{figure}

The entries $f_{vw}^{(n)}$ of the incidence matrix $F_{n}$ 
are given by the formula:
$$
f_{vw}^{(n)} = \begin{cases}
    1, & |v-w| \leq k \\
    0, &  \mbox{otherwise},
\end{cases}  \ \ w\in V_n, v \in V_{n+1}. 
$$ 
For $w \in V_{n}$, we have 
$|s^{-1}(w)| = 2k + 1$. Similarly, for $v \in V_{n + 1}$, 
we have the following properties: 
$|r^{-1}(v)| = 2k + 1$ if $ -k(n-1) \leq v \leq k(n-1)$,
$|r^{(-1)}(v)| = 2k +1 - s$ if $ v = k (n-1)  + s$
or $v = - k(n-1) -s$ where $s = 1, ..., 2k$.

Our goal is to determine the heights $H_{w}^{(n)}$
of the clopen sets 
$X_{w}^{(n)}$ for all $n \geq 1$ and $w \in V_n$.
For this, we consider the sequence of functions
$$f_{n}(x) = ( x^{k} + \ldots + x + 1 + \ x^{- 1} + 
\ldots + x^{- k})^{n}.
$$
Every  $f_{n}(x)$ is represented as follows:
$$
f_{n}(x) = \sum_{w = -nk}^{n k}{K}_{w}^{(n)}\cdot x^{w}
$$
where ${K}_{w}^{(n)}$ are some natural coefficients. 

\begin{lemma}
For the Bratteli diagram $B_k$ defined above,     
\begin{equation}\label{e37}
H_{w}^{(n)} = K_{w}^{(n)}
\end{equation}
for each $w \in V_{n}$ and  $n \in \N$.
\end{lemma}

\begin{proof}
Clearly, $H_{w}^{(1)} = K_{w}^{(1)} = 1$ for every
$w \in V_{1}$. 

Next, we check that the heights
$H_{w}^{(n)}$'s and the coefficients
$K_{w}^{(n)}$'s satisfy the same inductive equations. 
By definition of the Bratteli diagram $B_k$, 
\begin{equation}\label{e38}
{H}_{v}^{(n + 1)} =  \sum_{w \in r^{- 1}(v)}H_{w}^{(n)}
\end{equation}
for $v \in V_{n+1}$. The set $V_{n+1}$ is divided into 
three subsets:  $I_0 =\{- k(n-1), ..., 0, ...,  k(n-1)\}$, 
$I_- =\{-k(n+1), ...,  -k(n-1)  -1 \}$, and 
$I_+= \{k(n-1) +1, ..., k(n+1) \}$. Then relation 
\eqref{e38} implies that the following formulas hold:
\be\label{eq-heights}
\ba
{H}_{v}^{(n + 1)} = &\ \sum_{w= v-k}^{v+k} H_w^{(n)}, \ 
& v \in I_0\\
{H}_{v}^{(n + 1)} =  &\ \sum_{w= -nk}^{v+k} H_w^{(n)}, \ 
& v \in I_-\\
{H}_{v}^{(n + 1)} = &\ \sum_{w= v-k}^{nk} H_w^{(n)},\ 
& v \in I_+.\\
\ea 
\ee 

To find inductive equations for the coefficients 
$K_{w}^{(n)}$'s, we use the following computation:
$$
\ba
f_{n + 1}(x) = &\  \sum_{v =-(n + 1)k}^{(n + 1)k}
K_{v}^{(n + 1)}x^{v} \\ 
= &\ \left\lbrack\sum_{w=-n k}^{n k}K_{w}^{(n)} x^{w}\right\rbrack
\cdot  ( x^{- k} + \ldots + x^{-1} + 1 + 
x^{1} + \ldots + x^{k})\\
= &\ \left\lbrack\sum_{w =- nk}^{nk}K_{w}^{(n)}x^{w-k} \right\rbrack + \ldots 
 + \left\lbrack\sum_{w=-nk}^{nk} K_{w}^{(n)}x^{w} \right\rbrack  + \ldots + 
\left\lbrack\sum_{w =-nk}^{nk}K_{w}^{(n)} x^{w + k} \right\rbrack.\\ 
\ea 
$$
Applying the inductive assumption that $H^{(n)}_w =
K^{(n)}_w$, we reorder the right-hand side of the above 
equality according to the powers of $x$ and compare the 
coefficients with those of $f_{n+1}(x)$.
As a consequence, we obtain that $K_{v}^{(n + 1)}$ 
satisfy all equalities (\ref{eq-heights}). In this way, we 
have proved (\ref{e37}).
\end{proof}

In what follows, we will use the method described in 
Section
\ref{sect_inverse limits}. We will first find the formulas 
for the entries ${g'}_{vw}^{(n,m)}$ of
matrices ${G'}^{(n,m)}$. 

Fix a vertex $w \in V_{n}$ and consider all paths from  $w$ 
to the vertices of $V_{n + m}$.
obviously, if we go along those paths, then we can 
reach only the vertices from the subset 
$\{ w-mk, w-mk+1, \ldots, w+mk-1, w+mk\} \subset V_{n + m}$.
The set of all paths from the vertex $w\in V_n$ to 
$V_{n + m}$ is identical to the subset of all paths from 
$0 \in V_0$ to the vertices of the level $V_m$. Here $w$ is any 
vertex from $V_n$.
Based on this observation, we conclude that 
$${g'}_{vw}^{(n,m)}= K_{v-w}^{(m)} = K_{-mk +s}^{(m)}$$  
where $s \in \{ 0, 1, \ldots, 2mk\}$ is determined by the relation 
$v= w-mk+s$. We set 
${g'}_{vw}^{(n,m)}= 0$ otherwise.  This fact can be used to 
describe  the entries of 
the $v$-th row ${\overline{g}'}_{v}^{(n,m)} =
\langle  {g'}_{vw}^{(n,m)} :\ w \in V_{n} \rangle$ of the 
matrix ${G'}^{(n,m)}$,  $v \in V_{n + m}.$
 We take a number $m$ considerably larger than $n$, 
because in the sequel we want to find
$\lim\limits_{m \to \infty}{\overline{g}}_{v}^{(n,m)}$, $v = v_{m}$. 

Based on  \eqref{eq-rows g_v}, \eqref{e37}, and 
the observation above, we have the following formula 
for the rows of the matrix ${G'}^{(n,m)}$.

\begin{lemma}\label{lem rows via K}
For $v =-(n+m) k+s \in V_{n+m}$, where  $s = 0, 1,\ldots, 
2(m+n)k$, we have
\be\label{eq-rows in K}
{\overline{g}'}_{v}^{(n,m)} = 
\begin{dcases}
\langle K_{- mk + s}^{(m)},  \ldots,
K_{- mk}^{(m)}, 0, \ldots, 0 \rangle,  &\ 
s = 0, 1, \ldots, 2nk-1\\
\langle K_{-mk + s}^{(m)},\ldots, K_{-mk + s -2nk}^{(m)}  
\rangle, &\ s = 2nk, \ldots, 2mk\\ 
\langle 0, \ldots, 0, K_{mk}^{(m)}, \ldots, 
K_{-mk+s-2nk}^{(m)} \rangle,  &\ 
s = 2mk+1, \ldots, 2(n+m)k
\end{dcases}
\ee 
\end{lemma}

Recall that, in the formulation of this lemma,  we assume that $m$ is considerably larger than $n$. 

To determine the set $\Delta^{(n,\infty, cl)}$, we use Theorem
\ref{thm 4.11 edited}. For this, we are looking for 
the set of all limit points of the vectors
${\overline{g}}_{v}^{(n,m)}$'s 
(when $v = v_{m},\ m \rightarrow \infty)$. 

Define the following probability vectors: 
\begin{equation}\label{e40}
\overline{ y}_{v}^{(n,m)} = \frac{1}{\sum_{w = - n k}^{w = nk}{g'}_{vw}^{(m,n)}} \cdot  {\overline{g}'}_{v}^{(n,m)},
\end{equation}
where $v = -(n+m) k, \ldots, (n+m)k \in V_{n+m}$.

Applying Lemma \ref{lem rows via K}, we get 
\be\label{eq-prob rows in K}
\overline{y}_{v}^{(n,m)} = 
\begin{cases} \dfrac{1}{\sum_{u=0}^s K_{- mk + s - u}^{(m)}}
\langle K_{- mk + s}^{(m)}, \ldots,
K_{- mk}^{(m)}, 0, \ldots, 0 \rangle,  & 
s = 0,  \ldots, 2nk-1\\
\\
\dfrac{1}{\sum_{u=0}^{s} K_{- mk + s + u}^{(m)}}
\langle  K_{-mk + s}^{(m)},\ldots, K_{-mk + s - 2nk}^{(m)} \rangle, 
&\ s = 2nk, \ldots, 2mk\\ 
\\
\dfrac{1}{\sum_{u=s}^{s-2nk} K_{- mk + s - u}^{(m)}}
\langle 0, \ldots, 0, K_{mk}^{(m)}, \ldots, 
K_{-mk+s-2nk}^{(m)} \rangle,  &\ 
s = 2mk+1, \ldots, 2(n+m)k
\end{cases}
\ee 

Hence, we have the following result. 

\begin{theorem}\label{t6}
The set $\Delta^{(n,\infty, cl)}$ is the set of all vectors 
\begin{equation}\label{e42}
 {\overline{q}}^{(n,\infty)} =  \frac{1}{\sum_{w = - nk}^{nk}{y_{w}^{(n,\infty)}K_{w}^{(n)}}}
\langle  y_{ - nk}^{(n,\infty)} K_{ - nk}^{(n)}, \ldots, y_{n k}^{(n,\infty)} K_{nk}^{(n)} \rangle,  
\end{equation}
 where 
 $\overline{y}^{(n,\infty)} =
\langle y_{ - n k}^{(n,\infty)}, \ldots, y_{nk}^{(n,\infty)} \rangle$
is a limit of vectors $\overline{y}_{v}^{(n,m)}$,
where $v= v_{m}$, $m \rightarrow \infty$.
\end{theorem}

\textbf{\textit{Question.}} Are there explicit formulas for the
numbers $K_i^{(m)}$? In this case, we could determine the
values of the tail invariant measures on cylinder sets using relation \eqref{e42}.

\subsection{Generalized Bratteli diagram of uniformly bounded size}

Fix again a natural number $k \geq 1$ and define the 
generalized Bratteli diagram $B_k =(V, E)$. For $V = 
\bigcup_n V_n$, let $V_{n} = \mathbb{Z}$ for  $n \in N_0$.
The set $E = \bigcup_n E_{n}$ consists of edges between 
$V_{n}$ and $V_{n + 1}$ where  
$e_{wv}$ exists if and only if $v \in \{w - k, \ldots, w, 
\ldots, w+k\}$, $w \in V_{n}$ and $v \in V_{n + 1}$.
The entries $f_{vw}^{(n)}$ of the incidence matrix $F_{n}$ are 
given by the formula
$$
f_{vw}^{(n)} = \begin{cases}
    1, &\ \mbox{if} \ |v-w|\leq k \\
    0, &\ \mbox{otherwise}.
\end{cases}
$$
For each $w \in V_{n}$, we have $|s^{(-1)}(w)| = 2k + 1$, 
and for each $v \in V_{n + 1}$, 
we also  have $|r^{( - 1)}(v)|= 2 k + 1$.

It follows from the definition of $B_k$ that
$H_{w}^{(n)} = (2k + 1)^{n}$ for every $w \in V_{n}$, 
$n \in \N_0$. As in Subsection \ref{ss finite uniformly bndd},
we can show that the entries of ${G'}^{(n,m)}$ are determined
by the formula:   
${g'}_{vw}^{(n,m)} = K_{v - w}^{(m)}$ for $v \in \{w-m k, 
\ldots, w+mk\}$, and ${g'}_{vw}^{(n,m)} = 0$, otherwise.

Hence, the $v$-th row of ${G'}^{(n,m)}$ has the form
$${\overline{g}'}_{v}^{(n,m)} = 
\langle \ldots, 0, K_{mk}^{(m)}, K_{m k - 1}^{(m)}, \ldots, 
K_{- m k}^{(m)}, 0,  \ldots\rangle, 
$$ 
where $K_{mk}^{(m)}$ is the $(v-m k)$-th entry.

Further, we have
$$g_{vw}^{(n,m)} = {g'}_{vw}^{(n,m)}\cdot \frac{H_{w}^{(n)}}{H_{v}^{(n + m)}} =
\frac{1}{{(2k + 1)}^{m}} \cdot  {g'}_{vw}^{(n,m)},
$$
and
$${\overline{g}}_{v}^{(n,m)} = \frac{1}{{(2k + 1)}^{m}} 
\cdot   \langle \ldots, 0, K_{m k}^{(m)}, \ldots, 
K_{- m k}^{(m)}, 0,  \ldots \rangle.
$$

\begin{prop}
There is no probability tail invariant measure on the 
generalized Bratteli diagram $B_k$.
\end{prop}

\begin{proof}
    We will show that the set of all limit vectors of
${\overline{g}}_{v}^{(n,m)}$'s is a zero vector 
$\ol 0 = \langle \ldots, 0, 0, 0, \ldots \rangle$. 

Note that $K_{0}^{(m)}$ is the largest number among ${K_{i}^{(m)}}$'s where 
$i = m  k, \ldots, - m k$. We prove that
$\lim_{m \rightarrow \infty}\dfrac{K_{0}^{(m)}}{{(2k + 1)}^{m}} = 0$. 
Indeed, since 
\begin{equation}\label{e43}
\frac{K_{0}^{(m + 1)}}{{(2k + 1)}^{m + 1}} =
\frac{1}{2 k + 1} \cdot  \frac{{\lbrack K}_{k}^{(m)}
 + \ldots + K_{0}^{(m)} + \ldots + K_{ - k}^{(m)}\rbrack}
 {(2 k + 1)^{m}} 
\leq \frac{K_{0}^{(m)}}{{(2k + 1)}^{m}},
\end{equation}
the limit $g_{0} = \lim_{m \rightarrow \infty}
\frac{K_{0}^{(m)}}{{(2\cdot k + 1)}^{m}}$ 
exists. We will show that $g_0 = 0$.

Choosing subsequences of $m$'s, we can define  
$g_{i} = \lim_{m \rightarrow \infty}\dfrac{K_{i}^{(m)}}
{(2k + 1)^{m}}$, where $i = k, \ldots, 1, -1, \ldots, -k$. 
Obviously,  $0 \leq g_{i} \leq g_{0}$. It follows from 
(\ref{e43}) that 
$$
\frac{2 k}{2k + 1}  g_{0} = \frac{1}{2 k + 1} \cdot [ g_{k}  + \ldots + g_{1} + g_{- 1} + \ldots + g_{- k}].
$$
This implies that $g_{i} = g_{0}$ for $i = k, \ldots, 1, -1, \ldots, -k$.

Finally, we have
$$(2k+1)  g_{0} = 
\lim_{m \rightarrow \infty}\frac{{\lbrack K}_{k}^{(m)} + \ldots + K_{0}^{(m)} + \ldots + K_{- k}^{(m)}\rbrack}
{(2k + 1)^{m}} \leq 1. 
$$
This implies that $g_{0}\  \leq \frac{1}{2k + 1}$.

In a similar way we can prove that
$g_{0} \leq \frac{1}{{(2\cdot k + 1)}^{2}}$ and so on. 
Hence $g_{0} = 0$, and the proposition is proved.
\end{proof}

\section{Examples}\label{sect examples}
In this section, we consider several examples of generalized
Bratteli diagrams. Our goal is to show how the methods, 
developed in Section \ref{sect_inverse limits}, can be 
applied to the study of probability tail invariant measures.

\subsection{ Bratteli diagrams and substitutions}

The shift dynamical systems associated with substitutions on a 
finite  alphabet 
have been studied by many authors; we mention here only several of them 
\cite{Fogg2002}, \cite{Queffelec2010},
\cite{DurandHostSkau1999}, 
\cite{DurandPerrin2022}, \cite{Putnam2018}, etc.
The primary interest is usually focused on minimal 
substitution dynamical systems. It was shown that substitution dynamical systems are completely described by 
stationary Bratteli diagrams. In 
\cite{BezuglyiKwiatkowskiMedynets2009}, the authors 
constructed Bratteli diagrams for aperiodic substitution 
dynamics. 

In recent years, substitution dynamical systems 
have been considered on a countable or even 
compact alphabet, see 
\cite{Ferenczi2006}, \cite{BezuglyiJorgensenSanadhya2024}, 
\cite{Manibo2023}, \cite{Manibo_Rust_Walton_2022}.
The problem of finding finite (or sigma-finite) invariant  measures for substitution dynamical systems is highly 
non-trivial in this case. 
In \cite{DomingosFMV2022} investigated 
the existence of invariant probability measures for substitutions on
countable alphabet.
In particular, they found a sufficient condition 
under which there are no invariant probability measures.  We show below that this sufficient condition is, in fact,  a consequence of Theorem \ref{thm 4.11 edited}. 

Let $A$ be a countable set (an alphabet), 
$A^{*}$ be the set of all finite words on $A$, and $A^{\mathbb{N}_0}$ be the set of infinite words on $A$. 
A \textit{substitution} is a map 
$\sigma \colon A \rightarrow A^{*}$ 
such that for every $a \in  A$, the finite word $\sigma(a)$ 
is not empty. We can extend a map $\sigma$ to $A^{*}$ and $A^{\mathbb{N}_0}$ by concatenation:
$\sigma(a_{0} , a_{1}, \ldots) =
\sigma(a_{0}) \sigma(a_{1}) \ldots$. 
In particular, we can define words 
$\sigma^{n}(a),\ a \in A,\ n \in \N$, by setting 
$\sigma^{1}(a) = \sigma(a) = u_{0} u_{1} \ldots {u}_{k}$, 
$\sigma^{n + 1}(a) = 
\sigma^{n} (u_{0}) \sigma^{n}(u_{1}), \ldots, \sigma^{n}(u_{k})$.
A substitution $\sigma$ determines a shift dynamical system 
$(X_{\sigma}, S)$, where
$X_{\sigma}$ is the set of all  sequences $u \in A^{\mathbb{N}_{0}}$ (or $u \in  A^{\mathbb{Z}}$) 
such that any finite subword of $u$ occurs in $\sigma^{n}(a)$ 
for some $a \in  A$ and $n \in \N$ and $S$ is the shift on
$X_\sigma$. 
To avoid unnecessary complications, we will assume that 
$|{\sigma}^{n}(a)| \rightarrow \infty$ as $n \to \infty$ for any $a \in A$
where $|\sigma^{n}(a)|$ denotes the length of $\sigma^{n}(a)$.

Now we define a Bratteli diagram $B_{\sigma} = (V, E)$ 
associated with a substitution $\sigma$. 
Set $V_{n} = A$ $n \in \N_0$. For $a \in V_{1}$, the set
$r^{-1}(a)$ consists of single edges $e(u_i, a)$
connecting $a$ and the vertices $u_0, \ldots, u_k$ where 
$(u_{0} u_{1}, \ldots {u}_{k}) = \sigma(a)$.
 
The Bratteli diagram $B_{\sigma} = (V, E)$ is stationary and 
its incidence matrix $M = (M_{ij})$, $i, j= 0, 1, 2, \ldots$ 
is the matrix associated to $\sigma$,
i.e., $M_{ij}$ is the number of occurrences of the letter  
$j$ in the word $\sigma(i)$. 
In general, the dynamics of the Bratteli diagram
$B_{\sigma} = (V, E)$ is not a good model for the shift dynamical system
($X_{\sigma}, S)$. 
However, if we assume that every pair $(uv)$ 
(where $u$ is the last letter of some $\sigma(i)$ and $v$ is the first letter of some $\sigma(j)$) appears inside some 
$\sigma(k)$,
then both dynamical systems $(X_{\sigma}, S)$ and the Vershik map $(X_B, \varphi_B)$ are isomorphic (for a more detailed 
discussion of this relation see \cite{BezuglyiKwiatkowskiMedynets2009}).

In \cite{DomingosFMV2022}, the authors proved the following
theorem.

\begin{theorem}\label{thm-Ferenczi}
Let $\sigma  : \N_0 \to  \N_0$ 
be a bounded length substitution such that $\sigma$ 
has a periodic point $u$ and the matrix $M = M_\sigma$ of 
the substitution is irreducible and aperiodic. If $M$ 
satisfies 
\be\label{eq-Ferenczi}
\lim_{n \to \infty} \sup_{i\in \N_0} \frac{M^{(n)}_{ij}}
{\sum_{k \geq 0} M^{(n)}_{ik} } = 0 \qquad \forall j \in N_0
\ee
then the dynamical system $(X_\sigma, S)$ has no finite 
invariant measure.
\end{theorem}

\begin{prop}
Condition (\ref{eq-Ferenczi}) implies that the Bratteli 
diagram $B_{\sigma} = (V, E)$ has no tail invariant 
probability measure. 
\end{prop}

\begin{proof}
We use our standard notation of matrices related to a 
Bratteli diagram. Note that $M = F'$ in our notation. 
Because $B_\sigma$ is a stationary Bratteli diagram, we have
${G'}^{(n,m)} = {F'}^{m}$ and, for the stochastic matrix 
$G^{(n,m)}$, we can write 
$$
\overline{g}_{i}^{(n, m)} = 
\left\langle\frac{{f'}_{ij}^{(m)} \cdot H_{j}^{(n)}}{ \sum_{k = 0}^{\infty} {f'}_{ik}^{(m)} \cdot H_{k}^{(n)} } \ :\ j = 0, 1, 2, \ldots 
\right\rangle.
$$
Then using \eqref{eq-Ferenczi} we have, for each $j \in \N_0$,
$$
g_{ij}^{(n, m)} = \frac{{f'}_{ij}^{(m)} \cdot H_{j}^{(n)}}
{\sum_{k = 0}^{\infty} {f'}_{ik}^{(m)} \cdot H_{k}^{(n)}} 
\leq
\frac{{f'}_{ij}^{(m)} \cdot H_{j}^{(n)}}{\sum_{k = 0}^{\infty} {f'}_{ik}^{(m)}} 
 \rightarrow 0
 $$ 
Thus, $\lim_{m \to \infty}{\overline{g}}_{i}^{(n, m)} = 0$ 
for any sequence $\{ i_{m} \}$ and therefore 
$\mathrm{\Delta}^{(n, cl, \infty)} = \{ \langle 0, 0, \ldots \rangle \}$
for each  $n \in \N_0$.
\end{proof}

\subsection{Reducible Bratteli diagrams with infinitely many odometers}

We consider here a class of reducible non-stationary 
generalized Bratteli diagrams $B = B_{IO}$ consisting of 
infinitely many odometers connected by single edges. This class of diagrams was first considered in \cite{BezuglyiKarpelKwiatkowski2024}, where the authors used the procedure of measure extension from a subdiagram to obtain results concerning the number of ergodic tail invariant measures. Here we recall the obtained results and show how apply methods developed in Section \ref{sect_inverse limits} to these diagrams. For more results concerning tail invariant measures and Vershik maps for reducible generalized Bratteli diagrams with infinitely many odometers see \cite{BezuglyiKarpelKwiatkowski2024}. 
%We will focus on the study of tail invariant measures and their extensions. 

Let the generalized Bratteli diagram $B = B_{IO}$ be defined 
by the sequence of incidence matrices 
\be\label{eq-m-x nostat DIO}
{F}_{n}^{'} =
\left[\begin{array}{llllllll}
 a_{n}^{(1)} &      1 &      0 &      0 & \ldots & \ldots & \ldots & \ldots  \\
     0 &  a_{n}^{(2)} &      1 &      0 &      0 & \ldots & \ldots & \ldots  \\
     0 &      0 &  a_{n}^{(3)} &      1 &      0 &      0 & \ldots & \ldots  \\
     0 &      0 &      0 &  a_{n}^{(4)} &      1 &      0 &      0 & \ldots  \\
\ldots & \ldots & \ldots & \ldots & \ldots & \ldots & \ldots & \ldots  \\
\end{array} \right\rbrack,\qquad n \in \N_0,
\ee
where the natural numbers $a_{n}^{(i)} \geq 2$ for all $n, i$.
The index $n$ points out at the $n$-th level of the diagram $B$,
and $i$ indicates the odometer supported by the $i$-vertex
in each level, see Figure 2 where a part of the diagram between levels $V_n$ and $V_{n+1}$ is shown ($a_{n}^{(i)}$ indicates the number
of vertical edges).

\begin{figure}[htb!]\label{fig odometer}
\begin{center}
\includegraphics[scale=0.7]{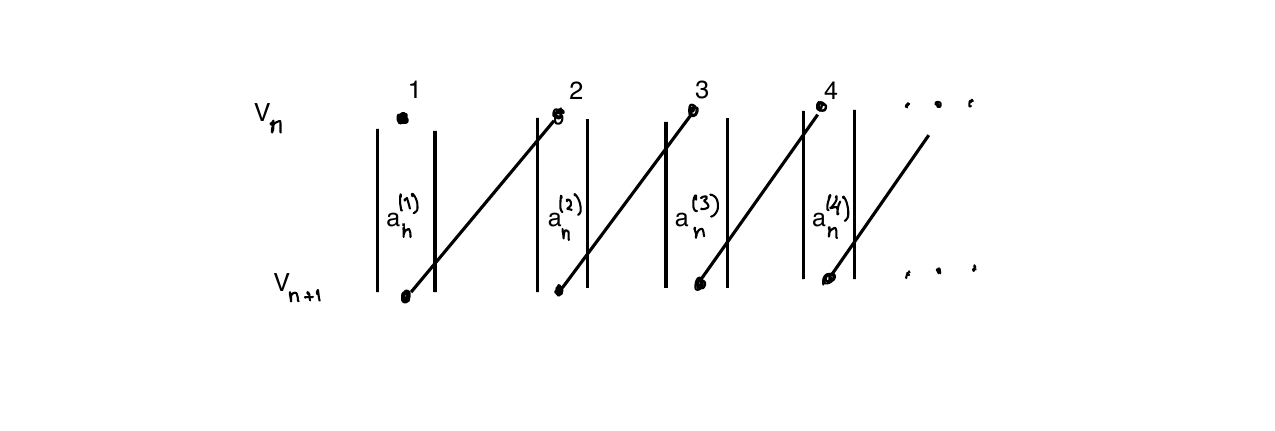}
\caption{The diagram $B_{IO}$.}
\end{center}
\end{figure}

The diagram $B_{IO}$ has a natural set of elementary vertex 
subdiagrams $\ol B(i)$ consisting of vertical odometers where
$i$ runs the set $\N$. 
There are exactly $a_n^{(i)}$ edges connecting the vertices
$i \in V_n$ and $i \in V_{n+1}$. 
 The subdiagram $\ol B(i)$ 
of $B$ admits a unique tail invariant probability measure 
$\ol \mu(i)$ on the path space $X_{\ol B(i)}$ such that for a 
cylinder set $[\ol e] = [e_1, \ldots, e_n], s(e_j) = r(e_j) =i$,
$$
\ol \mu(i) ([\ol e]) = \frac{1}{a^{(i)}_1\ \cdots\ 
a^{(i)}_n}.
$$
The measure extension procedure applied to $\ol B(i)$ 
gives us the measure $\wh{\ol \mu}(i)$ on the tail invariant set 
$\wh X_{\ol B(i)}$. It follows from Theorem 
\ref{TheoremIV_1} that 
$$
\wh{\ol \mu}(i)(\wh X_{\ol B(i)}) < \infty \ 
\Longleftrightarrow \ \sum_{n=1}^\infty \frac{H^{(n)}_{i+1}}
{a^{(i)}_1\ \cdots\ a^{(i)}_n} < \infty.
$$
Thus, it follows from the construction of $B_{IO}$ 
that there are infinitely many ergodic measures 
$\wh{\ol \mu}(i)$ on the path space $X_B$. Some of them may
be finite the others are infinite. We will give an example
below, for more examples see \cite{BezuglyiKarpelKwiatkowski2024}. Moreover, the measures  $\wh{\ol \mu}(i)$ 
and $\wh{\ol \mu}(j)$ are 
mutually singular ($i \neq j)$ because they are supported by 
non-intersecting tail invariant sets $\wh X_{\ov B(i)}$ and $\wh X_{\ov B(j)}$. Our goal is to show that there are no other ergodic measures.

\begin{remark}\label{rem-theta vs mu}
Let $\theta$ be a finite tail invariant measure on 
$\wh X_{\ol B(i)}$. Then there is a constant $C$ such that 
$\theta = C \wh{\ol \mu}(i)$. 

Indeed, let $C = \theta(X_{\ol B(i)})$, then $\ol \theta :=
\theta|_{X_{\ol B(i)}}= C \ol \mu(i)$. By tail invariance,
$$
\theta([e_0, \ldots, e_n]) =  \frac{C}{a^{(i)}_1\ \cdots\ 
a^{(i)}_n} = C \ol \mu(i)([e_0, \ldots, e_n]).
$$
\end{remark}

\begin{theorem}\label{prop-all erg meas}
Let $\mc M$ be the family of measures $\wh{\ol\mu}(i)$ 
such that $\wh{\ol \mu}(i)(\wh X_{\ol B(i)}) < \infty$. 
Then, after normalization, $\mc M$ coincides with the set of all ergodic probability 
tail invariant measures on the path space $X_B$ of the diagram 
$B$. 
\end{theorem}

\begin{remark} (i) The proof of Theorem~\ref{prop-all erg meas} can be found in \cite{BezuglyiKarpelKwiatkowski2024}, and it contains a version of the 
famous Rokhlin theorem about a canonical system of measures associated with a measurable partition. Note that $X_B$ is partitioned into sets $\wh X_{\ol B(i)}$ for $i = 1,2,\ldots$ Thus, for any probability ergodic invariant measure $\mu$ on $X_B$, we have $\mu(\wh X_{\ol B(i)}) = 1$ for some $i$. Hence by Remark~\ref{rem-theta vs mu}, we have 
$$\mu = \frac{\wh{\ov \mu}(i)}{\wh{\ov \mu}(i)\left(\wh X_{\ol B(i)}\right)}.$$

(ii) The result of Theorem \ref{prop-all erg meas} 
can be obtained in the case when the vertical
odometers are replaced with simple stationary 
standard Bratteli diagrams $\ol B_i$ with the incidence matrix $\ol F_i$. As for odometers, we will 
have a unique ergodic probability measure $\ol\mu_i$
on the path space $X_{\ol B_i}$. The measure $\ol\mu_i$
is completely determined by the values on cylinder sets
$\ol \mu_i([\ol e]) =\frac{\xi_v}{\lambda^n}$ where 
$\xi = (\xi_v)$ is the Perron-Frobenius eigenvector,
$\xi \ol F_i = \lambda\xi$ and $r(\ol e) = v \in V_n$, see, e.g. \cite{BezuglyiKwiatkowskiMedynetsSolomyak2010}.
Assuming that the extension
$\wh{\ol\mu}_i(\wh X_{\ol B_i})$ is finite, we get
that this measure is a unique ergodic measure (up to
a constant). The same arguments as in the proof of Theorem \ref{prop-all erg meas} can be repeated. 

(iii) For the Bratteli diagram $B_{IO}$,
the set of limit points $L(\{\ol g_v : v \in V_{n+m}, m \in 
\N \})$ coincides with $\{\lim_{m\to \infty} 
\ol g^{(n,m)}_{i_m}  \}$. 
Since the incidence matrices of the diagram are upper triangular,
we get that the limit of $\ol g^{(n,m)}_{i_m}$ is zero when $i_m 
\rightarrow \infty$ as $m \rightarrow \infty$. This means that
we should only consider the sequences $\{i_m\}$ such 
that $i_m = i$.
Denote by $\nu(i)$ a probability 
measure determined by the limit $\lim_{m\to \infty} 
\ol g^{(n,m)}_{i}$. Then it can be shown that the following 
fact holds:

\textbf{Claim.} If for some $i$ the extension of the measure $\ol \mu(i)$ is finite, then $\nu(i) = c_i \wh{\ol\mu}(i)$ 
for a constant $c_i$.
\end{remark}

From Theorem~\ref{prop-all erg meas} it follows, that $B_{IO}$ can have not more than countably many probability ergodic invariant measures. We formulate here a statement that was proved in  
\cite{BezuglyiKarpelKwiatkowski2024}. 
 
\begin{theorem}\cite{BezuglyiKarpelKwiatkowski2024}\label{Thm:DIO_measures_options}
A stationary generalized Bratteli diagram with infinitely many odometers can have only finitely many probability ergodic invariant measures. In particular, one can find diagrams which 
(i) have a unique probability ergodic invariant measure, (ii) have no probability invariant measure, but possess an infinite $\sigma$-finite invariant measure that takes finite values on all cylinder sets, and  (iii) have no invariant measure that takes finite values on all cylinder sets. A non-stationary generalized Bratteli diagram with infinitely many odometers can have countably infinitely many probability ergodic invariant measures.
\end{theorem}

Consider a non-stationary generalized Bratteli diagram $B$ defined by a sequence of 
natural numbers $\{a_{i} : i \in \N_0 \}$. Without 
loss of generality, we can assume that $a_n \geq 2$ for 
all $n$.
The diagram $B$ consists of an infinite sequence of
non-stationary odometers 
connected with the neighboring odometer by single edges.
More precisely, let  $V_{n} = \N$
and the incidence matrices $F^{'}_{n}$ has the form
\be\label{eq-m-x DIO}
{F}_{n}^{'} =
\left[\begin{array}{llllllll}
 a_{n} &      1 &      0 &      0 & \ldots & \ldots & \ldots & \ldots  \\
     0 &  a_{n} &      1 &      0 &      0 & \ldots & \ldots & \ldots  \\
     0 &      0 &  a_{n} &      1 &      0 &      0 & \ldots & \ldots  \\
     0 &      0 &      0 &  a_{n} &      1 &      0 &      0 & \ldots  \\
\ldots & \ldots & \ldots & \ldots & \ldots & \ldots & \ldots & \ldots  \\
\end{array} \right\rbrack,\qquad n \in \N_0.
\ee

Fix $i \geq 1$, set $W_{n} = \left\{ i \right\}$, $n = 1, 2, 
\ldots$ and define the subdiagram ${\overline{B}}_{(i)} = (\ol W, {\ol E})$ as above with the only difference that 
the set $\ol E_n$ is formed now by $a_n$ edges connecting 
the vertices $i \in V_n$ and $i\in V_{n+1}$. 
The unique tail invariant probability measure
$\overline{\mu} = \overline{\mu}(i)$ on
${\overline{B}}_{(i)} = (\ol W, {\ol E})$ is given by the formula
$\overline{\mu}([\ol e]) = \dfrac{1}{a_{0}\ \cdots\ a_{n}}$ 
where $r(\ol e) = i \in V_{n+1}$. 

By definition of the diagram, $H^{(n)} = H^{(n)}_i$ 
for all  $n$ and  $i$ (as usual, we set $H^{(0)}_i =1$). 
Then $H^{(n + 1)} = H_{i}^{(n + 1)}  = (a_{n} + 1) H^{(n)}$ 
which implies that 
$$H^{(n+1)} = (a_{0} +1) \ \cdots \ (a_{n} + 1), \quad n \in 
\N_0.
$$ 

The proof of the following statement can be found in \cite{BezuglyiKarpelKwiatkowski2024}.

\begin{prop} \label{prop sum a_n}
Let the sequence $(a_n)$ be such that $\sum_n a_n^{-1}
< \infty$. Then,
for every $i$, the extension $\wh{\ol \mu}(i)$ is finite.
The set of all ergodic finite tail invariant measures
on the path space of the diagram $B$ is 
formed by $\{\wh{\ol \mu}(i) \ : i \in \N\}$.
\end{prop}

The next statement is, in some sense, the converse to Proposition
\ref{prop sum a_n}.

\begin{prop}
Let the Bratteli diagram $B$ be defined by the sequence of
incidence matrices $F'_n$, see \eqref{eq-m-x DIO}. 
If the series $\sum_{n\geq 1} a_n^{-1}$ diverges, then 
the diagram does not admit finite tail invariant measures.
\end{prop}

\begin{proof}
The following formulas hold for the diagram $B$ (they 
can be easily proved by induction taking into account that
the height of a tower does not depend on the vertex):
$$
H^{(n)}_i = (a_1 +1) \ \cdots\ (a_{n-1} +1), \quad i \in \N.
$$
The entries of the matrix ${G'}^{(n,m)}$ can be found as 
follows:
$$
{g'}^{(n,m)}_{ii} = a_n \ \cdots \ a_{n+m-1}
$$
and, for $j =  i+1, \ldots i+ m$,
$$
{g'}^{(n,m)}_{ij} = (a_n \ \cdots\ a_{n+m-1}) \left[ 
\sum_{s_1 < s_2 < \cdots < s_{j-i}} 
\frac{1}{a_{s_1} \ \cdots \ a_{s_{j - i}}}\right]
$$
where $ n \leq s_1 < \cdots < s_{j-i} \leq n +m-1$.
Clearly,  ${g'}^{(n,m)}_{ij} =0$ for other $i, j$.
Now we can find the entries of the stochastic matrix 
${G}^{(n,m)}$: if $j = i+1, \ldots, i+m$, then
$$
{g}^{(n,m)}_{ij} = {g'}^{(n,m)}_{ij} \frac{H^{(n)}_j}
{H^{(n+m)}_i} = \left\lbrack \frac{a_{n}}{1 + a_{n}}
\ \cdots \ 
\frac{a_{n + m - 1}}{1 + a_{n + m - 1}} \right\rbrack \cdot 
\left[
\sum_{s_1 < s_2 < \cdots < s_{j-i}} 
\frac{1}{a_{s_1} \ \cdots \ a_{s_{j - i}}}\right]
$$
and 
$$
{g}^{(n,m)}_{ii} = \frac{a_{n}}{1 + a_{n}}
\ \cdots \ \frac{a_{n + m - 1}}{1 + a_{n + m - 1}} 
$$

Assume now that the set of tail invariant probability measures
$M_1(B)$ is not empty. Then, by Theorem \ref{thm 4.11 edited},
there is some $i$ such that the sequence of vectors 
$\{{\overline{g}}_{i}^{(n, m)}\}$ has the limit as $m \to 
\infty$. This means that 
$$
\lim_{m \to\infty} {g}^{(n,m)}_{ii} = 
\lim_{m \to\infty} \frac{a_{n}}{1 + a_{n}}
\ \cdots \ \frac{a_{n + m - 1}}{1 + a_{n + m - 1}} 
$$
exists. The latter is equivalent to
the convergence of the series $\sum_{n\geq 1} a_n^{-1}$. 
This is a contradiction.  
\end{proof}

The corresponding stochastic matrices $F_{n}$'s have 
the form
$$
{F}_{n} =
\left[\begin{array}{llllllll}
 \frac{a_{n}}{a_{n} + 1} &     \frac{1}{a_{n} + 1} &                        0 &                        0 &              \ldots & \ldots & \ldots & \ldots  \\
                       0 & \frac{a_{n}}{a_{n} + 1} &      \frac{1}{a_{n} + 1} &                        0 &                   0 & \ldots & \ldots & \ldots  \\
                       0 &                       0 &  \frac{a_{n}}{a_{n} + 1} &      \frac{1}{a_{n} + 1} &                   0 &      0 & \ldots & \ldots   \\
                  \ldots &                  \ldots &                   \ldots &                   \ldots &              \ldots & \ldots & \ldots & \ldots  \\
\end{array} \right\rbrack,
$$
i.e., $f_{ii}^{(n)} = \frac{a_{n}}{(a_{n} + 1)}$,
$f^{(n)}_{i(i + 1)} = \frac{1}{(a_{n} + 1)}$,
${f}^{(n)}_{ij} = 0$ for $j \neq i, i + 1$, $i \in \N$. 
Let
$\overline{g}_{i} = \left\langle f^{(n)}_{ij} : \ 
j \in \N \right\rangle$
be the $i$-th row of the matrix $F_{n}$. 
Then
$|{\overline{g}}_{i}| \leq \frac{1}{2^{a(i)}} + \frac{1}{2^{a(i + 1)}} \to 0$
as $i \to \infty$ where $a(i)$ is an enumeration of 
vertices in $V_n$. 
It follows from Theorem \ref{t2} that the linear maps
$F_{n}^{T} \colon \Delta_1 \to \Delta_1$ 
are continuous. 

We observe that the above statement is true in a more general situation when the entries of an incidence  matrix $F$ under its main diagonal are 
zeros.

\section{Uncountably many ergodic probability tail 
invariant measures}\label{ssect measures on B_infty}

This section studies a class of reducible generalized Bratteli diagrams $B_\infty$ whose incidence matrices are triangular. 
It turns out that such diagrams
have uncountably many ergodic probability tail invariant measures. 
We also consider subdiagrams of $B_\infty$, standard and generalized,
and answer the questions about internal tail invariant measures
on such subdiagrams and the finiteness of their extensions.

\subsection{Triangular generalized Bratteli diagram \texorpdfstring{$B_\infty$}{B_\infty}} 
We define the diagram $B_\infty = (V, E)$ by taking $V_n = \N$ 
for all $n \in \N$ so that each vertex $v\in V_n$ can be 
written as $(n,i)$. 
For $v= (n, i)\in V_{n}$, an edge $e(v,u)$ where $u \in 
V_{n + 1}$, 
exists whenever $u = (n+1, j),\ j = i, i+1, \ldots$.
The incidence matrices $F'_{n}$ are the same for all 
levels, $F'_{n} =F'$, where
$$
F' = 
\left\lbrack \begin{array}{ccccccc}
1 & 0 & 0 & 0 & \ldots &\ldots & \ldots \\
1 & 1 & 0 & 0 & 0      &\ldots & \ldots    \\
1 & 1 & 1 & 0 & 0      & 0     & \ldots  \\
\ldots & \ldots & \ldots &  \ldots & \ldots & \ldots & \ldots  \\
\end{array} \right\rbrack.
$$

The stochastic matrix $F$ obtained from the matrix $F'$ has a form
$$
F =
\left[\begin{array}{llllllll}
          1 &           0 &           0 &      0 &  \ldots & \ldots & \ldots & \ldots  \\
\frac{1}{2} & \frac{1}{2} &           0 &      0 &       0 & \ldots & \ldots & \ldots  \\
\frac{1}{3} & \frac{1}{3} & \frac{1}{3} &      0 &       0 & \ldots & \ldots & \ldots  \\
     \ldots &      \ldots &      \ldots & \ldots &  \ldots & \ldots & \ldots & \ldots  \\
\end{array} \right\rbrack.
$$
We claim that the matrix $F$ satisfies the assumption of Theorem \ref{t2}. 
Indeed, the $i$-th row ${\overline{g}}_{i} =
\left\langle \frac{1}{i},  \ldots, \frac{1}{i}, 0, \ldots \right\rangle$  of $F$ satisfies the relation: 
$|{\overline{g}}_{i}|  \leq \left\lbrack 
\frac{1}{2^{a(1)}} + \frac{1}{2^{a(2)}} + \ldots \right\rbrack \cdot \frac{1}{i}
\leq \frac{2}{i} \to 0$, as $i \to \infty$ 
where $a(\cdot)$ is an enumeration of the vertices. 
Thus, ${G'}^{(n,m)} = {F'}^{m} = \{f_{ij}^{(m)} 
i, j \in \N\}$ for all $n \geq 1$, and these matrices generate continuous mappings.

To find the matrices 
${F'}^{n} = \left( {f'}_{ij}^{n} \ : i, j = 1, 2, \ldots \right)$, we will use the numbers the 
$S_{i}^{(k)},\ k \in \N_0,\ i\in \N$, which are defined as follows:
\be\label{eq=prop of S}
S_{i}^{(0)} = 1, \quad 
S_{i}^{(k + 1)} = S_{1}^{(k)}  + \ldots + S_{i}^{(k)}
\ee 
The numbers $S_{i}^{(k)}$ are used in the Cesaro summability 
method, and it is known \cite{Hardy1992} that they are 
\be\label{eq-numbers S}
S_{i}^{(k)} = 
\begin{pmatrix}
i + k - 1 \\
k \\
\end{pmatrix}.
\ee
Using the induction it is not hard to prove that 
${f'}_{ij}^{(n)} = S_{i - j + 1}^{(n - 1)}$
whenever $j = 1, \ldots, i$ 
and ${f'}_{ij}^{(n)} = 0$, otherwise 
($i = 1, 2, \ldots,\ n = 1, 2, \ldots$). Thus,
\be\label{eq-F' for x2}
{F'}^{(m)} = \left[ \begin{array}{ccccccc}
S_{1}^{(m - 1)} & 0                     & 0                     & 0      & \ldots & 0 & \ldots \\
S_{2}^{(m - 1)} & S_{1}^{(m - 1)}       & 0                     & 0      & \ldots & 0  & \ldots \\
\ldots          & \ldots                & \ldots                & \ldots & \ldots & \ldots & \ldots  \\
S_{k}^{(m - 1)} & S_{k - 1}^{(m - 1)} & S_{k - 2}^{(m - 1)} & \ldots & \ldots & S_{1}^{(m - 1)} & \ldots \\
\ldots          & \ldots                & \ldots                & \ldots & \ldots & \ldots & \ldots \\

\end{array} \right],
\ee 
It follows from \eqref{eq-F' for x2} that 
\be\label{eq=H=S}
H_{i}^{(n)} = S_{i}^{(n - 1)} = \begin{pmatrix}
i + n - 2 \\
n-1 \\
\end{pmatrix},\quad n, i \in \N,
\ee 
and we find that 
$$
g_{ij}^{(n,m)} = \frac{{f'}_{ij}^{m}  H_{j}^{(n)}}
{\sum_{s = 1}^{i}{{f'}_{is}^{(m)}  H_{s}^{(n)}}}
=
\frac{S_{i - j + 1}^{(m - 1)} H_{j}^{(n)}}{\sum_{s = 1}^{i}{S_{i - s + 1}^{(m - 1)} H_{s}^{(n)}}}
$$
if $j = 1,  \ldots i$ and $g_{ij}^{(n,m)} = 0$ otherwise.

To describe the set $\Delta^{(n, \infty, cl)}$, 
we use the method developed in 
Theorem  \ref{thm 4.11 edited}. 
For this, we have first to find the set of all limit points $\ol {y}^{(n, \infty)}$ of the vectors
$$
\overline{y}_{i}^{(n, m)} = \left\langle {y}_{i j}^{(n,m)} = \frac{{f'}_{ij}^{(m)}}{\sum_{s = 1}^{\infty}{f'}_{is}^{(m)}} 
\ : j = 1, 2, \ldots \right\rangle, \quad i\in \N,
$$ 
when $m \to \infty$. 

\begin{lemma}\label{lem limi vectors p_a}
The set of limits of the vectors 
${\overline{{y}}}_{i}^{(n, m)}$ is formed by the vectors 
$$
\overline{{y}}_{a}^{(n, \infty)} = 
 \lim_{m \to \infty}\overline{{y}}_{i_m}^{(n,m)} = 
\left\langle \frac{1}{(a + 1)}, \frac{a}{{(a + 1)}^{2}}, \frac{a^{2}}{{(a + 1)}^{3}}, \ldots \right
\rangle,\quad  0\leq a < \infty.
$$
\end{lemma}

\begin{proof}
Recall that we have 
${f'}_{ij}^{(m)} = S_{i - j + 1}^{(m - 1)}$ if $j = 1, \ldots, 
i$ and ${f'}_{ij}^{(m)} = 0$ if $j > i$. Moreover, 
$$\sum_{s = 1}^{\infty} {f'}_{is}^{(m)} = 
\sum_{s = 1}^{i} S_{i - s + 1}^{(m - 1)} 
= S_{i}^{(m)}.
$$ 
Let  
$\overline{{y}} = \lim_{m \to \infty}
{\overline{{y}}}_{i}^{(n,m)},\ i= i_{m}$.
Taking a subsequence of $\{m\}$, we can assume that 
$\dfrac{i_{m}}{m} \rightarrow a$ for some $0 \leq a < \infty$.
Fix $j \geq 1$. Since $i_{m} \rightarrow \infty$, 
we can assume that $i_{m} > j$. Compute for $j= 2, \ldots, i$
$$
{y}_{ij}^{(n,m)} = \frac{S_{i - j + 1}^{(m - 1)}}{S_{i}^{(m)}} =
\frac{\begin{pmatrix}
i + m - j - 1 \\
m - 1 \\
\end{pmatrix}\ }{\begin{pmatrix}
i + m - 1 \\
m \\
\end{pmatrix}} =
\left[ \frac{(i + m - j - 1)!}{(m - 1)! \cdot (i - j)!} \right] 
\cdot \left[ \frac{m! \cdot (i - 1)!}{(i + m - 1)!} \right] 
$$
$$
= \frac{m \cdot (i - j + 1) \cdot \ldots \cdot (i - 1)}
{(i + m - j) \cdot \ldots \cdot  (i + m - 1)} =
\frac{ \left[ \frac{(i - j + 1)}{m} \cdot \ldots \cdot \frac{(i - 1)}{m} \right] }{\left[ \frac{(i + m - j)}{m} \cdot \ldots \cdot \frac{(i + m - 1)}{m}\right] }
$$
If $j = 1$, we have ${y}_{i1}^{(n, m )} = \dfrac{m}
{i+m-1}$. Taking the limit in the formulas for 
${y}_{ij}^{(n,m)}$ as $m \to \infty$, we get that 
$$
\lim_{m\to \infty} {y}_{ij}^{(n,m)} = 
\lim_{m \to \infty}{\frac{ \left[ \frac{(i - j + 1)}{m} \cdot \ldots \cdot \frac{(i - 1)}{m} \right] }{ \left[ \frac{(i + m - j)}{m} \cdot \ldots \cdot \frac{(i + m - 1)}{m} \right] }} 
= \frac{a^{j - 1}}{{(a + 1)}^{j}}.
$$
If the sequence $\{ i_{m} \}$ is bounded (we can assume 
$i_{m} = i$), then $a = 0$, and in this case
${\overline{p}}_{0}^{(n, \infty)} = \langle 1, 0, 0, \ldots \rangle$.

This proves the lemma.
\end{proof}

Following the method described in Section 
\ref{sect_inverse limits}, we need to find the limit set
$L^{(n)}(\{\ol g_v \})$ in terms of the vectors 
$\overline{{y}}_{a}^{(n, \infty)}$. 

\begin{lemma}\label{lem:H(n,a)}
The vectors $\overline{{y}}_{a}^{(n, \infty)}$ belong to the set
$P^{(n)}$ and satisfy conditions \eqref{eq-conditions P(n)} 
(defined in Theorem \ref{thm 4.11 edited}) 
and \eqref{eq limit vectors p(infty)}. That is
\be\label{eq-H(n,a)}
H(n, a):= \sum_{u \in V_n} {y}_u^{(n, \infty)} H_u^{(n)}< \infty
\ee 
and 
\be\label{eq-lim H(n,a)}
\lim_{m \to \infty} \sum_{u \in V_n} 
\left[ \frac{{g'}^{(m,n)}_{vu}}{\sum_{w \in V_n}{g'}^{(m,n)}_{vw}} \right] H_u^{(n)} = H(n, a).
\ee
\end{lemma}

\begin{proof}
We first show that, for all $n \geq 1$,
\be \label{eq-formula for H(n,a)}
H(n, a) = (1 + a)^{n-1}.
\ee
Note that by Lemma \ref{lem limi vectors p_a} 
$$
H(n, a) =\sum_{j=1}^\infty \frac{a^{j-1}}{(1 +a)^j} H_j^{(n)}. 
$$
For $n =1$, $H_j^{(1)} =1$ and 
$$
H(1, a) = \frac{1}{1 + a} \sum_{j=1}^\infty 
\left(\frac{a}{1 +a}\right)^j = 1. 
$$
In the following computation, we show that $H(n,a)$ is finite. 
This fact proves that \eqref{eq-H(n,a)} holds.
$$
\ba
H(n,a) = &\ \sum_{j=1}^\infty \frac{a^{j-1}}{(1 +a)^j} H_j^{(n)} 
= \sum_{j=1}^\infty \frac{a^{j-1}}{(1 +a)^j} S_j^{(n-1)}\\
= & \  \sum_{j =1}^{\infty}\frac{a^{j - 1}}{(a + 1)^{j}}  
\left( \begin{array}{c}
j + n - 2 \\
n - 1 \\
\end{array} \right)  = 
\frac{1}{(n - 1)!} \cdot
\sum_{j = 1}^{\infty}\frac{a^{j - 1}}{(a + 1)^{j}} 
\frac{(j + n - 2)!}{(j - 1)!}\\
= &\ \frac{1}{(n - 1)!} \cdot
\sum_{j = 1}^{\infty}\frac{a^{j - 1}}{(a + 1)^{j}} \cdot 
\left[j (j+1) \cdot \ldots \cdot (j+n-2) \right]  < \infty.
\ea
$$

Next, we prove that \eqref{eq-formula for H(n,a)} holds. 
$$
\ba 
H(n+1, a) = &\ \sum_{j=1}^\infty \frac{a^{j-1}}{(1 +a)^j} H_j^{(n+1)} =  \sum_{j=1}^\infty \frac{a^{j-1}}{(1 +a)^j} 
\sum_{l=1}^j H_l^{(n)} \\
= &\ \sum_{l=1}^\infty H_l^{(n)} \sum_{j=l}^\infty 
\frac{a^{j-1}}{(1 +a)^j} = \sum_{l=1}^\infty H_l^{(n)} 
\left(\frac{a}{1 + a} \right)^{l-1}\\
= &\ (1+a) H(n, a),\\
\ea 
$$
and \eqref{eq-formula for H(n,a)} follows. 

It remains to show that condition \eqref{eq-lim H(n,a)} is 
satisfied.
It was proved in Lemma \ref{lem limi vectors p_a} that 
\be\label{eq < >}
\sum_{j = 1}^{i} {y}_{ij}^{(n, m)} H_{j}^{(n)} =  
{y}_{i, 1}^{(n,m)} + 
\sum_{j = 2}^{i} H_j^{(n)} \frac{\frac{i-j+1}{m}\ \cdots \
\frac{i-1}{m}}{\frac{i +m-j}{m}\ \cdots \
\frac{i +m -1}{m}}
\ee 
Fix $k \geq 1$ and take $i = i_m >k$ as $m \to \infty$. Then 
$$
\sum_{j = 1}^{i} {y}_{ij}^{(n, m)} H_{j}^{(n)} >   
\frac{m}{i+m -1} + 
\sum_{j = 2}^{k} H_j^{(n)} \frac{\frac{i-j+1}{m}\ \cdots \
\frac{i-1}{m}}{\frac{i +m-j}{m}\ \cdots \
\frac{i +m -1}{m}}
$$
As $m \to \infty$ (remember that $\dfrac{i_{m}}{m} \to a$), we 
get from the above inequality that for every $k \in \N$
$$
\liminf_{m \to \infty} \sum_{j = 1}^{i}{p_{ij}^{(n,m)} H_{j}^{(n)}} \geq 
\sum_{j = 1}^{k} \frac{a^{j - 1}}{(a + 1)^{j}} H_{j}^{(n)}.
$$
Hence
\be\label{eq-liminf}
\liminf_{m \to \infty} \sum_{j = 1}^{i}{{y}_{ij}^{(n,m)} H_{j}^{(n)}} \geq 
\sum_{j = 1}^{\infty} \frac{a^{j - 1}}{(a + 1)^{j}} H_{j}^{(n)}
= H(n, a) = (a+1)^{n-1}.
\ee

Take  $b > a$ (then $\frac{a}{a+1} < \frac{b}{b+1}$). It follows 
from $\frac{i_m}{m} \to a$ that, for sufficiently large $m$, 
$$
\frac{\frac{i - l}{m}}{\frac{i- l}{m} + 1}
< \frac{b}{b + 1},\quad l= 0, 1, \ldots j-1, \ i = i_m.
$$
Therefore \eqref{eq < >} can be estimated from above 
$$
\sum_{j = 1}^{i} {y}_{ij}^{(n, m)} H_{j}^{(n)} < 
\frac{m}{i+m -1} + \sum_{j=2}^i H_j^{(n)} \frac{m}{m + i -j} \left( \frac{b}{1+b}\right)^{j-1}
$$
Taking the limit when $m \to \infty$, we obtain the inequality
$$
\limsup_{m\to \infty} \ \sum_{j = 1}^{i} {y}_{ij}^{(n, m)} 
H_{j}^{(n)} \leq \frac{1}{1+a} +  
 \sum_{j=2}^\infty H_j^{(n)} \frac{1}{1+a} 
\left( \frac{b}{1+b}\right)^{j-1}. 
$$
Since this relation holds for all $b > a$, we can deduce that
\be\label{eq-limsup}
\ba 
\limsup_{m\to \infty} \ \sum_{j = 1}^{i} {y}_{ij}^{(n, m)} 
H_{j}^{(n)} \leq &\  \frac{1}{1+a} + \lim_{b \to a} \left(
\sum_{j=2}^\infty H_j^{(n)} \frac{1}{1+a} 
\left( \frac{b}{1+b}\right)^{j-1} \right)\\
= &\ \frac{1}{1+a} + 
\sum_{j=2}^\infty H_j^{(n)} \frac{a^{j-1}}{(1+a)^j}\\
= &\ H(n,a)\\
= &\ (a+1)^{n-1}.
\ea 
\ee 
It follows from \eqref{eq-liminf} and \eqref{eq-limsup} that 
$$
\lim_{m\to \infty} \ \sum_{j = 1}^{i} {y}_{ij}^{(n, m)} 
H_{j}^{(n)} = (a+1)^{n-1} = \sum_{j = 1}^{\infty} 
y_{ij}^{(n, m)} H_{j}^{(n)},
$$
and \eqref{eq-lim H(n,a)} is proved. 
\end{proof}

Let the vectors $\ol q_{{a}}^{(n, \infty)} = 
\langle q_{{a,i}}^{(n, \infty)} : \ i \in V_n\rangle$ be defined 
as in Theorem \ref{thm 4.11 edited}, see \eqref{e24}, and let $\ol p_a^{(n, \infty)} = \left\langle p_{a, i }^{(n, \infty)} \ |\ i \in V_n \right\rangle$ be such that  $q_{a, i}^{(n, \infty)} = H_i^{(n)}  p_{a, i}^{ (n, \infty) },\ n, i = 1, 2, \ldots$.

Then we
can find the entries of $\ol q_{a}^{(n, \infty)}$:
$$
\ba 
q_{a,i}^{(n, \infty)} =  &\frac{1}{ \sum_{j\geq 1}
p^{(n, \infty)}_j H_j^{(n)} } \cdot p^{(n, \infty)}_i 
H_i^{(n)} \\
= &\ \frac{1}{(a+1)^{n-1}}\frac{a^{i-1}}{(a+1)^i} S_i^{(n-1)}\\
= &\ \frac{a^{i-1}}{(a+1)^{n +i -1}} 
\left( \begin{array}{c}
n + i- 2 \\
n-1
\end{array} \right), \qquad i \geq 1.
\ea 
$$
Thus, we proved that 
\be\label{eq=q}
\overline{q}_{a}^{(n,\infty)} = 
\left\langle 
\frac{1}{{(a + 1)}^{n}},  \ 
\frac{a}{{(a + 1)}^{n + 1}} 
\left( 
\begin{array}{c}
n  \\
1
\end{array} \right), 
\frac{a^{2}}{{(a + 1)}^{n + 2}}
\left( \begin{array}{c}
n + 1 \\
2
\end{array} \right), \ 
 \ldots
\right\rangle
\ee  
According to Theorem \ref{thm 4.11 edited}, we get that 
$$
L^{(n)} (\{ \overline{g}_{i}^{(n,m)} : i \in V_{n + m}, 
m \in \N\}) = 
\{ \overline{q}_{a}^{(n, \infty)} :\ 0 \leq a < \infty \}.
$$
It follows from (\ref{eq=q}) that
$$\ol p_{a}^{(n, \infty)} = 
\left\langle \frac{1}{(a + 1)^{n}}, 
\frac{a}{(a + 1)^{n + 1}}, 
\frac{a^{2}}{(a + 1)^{n + 2}},  
\ldots
\right\rangle =
\left\langle
\frac{a^{(i - 1)}}{H(n, a) \cdot (a + 1)^{i}} \ |\ i \in \N
\right\rangle,
$$
because 
$H_i^{(n)} = 
\left(\begin{array}{c}
i + n - 2\\
n - 1
\end{array}\right)
=
\left(\begin{array}{c}
i + n - 2\\
i - 1
\end{array}\right),
\ n, i \in \N.
$

\begin{prop}\label{prop mu_a} 
For every $0 \leq a < \infty$, the sequence of vectors
$\{ \overline{q}_{{a}}^{(n, \infty)}: \ n \in \N \}$ 
(or $\{ \overline{p}_{ a }^{(n, \infty)}: \ n \in \N \}$) 
determines uniquely a tail invariant probability measure 
$\mu_{a}$ on the 
path space of the Bratteli diagram $B_\infty = (V, E)$. 
\end{prop}

\begin{proof}
It suffices to show that 
\be\label{eq8.12}
{F'}^T ( \overline{p}_{a}^{(n+1, \infty)} ) = 
\overline{p}_{a}^{(n, \infty)},\quad n \in \N.
\ee
We compute
$$
\ba 
{F'}^T \left\langle \frac{a^{i-1}}{H(n +1,a) 
(a +1)^i} :\ i \in \N \right\rangle 
= &\  \frac{1}{H(n+1, a)} \left\langle \sum_{i =j}^\infty 
\frac{a^{i}}{(a+1)^{i+1}} :\ j = 0,1, 2, \ldots
\right\rangle \\
= &\ \frac{1}{H(n+1, a)} \left\langle  
\frac{a^j}{(a+1)^{j}} : \ j=0,1, 2, \ldots \right\rangle \\
= &\ \frac{1}{H(n, a)} \left\langle  
\frac{a^j}{(a+1)^{j+1}} : \ j =0,1, 2, \ldots \right\rangle\\
= &\ {\overline{p}_{a}^{(n, \infty)}}.
\ea 
$$

Thus, we have
$$
\mu_a([\ol e]) = \frac{a^{j-1}}{H(n,a) (a+1)^j}.
$$
and 
$$
\mu_a(X^{(n)}_j) = \frac{a^{j-1}}{H(n,a) (a+1)^j} H^{(n)}_j =
 \frac{a^{j-1}}{(a+1)^{n+j-1}}
\left( 
\begin{array}{c}
n + j - 2 \\
n-1
\end{array} \right).
$$
One can easily check that $\mu_a$ is a probability measure 
since, by definition of $H(n,a)$, we have
$$
\sum_{j \geq 1}\mu_a(X^{(n)}_j) =\frac{1}{H(n,a)}\sum_{j\geq 1}
\frac{a^{j-1}}{ (a+1)^j} H^{(n)}_j =1,\qquad n \in \N.
$$
\end{proof}

Our next goal is to show that every measure $\mu_a$ is ergodic. 
For this, we need to use the notion of \textit{completely
monotonic sequences} and their relations with the 
\textit{Hausdorff moment problem}, see e.g. \cite{Widder1941},  
\cite{Akhiezer2021}.  

Let $\ol c = \{c_i : \ i \in \N\}$ be a given sequence of 
positive numbers. We define new sequences $\Delta^n(\ol c), 
n\in \N_0$, where $\Delta^{0}(\overline{c}) = \overline{c}$
and the sequences 
$\Delta^1 ( \overline{c} ), \Delta^{2}(\overline{c}), \ldots$ 
(called  the sequences of differences of $\overline{c}$) 
are defined as follows:
$$
\Delta^1(\overline{c}) = \left\{ c_{1} - c_{2}, c_{2} - c_{3}, c_{3} - c_{4}, \ldots \right\},
$$
$$
\Delta^{2}(\overline{c}) = \Delta^1( \Delta^1(\overline{c})), 
$$
and so on, $\Delta^{n}(\ol c) = \Delta^1(\Delta^{n-1}(\ol c))$.

\begin{remark}
In the literature, one can see another definition of 
the difference operator $\delta$ acting on sequences $\ol c = 
\{c_n\}_n$: $\ol c$. By definition,  
$\delta c_n = c_{n+1} - c_n$ and 
$\delta^{k+1} (\ol c) = \delta(\delta^k(\ol c))$. Then 
$\Delta^k(\ol c) = (-1)^k \delta^k(\ol c)$.
\end{remark}

The terms of the sequences $[\Delta^k(\ol c)]_i$ can be found 
by the following formulas: 
$$
\lbrack \Delta^{2} \left( \overline{c} \right) \rbrack_{i}
= c_{i} - 2 \cdot c_{i + 1} + c_{i + 2},
$$
$$
\lbrack \Delta^{3} ( \overline{c} ) \rbrack_{i} =
c_{i} - 3  c_{i + 1} + 3 c_{i + 2} - c_{i + 3}, 
$$
and, in general,
$$
\lbrack \Delta^{k} ( \overline{c}) \rbrack_{i} =
c_{i}- \left( \begin{array}{c}
k \\
1 \\
\end{array} \right) c_{i + 1} + 
\left( \begin{array}{c}
k \\
2 \\
\end{array} \right) c_{i + 2} - 
\ldots 
+ (-1)^{k-1} 
\left( \begin{array}{c}
k \\
k - 1 \\
\end{array} \right)  c_{i + k - 1} +
(-1)^{k}
c_{i + k},
$$
for all $i, k$. 

It is said the sequence $\overline{c}$ is \textit{completely monotonic} if
$\lbrack \Delta^{k} ( \overline{c} ) \rbrack_{i} > 0$
for each $k \in \N_0$ and $i \in \N$.

Recall the following result proved by Hausdorff 
in \cite{Hausdorff-1_1921},
\cite{Hausdorff-2_1921}

\begin{theorem}\label{thm-Hausdorff}
A sequence 
$\overline{c} = \{ c_{1}, c_{2}, c_{3}, c_{4}, \ldots \}$ 
is completely monotonic if and only if there exists a 
positive finite measure $\theta$
on the interval $[0, 1]$ such that, for every $i \in \N$,
$$c_{i} = \int_{0}^{1} x^{i} \theta(dx).
$$ 
Moreover, there is a one-to-one correspondence between the set 
of all completely monotonic 
sequences and the set of all finite positive measures on the 
interval $[0, 1]$.
\end{theorem}

Now, we can use a different characterization of the set 
$M_1(B_\infty)$ of tail invariant probability
measures on the path space of the generalized Bratteli diagram 
$B_\infty$. 

Every measure
$\mu \in M_1(B_\infty)$ is uniquely determined by a sequence 
of positive infinite vectors 
$\overline{p}^{(n, \infty)} =  \langle p^{(n, \infty)}_i : i \in \N\rangle$.
Since $\mu$ is a probability measure, the vector 
$\ol p^{(1, \infty)}$ is probability. 

From (\ref{eq8.12}) we get
\be\label{eq-compl mon}
p^{(n)}_i = p^{(n+1)}_i +p^{(n+1)}_{i+1} + \ldots. 
\ee 

\begin{lemma}\label{lem compl mon}
Every tail invariant probability measure $\mu \in M_1(B_\infty)$ is uniquely determined by a completely 
monotonic sequences  $\overline{p} = \left\langle p_i\ |\  i =1, 2, \ldots \right\rangle$  
in such a way that, $p_{i} = p_{i}^{(1, \infty)}, \ i =1, 2, \ldots$.
For each cylinder set $[\ol e]$ such that 
$r(\ol e) = i \in V_n$, we have 
$\mu([\ol e]) = \lbrack\Delta^{n - 1} (\overline{p}) \rbrack_{i}$,  $n \in \N$.
\end{lemma}

\begin{proof}
Fix some $n$, and show that $\ol p^{(n)}$ is a completely 
monotonic sequence. Indeed, it follows from 
\eqref{eq-compl mon} that $[\De p^{(n)}]_i = p^{(n+1)}_i > 0$.
For $\De^2$, we compute
$$
[\De^2 p^{(n)}]_i = [\De p^{(n)}]_i - [\De p^{(n)}]_{i+1} =
p^{(n)}_i  - 2p^{(n)}_{i+1} + p^{(n)}_{i+2} =
p^{(n+1)}_i - p^{(n +1)}_{i+1} = p^{(n+2)}_i >0.
$$
By induction, we deduce that $\De^k (\ol p^{(n)}) = 
\ol p^{(n+ k)}$, or 
$[\De^k p^{(n)}]_i = p^{(n+k)}_i > 0$ for all natural numbers
$n$ and $k$.  
\end{proof}

It follows from Proposition \ref{prop mu_a} that the measure 
$\mu_{a}$ is determined by the probability completely 
monotonic vectors (sequence)
$${\overline{p}_{a}^{(1, \infty)}} = \left\langle 
\frac{1}{(a + 1)}, \frac{a}{{(a + 1)}^{2}}, 
\frac{a^{2}}{{(a + 1)}^{3}},
\ldots
\right\rangle.
$$ 

\begin{theorem}\label{thm measures for Binfty}
The family of measures $\{ \mu_a : 0 < a< \infty\}$ forms the 
set of all ergodic probability tail invariant measures
on the generalized Bratteli diagram $B_\infty$.
\end{theorem} 

\begin{proof}
Assume that
\be\label{eq-lin comb}
\mu_{a} = \lambda  \mu + (1 - \lambda)  \rho,
\ee
where 
$0 < \lambda < 1$, and $\mu$ and $\rho$ are tail 
invariant probability measures on $X_{B_\infty}$. It follows 
from Lemma \ref{lem compl mon} that 
$\mu$ and $\rho$ are determined by some completely monotonic 
positive vectors $\overline{ p}^{1} = \langle p^{(1)}_1, 
p^{(1)}_2 \ldots \rangle$ and
$\overline{w}^{1} = \langle w^{(1)}_1, w^{(1)}_2, 
\ldots\rangle$, respectively. Then, we have 
$$
\frac{a^{k - 1}}{{(a + 1)}^{k}} = 
\lambda p^{(1)}_k + (1 - \lambda) w^{(1)}_k
$$ 
for each $k = 1, 2, \ldots$.

Use Theorem \ref{thm-Hausdorff} and denote by $\theta_{a}$, $\theta_{\mu}$, and $\theta_{\rho}$ 
the corresponding finite measures on the interval $[0, 1]$ 
defined by the sequences
${\overline{y}}_1(a), \overline{p}^{(1)}$, and
$\overline{w}^{(1)}$.
Then we have 
$$\frac{a^{k - 1}}{{(a + 1)}^{k}} =
\int_{0}^{1} x^{k} \theta_{a}(dx), \ \  
p^{(1)}_k =
\int_{0}^{1} x^{k} \theta_{\mu}(dx), \ \  
w^{(1)}_k =
\int_{0}^{1}x^{k} \theta_{\rho}(dx),
$$ 
and for all $k \in \N$,
$$\int_{0}^{1} x^{k} \theta_{a}(dx) = 
\lambda 
\int_{0}^{1} x^{k} \theta_{\mu}(dx) + 
(1 - \lambda) 
\int_{0}^{1} x^{k} \theta_{\rho}(dx). 
$$
This means that 
$\theta_{a} = \lambda \theta_{\mu} +(1 -\lambda)\theta_{\rho}$. 
To finish the proof, we observe that   
$$\theta_{a} = \frac{1}{a} \delta_{\frac{a}{a + 1}}
$$ 
is a delta-measure concentrated at  $\frac{a}{a + 1}$.
Indeed, 
$$
\int_{0}^{1} x^{k} \theta_{a}(dx) =
\frac{a^{k - 1}}{{(a + 1)}^{k}} = 
\frac{1}{a}  \int_{0}^{1} x^{k} \delta_{\frac{a}{a + 1}}(dx)
$$
for every $k$. Therefore, relation \eqref{eq-lin comb} 
is impossible, and this means that
$\mu_{a}$ is an ergodic measure.    
\end{proof}

\subsection{Subdiagrams of \texorpdfstring{$B_\infty$}{B_\infty}}
Now we present some natural 
subdiagrams of the generalized Bratteli diagram $B_\infty$. 
Let $k$ be a fixed natural number; take the sequence of vertices  
$W_{1} = \{k\},\ W_{2} = \{k,k+1\},\ \ldots, W_{n} = \{k, k+1, \ldots, k+n-1\}, \ldots$. 
We will use this sequence to define two subdiagrams of $B_\infty$.
They are the vertex subdiagram $\ol B(W, k)$ of $B_\infty$, supported
by the sequence $\{W_n\}$ and an edge subdiagram $\ol B_k$ which is defined by the sequence of incidence matrices $\ol F_n = (\ol f^{(n)}_{ij}) $, $n \in \N$, where  
\begin{equation}
\ol f^{(n)}_{ij} =
\begin{cases}
    1, & \textrm{if}\  j= i \ \textrm{or}\  j = i-1 \ \textrm{and} \ 
    i = k, \ldots, k+n -1 \\
    0, & \textrm{otherwise}
\end{cases}
\end{equation}

We first consider the edge subdiagram $\ol B_k$.  It is not hard
to see that, by definition, $\ol B_k$ is isomorphic to the 
\textit{classical Pascal graph} because for every $n$ the vertex 
$j \in W_n$ is the source for exactly two edges connecting $j$ with the vertices $j$ and $j+1$ from $W_{n+1}$.

It is well known (see, e.g. \cite{MelaPetersen2005} or 
Section \ref{sect Bratteli-Pascal}) that 
the path space $X_{\ol B_k}$ of $\ol B_k$ supports uncountably 
many non-atomic ergodic tail invariant measures $\nu_{p}$,
$0 < p < 1$, such that for a cylinder set $[\ol e]$ with 
the source in $k$ and range in $i \in W_{n+1}$, we have
$$
\nu_{p} ([\overline{e}])= p^{n +k-i} (1 - p)^{i - k},
$$
where $i = k, k+1, \ldots, k + n$. 

Our goal is to find out whether the measure extension 
$\wh \nu_p$ on $\wh X_{\ol B_k} \subset X_{B_\infty}$ is finite or
infinite. For this, we use 
criterion (\ref{add3_8}) from Proposition \ref{prop-edge meas ext}. 
Denoting by $F'$ the incidence matrix of $B_\infty$, we observe that the entries $\wt f^{(n)}_{ij}$ of the matrix $\wt F_n = F' - \ol F_n$ 
are identified with the edges that were deleted from $B_\infty$
to produce the Pascal graph. It can be seen that
$\wt f^{(n)}_{ij} =1$ if and only if  $i = k, \ldots k +n -1$ and 
$ 1 \leq j \leq i-2$. 
Hence, it follows from (\ref{add3_8}) that
\be\label{eq=meas ext}
\wh \nu_p(\wh X_{\ol B_k}) < \infty \ \Longleftrightarrow \ 
 I= \sum_{n\geq 1} \sum_{i \in V_{n+1}} \sum_{j \in V_n} 
\wt f^{(n)}_{ij} H_j^{(n)} p^{n+k-i}(1 - p)^{i-k} < \infty, 
\ee
where $H_j^{(n)}$ is the height of the $j$-th tower in $B_\infty$.

\begin{prop}
(1) The extension $\wh \nu_p$ of the measure $\nu_p$ is infinite. 

(2) The path space $\wh X_{\ol B_k}$ has zero measure for any
ergodic tail invariant probability measure on $B_\infty$.
\end{prop}

\begin{proof}
(1) We will show that the series $I$ in \eqref{eq=meas ext} is 
divergent.

We fix some $n$ and, using \eqref{eq=prop of S},
\eqref{eq=H=S}, and \eqref{eq-numbers S}, obtain 
$$
\ba 
I_n := &\   \sum\limits_{i \in  V_{n + 1}}{\sum\limits_{j \in  V_{n}}
{\wt f}_{ij}^{(n)}} \ H_{j}^{(n)} {(1 - p)}^{i - k} p^{n + k - i}\\ 
= &\ p^{n} H_{k}^{(n)} + \sum\limits_{i = k + 1}^{k + n} p^{n + k - i}  {(1 - p)}^{i - k} 
 \sum\limits_{j = 1}^{i - 2} H_{j}^{(n)}   \\
= &\ p^n H_{k}^{(n)} +  \sum\limits_{i = k + 1}^{k + n} p^{n + k - i} (1 - p)^{i - k} \cdot H_{i - 2}^{(n + 1)} \\ 
= &\  p^{n}   
\left( \begin{array}{c}
n + k  - 2 \\
n - 1 \\
\end{array} \right) +
\sum\limits_{i = k + 1}^{k + n} p^{n + k - i} (1 - p)^{i - k} 
\left( \begin{array}{c}
n + i - 3 \\
n \\
\end{array} \right). \\ 
\ea 
$$
Changing the index of summation, we get 
$$
\ba
I_n = &\   p^{n}  
\left( \begin{array}{c}
n + k  - 2 \\
n - 1 \\
\end{array} \right) +
\sum\limits_{i = 0}^{n - 1} p^{n - i - 1} (1 - p)^{i + 1}  
\left( \begin{array}{c}
n + i + k - 2 \\
n \\
\end{array} \right) \\ 
= &\ p^{n} 
\left( \begin{array}{c}
n + k  - 2 \\
n - 1 \\
\end{array} \right) + p^{n} 
\sum\limits_{i = 0}^{n - 1} \left( \frac{1 -p}{p} \right)^{i + 1} 
\left( \begin{array}{c}
n + i + k - 2 \\
n \\
\end{array} \right). \\
\ea 
$$
Then, the measure ${\widehat{\nu}}_{p}$ is infinite if the series
$$
I= \sum\limits_{n = 1}^{\infty} p^{n} \cdot 
\left\lbrack  
\left(\begin{array}{c}
n + k - 2 \\
n - 1 \\
\end{array} \right) +
\sum\limits_{i = 0}^{n - 1} \left( \frac{1 - p}{p} \right)^{i + 1} 
\cdot 
\left(\begin{array}{c}
n + i + k - 2 \\
n  \\
\end{array} \right) 
\right\rbrack
$$
is divergent. For this, it suffices to show that 
 $$
 J= \sum\limits_{n = 1}^{\infty} p^{n} \cdot 
\left\lbrack  
\sum\limits_{i = 0}^{n - 1} \left( \frac{1 - p}{p} \right)^{i + 1} 
\left(\begin{array}{c}
n + i + k - 2 \\
n  \\
\end{array} \right) 
\right\rbrack
= \infty
$$
The series $J$ can be represented in the following form and then 
estimated from below:
$$
\ba 
J = &\ 
\sum\limits_{i = 0}^{\infty} \left( \frac{1 - p}{p} \right)^{i + 1} 
\sum_{n = i + 1}^{\infty}
p^{n} 
\left(\begin{array}{c}
n + i + k - 2 \\
n \\
\end{array} \right)\\
= &\ 
 \sum\limits_{i = 0}^{\infty} (1 - p)^{i + 1}
\cdot 
\sum\limits_{n = 0}^{\infty} p^n \cdot
\left( \begin{array}{c}
n + 2 i + k -1 \\
k + i - 2 \\
\end{array} \right)\\
\geq &\ 
 \sum\limits_{i = 0}^{\infty} (1 - p)^{i + 1}
\cdot 
\sum\limits_{n = 0}^{\infty} p^n \cdot
\left( \begin{array}{c}
n +  i + k -2 \\
k + i - 2 \\
\end{array} \right)
\ea
$$
Next, we use the equality
$$
\sum\limits_{n = 0}^{\infty}\left( \begin{array}{c}
n + s \\
s \\
\end{array} \right)
 \cdot  p^{n} =
\frac{1}{{(1 - p)}^{s + 1}}
$$ 
whenever $|p| < 1$ and $s = 0, 1, 2,\ldots$
Then,  
$$\sum\limits_{n = 0}^{\infty}\left( \begin{array}{c}
n + i + k - 2 \\
i + k - 2 \\
\end{array} \right)
\cdot p^{n} = 
\frac{1}{{(1 - p)}^{i + k-1}}.
$$
Finally, we conclude that 
$$J \geq
\sum\limits_{i = 0}^{\infty} (1 - p)^{i+1} 
\frac{1}{(1 - p)^{i + k-1}} 
= \sum\limits_{i = 0}^{\infty} \frac{1}{{(1 - p)}^{k - 2}} = \infty.
$$
In this way we have proved that the extension ${\widehat{\nu}}_{p}$
of the measure $\nu_{p}$ is infinite. 

(2)
Now we prove that
$\mu_{a}(\widehat{X}_{\overline{B}_{k}}) = 0$ for every
$0 < a < \infty$. For this, it is enough to show that 
$\mu_{a}(X_{\overline{B}_{k}}) = 0$. Let $Y^{(n)}_i$ be the set of all
paths $x = (x_n)$ from $X_{\ol B_\infty}$ such that the finite
path $(x_1,...,x_{n})$ lies in $\ol B_k$ and $s(x_n) =i \in W_{n}$. 
Then 
\be\label{eq=path space sbdg}
X_{\ol B_k} = \bigcap_{n=1}^\infty \bigcup_{i \in W_n} Y^{(n)}_i
\ee 
and
$$
\mu_{a}(\widehat{X}_{\overline{B}_{k}}) = 
\lim\limits_{n \to \infty} \mu_a \left(\bigcup_{i \in W_n} Y^{(n)}_i
\right)
$$
where $\mu_a(Y^{(n)}_i)$ can be found in Proposition \ref{prop mu_a}.
To find the height of every tower $Y^{(n)}_i$, we use also the fact that $\ol B_k$ is the Pascal graph.
Therefore,
$$
\ba 
\mu_{a}(\widehat{X}_{\overline{B}_{k}}) = &\ 
\lim\limits_{n \to \infty} \left( \frac{a^{k}}{{(1 + a)}^{n + k}}
\cdot \left\lbrack 
\left( \begin{array}{c}
n \\
0 \\
\end{array} \right) + 
\left( \begin{array}{c}
n \\
1 \\
\end{array} \right) 
 \left( \frac{a}{a + 1} \right) + \ldots +
\left( \begin{array}{c}
n \\
n \\
\end{array} \right)  
\left( \frac{a}{a + 1} \right)^{n} \right\rbrack \right)\\
= &\ 
\lim\limits_{n \to \infty}
\left\lbrack \frac{a^{k}}{{(1 + a)}^{n + k}} \right\rbrack
\cdot \left( 1 + \frac{a}{a + 1} \right)^{n} \\
= &\ 
{\left\lbrack \frac{a}{a + 1} \right\rbrack}^{k} \cdot  
\lim\limits_{n \to \infty}  \frac{(1 + 2a)^{n}}{{(1 + a)}^{2n}} \\
= &\ 0. 
\ea
$$ 
\end{proof}

\begin{remark}
Let a subdiagram $\ol B$ of $B_\infty$
be defined by the sequence $\ol W_n = \{1, \ldots, k\}$. 
The incidence matrix $\ol F$ is the low triangular $k\times k$ matrix 
such that $\ol f_{ij} = 1$ if $i \leq j$ and zero otherwise. 
Then the path space of $\ol B$ is countable and 
the only tail invariant measure is the delta measure supported by the infinite path going through vertex 1. 
Indeed, the ``normalized'' vectors  
${\overline{y}}_{i}^{(n,m)}$ of the rows of the
matrix ${\ol G'}^{(n,m)} = \ol {F'}^m$ have the form 
$$
{\overline{y}}_{i}^{(n,m)} = \frac{1}{S_{i}^{(m)}} \cdot 
\left\langle S_{i}^{(m - 1)}, S_{i - 1}^{(m - 1)}, \ldots, S_{1}^{(m - 1)}, 0, \ldots, 0 \right\rangle,\ i = 1, \ldots, k.
$$
We can easily show that  
$\overline{y}_{i}^{(n,\infty)} =
\lim_{m \to \infty}{\overline{y}}_{i}^{(n,m)} =
\left\langle  1, 0,  \ldots, 0\right\rangle$ because 
$$\lim_{m \to \infty} \frac{S_{i}^{(m - 1)}}{S_{i}^{(m)}} = 1,
\quad
\lim_{m \to \infty} \frac{S_{i}^{(m - 1)}}{S_{i+1}^{(m)}} = 0,\  i=1,  \ldots, k.$$
\end{remark}

\subsection{Values of the measures \texorpdfstring{$\mu_a$}{\mu_a} on subdiagrams \texorpdfstring{$B(W, k)$}{B(W, k)}}
\label{sub_values_measures_mu_a}
We consider now the vertex subdiagram $\ol B(W, k)$ of $B_{\infty}$
that was defined at the beginning of this subsection. 
Recall that $k$ is fixed and $W_n = \{k,k+1, \ldots, k+n-1\}$. 
For the reader's convenience, we include Figure 3:

\begin{figure}[htb!]\label{fig B(W,k)}
\begin{center}
\includegraphics[scale=0.75]{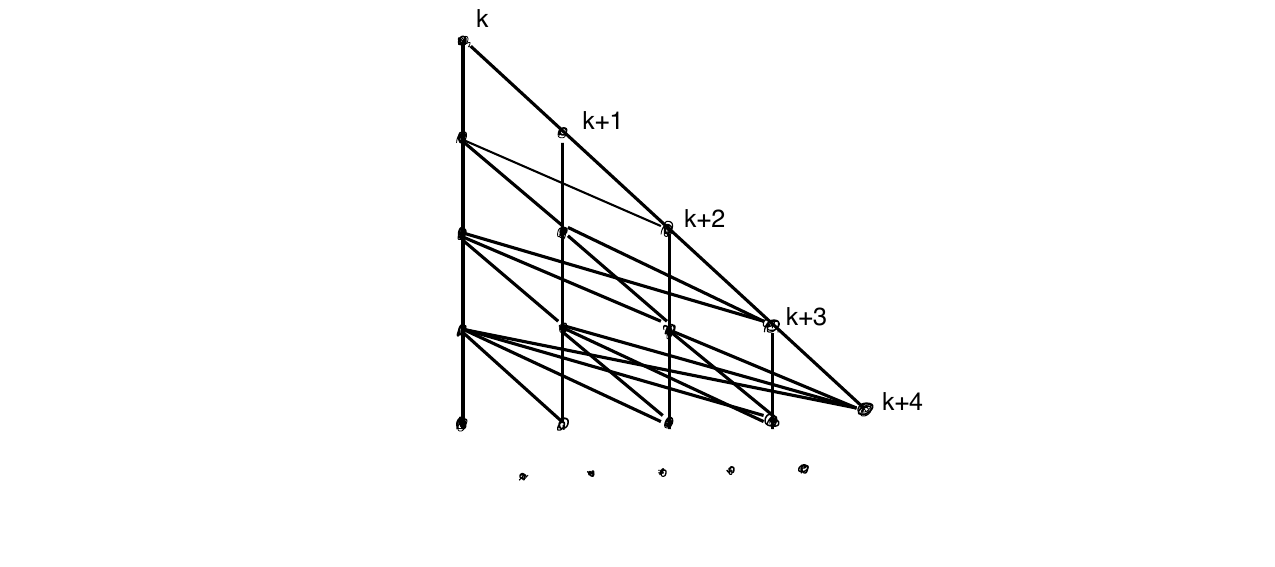}
\caption{The subdiagram $\ol B(W,k)$.}
\end{center}
\end{figure}

Let $h_{i}^{(n)}$ be the heights of the towers 
$\ol X_{i}^{(n)}$, $i = k, k+1, \ldots, k + n - 1$, 
inside $\overline B(W, k)$ where $n \geq 1$.
It is evident that $h_{k}^{(n)} = 1$ for all $n$.
Moreover, from the definition of the subdiagram $\overline B(W, k)$,
we get the relation
 \begin{equation}\label{eq=heights h}
h_{i}^{(n + 1)} =  h_{k}^{(n)} + \ldots + h_{i}^{(n)},
\end{equation} 
 whenever  $i = k, \ldots, k + n - 1$, and  
$h_{k + n}^{(n + 1)} =  h_{k + n - 1}^{(n + 1)}$. 

We recall that $H_{i}^{(n)}$ denotes the height of the tower $X_{i}^{(n)}$ in the diagram $B_{\infty}$, $n, i = 1, 2, \ldots$. 
We know that
$H_{i}^{(n)} = 
\left( \begin{array}{c}
i + n - 2 \\
n - 1 \\
\end{array} \right)$, see \eqref{eq-numbers S}, \eqref{eq=H=S}. 

\begin{lemma}\label{lem heights h}
For $n = 2, 3, \ldots$ 
and $i = k + 1, \ldots, k + n - 1$, the heights of the towers 
$\ol X_{i}^{(n)}$ and $X_{i}^{(n)}$ satisfy the relation
\begin{equation}\label{eq2strars}
h_{i}^{(n)} =
H_{i - k + 1}^{(n)} - H_{i - k}^{(n + 1)}.
\end{equation}

\end{lemma}

\begin{proof} It is obvious that (\ref{eq2strars}) holds 
for $n = 2$ and $i = k + 1$. 
Assume that (\ref{eq2strars}) is true for some $n$. 
Then, using (\ref{eq=heights h}), we get that, for $i = k + 1, \ldots, k + n - 1$,
$$
\ba 
h_{i}^{(n + 1)} = &\ 1 + h_{k + 1}^{(n)} + \ldots + h_{i}^{(n)} \\
= &\ \left\lbrack  1 + H_{2}^{(n)} + \ldots + H_{i - k + 1}^{(n)} \right\rbrack -
\left\lbrack H_{1}^{(n + 1)} + \ldots + H_{i - k}^{(n + 1)} \right\rbrack \\
= &\ H_{i - k + 1}^{(n + 1)} - H_{i - k}^{(n + 2)}.
\ea 
$$ 
It remains to show that (\ref{eq2strars}) holds for $i = k + n$: 
$$
\ba
h_{k + n}^{(n + 1)} =   h_{k + n - 1}^{(n + 1)} = &\ 
H_{n}^{(n + 1)} - H_{n - 1}^{(n + 2)}\\
= &\  
\left( \begin{array}{c}
2n - 1 \\ 
n \\
\end{array} \right) - 
\left( \begin{array}{c}
2n - 1 \\
n + 1 \\
\end{array} \right) \\
 = &\  
\frac{(2n - 1)! \cdot 2}{(n + 1)! \cdot (n - 1)!} = \frac{(2n)!}{(n + 1)! \cdot n!}\\
= &\  
\left( \begin{array}{c}
2n \\
n \\
\end{array} \right) - 
\left( \begin{array}{c}
2n \\
n + 1 \\
\end{array} \right)\\
= &\ 
H_{(n + 1)}^{(n + 1)} - H_{(n)}^{(n + 2)}.
\ea
$$ 
\end{proof}

For $0 < a <\infty$, let $\mu_a$ be the ergodic probability measure
defined in Theorem \ref{thm measures for Binfty}. 
We use \eqref{eq=path space sbdg} to describe the path space 
$\ol X_{\overline{B}_{k}}$ of the subdiagram $\ol B(W, k)$ (we write 
$X_{\overline{B}_{k}}$ instead of $\ol X_{\overline{B}(W, k)}$ for convenience).

\begin{prop}\label{prop pos measure mu_a} 
Let $Z_n:= \bigcup_{i \in W_n} Y^{(n)}_i$ and let 
 $Y_i^{(n)}$ denote the inner tower of the subdiagram $\ol B_k$
 corresponding to the vertex $i \in W_n$. Set  
$\mu_{n, a} = \mu_a(Z_n)$.  Then
\be\label{eq=path meas}
\mu_{a}( \ol {X}_{\overline{B}_{k}}) =
 \lim_{n\to \infty}\mu_{n, a},
\ee
and $\mu_{a}( \ol {X}_{\overline{B}_{k}}) > 0$ if and only if $0< a < 1$.
\end{prop}

\begin{proof}
    
Since the sequence of sets $Z_n:= 
\bigcup_{i \in W_n} Y^{(n)}_i$ is decreasing, we conclude that 
\eqref{eq=path meas} holds. 

Further, we note that $\mu_{1, a} = \dfrac{a^{k-1}}{(a+1)^k}$.
Using the definition of $Z_n$, we can write 
$$
\mu_{n, a} = \frac{1}{{(a + 1)}^{n}}  \left\lbrack 
h_{k}^{(n)}  \left( \frac{a}{a + 1} \right)^{k - 1} + 
h_{k + 1}^{(n)}  \left( \frac{a}{a + 1} \right)^{k} + \ldots +
h_{k + n - 1}^{(n)}  \left( \frac{a}{a + 1} \right)^{k + n - 2}
\right\rbrack.
$$
Denoting $z= \frac{a}{a + 1}$ 
and 
$I_{n} = h_{k}^{(n)} + h_{k + 1}^{(n)} z + \ldots +
h_{k + n - 1}^{(n)} z^{n - 1}$, we can write 
 $$
 \mu_{n, a} = \frac{a^{k - 1}}{{(a + 1)}^{n + k - 1}} I_{n}.
 $$ 

We want to represent $I_{n + 1}$ in terms of $I_{n}$. Note that
the following relation holds:

$$
\ba 
I_{n + 1} = &\ 
\sum_{j =0}^{n} h_{k +j}^{(n + 1)} z^j \\
= &\ \sum_{j=0}^{n-1} \left( \sum_{l=0}^j h^{(n)}_{k+l}  \right) 
z^j  + \left( h_{k}^{(n)}+\ldots + h_{k + n - 1}^{(n)} \right) z^{n}\\
= &\ \sum_{l=0}^{n-1} h^{(n)}_{k+l}z^l (1 + z + \ldots + z^{n-l})\\
= &\ \frac{1}{1-z}\sum_{l=0}^{n-1} h^{(n)}_{k+l}z^l (1 - z^{n-l+1})\\
=& \ \frac{1}{1-z}\left[  \sum_{l=0}^{n-1} h^{(n)}_{k+l}z^l - z^{n+1}  
\sum_{l=0}^{n-1} h^{(n)}_{k+l}\right] \\
= &\ \frac{1}{1-z}\left[ I_n - z^{n+1} h^{(n+1)}_{k+n-1}\right]. \\
\ea
$$

Because $\frac{1}{1 - z} = a + 1$, we obtain that 
$$
\ba
\mu_{n + 1, a} = & \ 
\frac{a^{k - 1}}{{(a + 1)}^{n + k}} \cdot I_{n + 1} \\
= &\ \frac{a^{k - 1} }{{(a + 1)}^{n + k -1}} \cdot
\left\lbrack I_{n} - {\left( \frac{a}{a + 1} \right)}^{n + 1}
h_{k + n - 1}^{(n + 1)}
\right\rbrack\\
=&\ \mu_{n, a} - \frac{a^{k + n}}{{(a + 1)}^{2n + k}} 
h_{k + n - 1}^{(n + 1)}\\
\ea
$$

It follows from  (\ref{eq2strars}) (see also 
\eqref{eq-numbers S} and \eqref{eq=H=S}) that
$$
\ba
h_{k + n - 1}^{(n + 1)} = & \ 
H_{n}^{(n + 1)} - H_{n - 1}^{(n + 2)}\\ 
= &\ \left( 
\begin{array}{c}
2n - 1 \\
n \\
\end{array} 
\right) 
- 
\left( 
\begin{array}{c}
2n - 1 \\
n + 1 \\
\end{array} 
\right)\\
=&\ \frac{(2n - 1)! \cdot 2}{(n + 1)! \cdot (n - 1)!}\\
=& \ \frac{(2n)!}{(n + 1)! \cdot n!}\\
= &\ 
\frac{1}{n + 1} 
\left( 
\begin{array}{c}
2n \\
n \\
\end{array} 
\right)\\
\ea 
$$
Thus, we obtain  that, for all $n \in \N$,
$$
\ba 
\mu_{n, a} - \mu_{n + 1, a} = &\ 
\frac{a^{k + n}}{{(a + 1)}^{2n + k}} \frac{1}{(n + 1)} 
\left( 
\begin{array}{c}
2n \\
n \\
\end{array} 
\right) \\
=&\ \left( \frac{a}{a + 1} \right)^{k} \frac{1}{n + 1} 
\left( 
\begin{array}{c}
2n \\
n \\
\end{array} 
\right) 
\left( \frac{a}{{(a + 1)}^{2}} \right)^{n}
\ea
$$
It follows from this relation that
$$
\mu_{1, a} - \mu_{m + 1, a} = 
\left( \frac{a}{a + 1} \right)^{k} \sum_{n = 1}^{m} \frac{1}{n + 1} 
\left( \begin{array}{c}
2n \\
n \\
\end{array} \right)
 \left( \frac{a}{{(a + 1)}^{2}} \right)^{n} 
$$

Recall that $C_n = \dfrac{1}{n+1} \left( \begin{array}{c}
2n \\
n \\
\end{array} \right)$ is called the Catalan number, and the generating function for the corresponding series is well known, see e.g. \cite{Stanley2015}
\be\label{eq=Catalan}
C(x) =  \sum_{n = 0}^{\infty}  \frac{1}{n + 1} 
\left( \begin{array}{c}
2n \\
n \\
\end{array} \right) x^n = \frac{1 - \sqrt{1 - 4x}}{2x}, \quad |x| < \frac{1}{4} .
\ee 
Using this fact, we can substitute $x = \dfrac{a}{{(a + 1)}^{2}}$ in \eqref{eq=Catalan} and find that 
$$
\ba 
\mu_{a}(\widehat{X}_{{\overline{B}}_{k}}) = &\  
\lim\limits_{m \to \infty} \mu_{m + 1, a} \\
= &\ \mu_{1,a} - \left( \frac{a}{a + 1} \right)^{k} 
\sum_{n = 1}^{\infty}  
\frac{1}{n + 1} 
\left( \begin{array}{c}
2n \\
n \\
\end{array} \right)
 \left( \frac{a}{{(a + 1)}^{2}} \right) ^{n} \\
 =&\ \left( \frac{a}{a + 1} \right)^{k} 
\left\lbrack\ \frac{1}{a} - \sum_{n = 1}^{\infty} 
 \frac{1}{n + 1}  
\left( \begin{array}{c}
2n \\
n \\
\end{array} \right)\
\left( \frac{a}{{(a + 1)}^{2}} \right)^{n} 
\right\rbrack\\
= &\ \left( \frac{a}{a + 1} \right)^{k} 
\left\lbrack\ \frac{1}{a} - C\left(\frac{a}{{(a + 1)}^{2}}\right) +1
\right\rbrack.
\ea
$$
By a direct computation, we obtain that 
$$
C\left(\frac{a}{{(a + 1)}^{2}}\right) = 
\frac{(a+1)[(a +1) - |a-1|]}{2a} = 
\begin{cases}
    \dfrac{a+1}{a} \ & \ \mbox{if} \ a \geq 1\\
    a +1  \ & \ \mbox{if} \ 0< a <1. 
\end{cases} 
$$
Finally, 
$$
\mu_{a}(\widehat{X}_{{\overline{B}}_{k}})  = \begin{cases}
    0, \ &\ \mbox{if} \ a \geq 1\\
    
    \left(\dfrac{a}{a + 1}\right)^{k-1}(1-a),\ & \ \mbox{if} \ 0< a <1. 
\end{cases}
$$
\end{proof}

\begin{remark}
We consider the generalized Bratteli diagram $B_\infty$ and the 
tail invariant measure $\mu_a$ on the path space $X_{B_\infty}$.
Recall that the measure $\mu_a$ defines the probability distribution
on the vertices $i\in V_1$ of the first level: $\mu_a(\{i\}) = \dfrac{a^{i-1}}{(1+a)^i}, i = 1, 2, \ldots$. Define a Markov kernel
$P = (p_{ij})$ by setting 
$$
p_{ij} = \frac{a^{j-i}}{(1+a)^{j-i+1}}, \qquad j \geq i,
$$
and $p_{ij} = 0$ if $j< i$. 

Hence, $\mathrm{Prob}(i \to i) = \frac{1}{1 +a}$ and 
$$
\mathrm{Prob}(i \to \{j | j > i\}) = \sum_{j >i} p_{ij} = \frac{a}
{1 +a}.
$$

It follows from Proposition \ref{prop pos measure mu_a} that
the following result holds.
\end{remark}

\begin{corol}

The following are equivalent:

(i) $\mathrm{Prob}(i \to i) > \mathrm{Prob}(i \to \{j\ |\ j > i\})$,

(ii) $\mu_{a}(\widehat{X}_{{\overline{B}(W, i)}}) > 0$,

(iii) $0 < a <1$.
\end{corol}

\subsection{Tail invariant probability measures on the subdiagram 
\texorpdfstring{$B(W, k)$}{B(W, k)}}

We consider again the vertex subdiagrams $B(W, k)$, $k\in \N,$
defined in subsection \ref{sub_values_measures_mu_a}.
We recall that the heights of the 
towers in this 
subdiagram have been found in Lemma \ref{lem heights h} by the formula:
\be\label{eq=heights}
h_{i}^{(n)} =  H_{i - k + 1}^{(n)} - H_{i - k}^{(n + 1)}, \ 
i = k + 1, \ldots, k + n - 1,
\ee
where 
$$
H_{i}^{(n)} = 
\left( 
\begin{array}{c}
i + n - 2 \\
n - 1 \\
\end{array} 
\right), 
\quad n, i = 1, 2, \ldots.
$$

Let the matrices ${G'}^{(n, m)}= \{ {g'}_{ij}^{(n, m)} \}$ ($m \geq 1$) 
be defined as above, that is they are the product of the incidence matrices
between the levels $V_{n}$ and $V_{n+m}$, 
and let $G^{(n, m)}= \{ g_{ij}^{(n, m)} \} $ denote the corresponding stochastic matrices, where 
$i = k,  \ldots, k + n + m - 1,\ j = k,  \ldots, k + n - 1$ and 
$n, m = 1, 2, \ldots$.

\begin{lemma} The following formula holds:
\be\label{eq=entries}
{g'}_{ij}^{(n, m)} = \begin{dcases}
S_{i - j + 1}^{(m - 1)}, &  i = k,  \ldots, k + n - 1,
\quad j = k, \ldots, i,\\
S_{i - j + 1}^{(m - 1)}, & i = k + n,\quad j = k, \ldots, k + n - 1,\\
S_{i - j + 1}^{(m - 1)} - S_{m + n + k - j + 1}^{(i - k - n - 1)}, & i = k + n + 1, \ldots, k + n + m - 1,\quad j = k, \ldots, k + n - 1.
\end{dcases}
\ee
For $i= n + m +k$ and $j=k, k+1, \ldots, k + n -1$, we have the equality 
${g'}_{n+m+k,j}^{(n, m)} = {g'}_{n+m+k-1,j}^{(n, m)}$.
\end{lemma}

\begin{proof}
For $m = 1$, ${G'}^{(n,1)} = F_n$ is the matrix of $(k +n) \times (k+n-1)$
size, and therefore we can write 
$$
{G'}^{(n,1)} = 
\left\lbrack 
\begin{array}{rrrrrrrrrrrrr}
1 & 0 & 0 & 0  & \ldots &\ldots &\ldots   & 0 \\
1 & 1 & 0      &  0    & \ldots    & \ldots     & \ldots    & 0\\
1 & 1 & 1      &  0          & \ldots    & \ldots  &\ldots   & 0\\
\ldots &\ldots & \ldots   &\ldots     &\ldots & \ldots & \ldots & \ldots \\
1 & 1 & 1      & 1        &\ldots   &\ldots &\ldots & 1 \\
1 & 1 & 1      & 1      & \ldots    &\ldots &\ldots & 1\\
\end{array} 
\right\rbrack,
$$
where the entries of the last two rows equal 1. The rows are indexed by $i = k,
k + 1, \ldots, k + n$, and the columns are indexed by $j = k,  \ldots, 
k + n - 1$.  

Writing the rows of the matrices ${G'}^{(n, m)}$ and ${G'}^{(n, m + 1)}$
as vectors, we obtain the following relations:
$$
{\overline{g}'}_{s}^{(n,m + 1)} =
\sum_{i = 1}^{s}{\overline{g}'}_{i}^{(n,m)}, \quad s = k,  \ldots, 
k + n + m - 1,
$$
and 
\begin{equation}\label{1star}
{\overline{g}'}_{k + n + m}^{(n, m + 1)} = {\overline{g}'}_{k + n + m - 1}^{(n, m + 1)},\quad  
n, m = 1, 2, \ldots.
\end{equation}
Using these relations and the induction assumption, we get that for
$s = k,  \ldots , k + n + m - 1, \ j = k,  \ldots, k + n - 1$ 
$$
{g'}_{sj}^{(n,m + 1)} =  
\sum_{i = j}^{s}{g'}_{i j}^{(n, m)} =
\sum_{i = j}^{s} S_{i - j + 1}^{(m - 1)}
$$ 
if $s = k,  \ldots, k + n$, and
$$
{g'}_{sj}^{(n, m + 1)} = 
\sum_{i = j}^{s} S_{i - j + 1}^{(m - 1)} - 
\sum_{i = k + n + 1}^{s} S_{m + n + k - j + 1}^{(i - k - n - 1)},
$$
if $s = k + n + 1, \ldots, k + n + m - 1$.
Further, we use 
$$
S_{l}^{(0)} + S_{l}^{(1)}  + \ldots + S_{l}^{(u)} = S_{l + 1}^{(u)}.
$$
and get that 
$$
\sum_{i = j}^{s}S_{i - j + 1}^{(m - 1)} = 
\sum_{i = 1}^{s - j + 1}S_{i}^{(m - 1)} = 
S_{s - j + 1}^{(m)},\quad s = k,  \ldots, k + n + m - 1
$$
and
$$
\sum_{i = k + n + 1}^{s} S_{m + n + k - j + 1}^{(i - k - n - 1)} = 
\sum_{u = 0}^{(s - k - n - 1)} S_{m + n + k - j + 1}^{(u)} 
= S_{m + n + k - j + 2}^{(s - k - n - 1)}.
$$
Therefore, we proved that 
$$
{g'}_{sj}^{(n, m + 1)} = S_{s - j + 1}^{(m)}, \quad
s = k,  \ldots, k + n,
$$
and
$$
{g'}_{sj}^{(n, m + 1)} = S_{s - j + 1}^{(m)} - 
S_{m + n + k - j + 2}^{(s - k - n - 1)}, \quad s = k + n + 1, \ldots, 
k + n + m - 1.
$$
relation \eqref{eq=entries} holds. 
\end{proof}

We will now transform formulas \eqref{eq=heights} and 
\eqref{eq=entries} and write them in a more convenient form. For this,
\begin{equation}\label{eq=new form of h}
\ba 
h_{i}^{(n)} = &\ 
\left( 
\begin{array}{c}
n + i - k - 1 \\
n - 1 \\
\end{array} 
\right) -
\left( 
\begin{array}{c}
n + i - k - 1 \\
n \\
\end{array} 
\right) \\
= &\ 
\left( 
\begin{array}{c}
n + i - k - 1 \\
n - 1 \\
\end{array} 
\right) 
\cdot 
\left\lbrack  
1 - \frac{(n + i - k - 1)!}{n! \cdot (i - k - 1)!}
\cdot 
\frac{(n - 1)! \cdot (i - k)!}{(n + i - k - 1)!} 
\right\rbrack \\
= &\ 
\left( 
\begin{array}{c}
n + i - k - 1 \\
n - 1 \\
\end{array} 
\right) 
\cdot  
\left\lbrack  
1 - \frac{i - k}{n} 
\right\rbrack \\
= &\ 
H_{i - k + 1}^{(n)} \cdot  
\left\lbrack  1 - \frac{i - k}{n} \right\rbrack .
\ea
\end{equation}

Similarly, we calculate for $i = k + n + 1, \ldots, k + n + m -1$, 
$j = k, \ldots, k + n - 1$ 
$$
{g'}_{ij}^{(n, m)} =
S_{i - j + 1}^{(m - 1)} - S_{(m + n + k - j + 1)}^{(i - k - n - 1)} 
= S_{i - j + 1}^{(m - 1)} \cdot  
\left \lbrack  
1 - \frac{(i - k - n) \cdot \ldots \cdot (i - j)}{m \cdot \ldots \cdot 
(m + n + k - j)} 
\right\rbrack.
$$
Setting 
$$
r_{ij} = \frac{(i - k - n) \cdot \ldots \cdot (i - j)}{m \cdot \ldots \cdot (m + n + k - j)}, \quad i = k + n + 1, \ldots, k + n + m -1,\  j = k, \ldots, k + n - 1,
$$
and  $r_{ij} = 0$ for  $i = k,  \ldots, k + n - 1$, 
$j = k, \ldots, i$, 
or for $i = k + n$, $j = k, \ldots, k + n - 1$, we obtain that  
\begin{equation}\label{eq=2a}
{g'}_{ij}^{(n, m)} =
S_{i - j + 1}^{(m - 1)} \cdot \left\lbrack  1 - r_{ij} \right\rbrack,
\end{equation}

\begin{prop}\label{prop meas nu_a}
The standard Bratteli diagram $\ol B(W, k)$ admits uncountable 
many  tail invariant probability measures $\nu_a$, $0 \leq a \leq 1$. 
\end{prop}

\begin{proof}
We first calculate the entries $g_{ij}^{(n, m)}$ 
of the stochastic matrices 
$G^{(n, m)}$. Then we find the probability vectors  
${\overline{q}}^{(n, \infty, k)} =
\left\langle q_{j}^{(n, \infty, k)} \ | \ j = k,  \ldots, k + n - 1 \right\rangle$ 
as the limit points of vectors ${\overline{g}}_{i}^{(n, m)}= 
\left\langle g_{ij}^{(n, m)} \ |\ j =k,  \ldots, k + n - 1 \right\rangle$ 
as $m \to \infty$.

Using \eqref{eq_g(n,m) and g'(n,m) e21}, we have 
$$
\ba 
g_{ij}^{(n, m)} = &\  {g'}_{ij}^{(n, m)}
\cdot 
\frac{h_{j}^{(n)}}{h_{i}^{(n + m)}}  \\
= &\
S_{i - j + 1}^{(m - 1)} \cdot 
\left\lbrack  1  -  r_{ij} \right\rbrack 
\cdot 
\frac{h_{j}^{(n)} }
{ H_{i - k + 1}^{(n + m)} \cdot \left\lbrack 1 - \frac{i - k}{n + m} \right\rbrack } =\\
= &\ 
\left\lbrack 
S_{i - j + 1}^{(m - 1)} \cdot \frac{1}{H_{i}^{(n + m)}}
\right\rbrack 
\cdot 
\left( 1 - r_{ij} \right)
\cdot
\frac{H_{i}^{(n + m)}}{H_{i - k + 1}^{(n + m)}}
\cdot 
h_j^{(n)}
\cdot 
\left\lbrack
\frac{n + m}{n + m + k - i}
\right\rbrack
\ea .$$

Further, we calculate
$$\frac{H_{i}^{(n + m)}}{H_{i - k + 1}^{(n + m)}} = 
\left( 
\begin{array}{c}
n + m + i - 2 \\
n + m - 1 \\
\end{array} 
\right) 
\cdot
\left( 
\begin{array}{c}
n + m + i - k - 1 \\
n + m - 1 \\
\end{array} 
\right)^{-1} =$$
$$
\frac{(n +  m + i - 2)!}{(n + m - 1)! \cdot (i - 1)!}
\cdot  
\frac{(n + m - 1)! \cdot (i - k)!}{(n + m + i - k - 1)!} = 
\frac{(n + m + i - k) \cdot \ldots \cdot (n + m + i - 2)}{(i - k + 1) \cdot \ldots \cdot (i - 1)} 
$$

Now assume that $m \to \infty$ and $\dfrac{i_{m}}{m} \to a$. 
We remark that $0 \leq a \leq 1$ because
$k \leq {i}_{m} \leq m + n + k$. In what follows, we assume that
$ 0 < a <1$. 

We compute successively (see the proofs of Lemma \ref{lem:H(n,a)} and 
Proposition \ref{prop mu_a}):
$$
p_{j}^{(n, \infty)} = 
{\lim\limits_{m \to \infty} \left\lbrack S_{i - j + 1}^{(m - 1)} \cdot \frac{1}{H_{i}^{(n + m)}} \right\rbrack} =  
\frac{a^{j - 1}}{{(1 + a)}^{n + j - 1}},
$$ 
$$\lim\limits_{m \to \infty} {\frac{H_{i}^{(n + m)}}{H_{i - k + 1}^{(n + m)}}} =
\lim\limits_{m \to \infty} \frac{(n + m - i - k) \cdot \ldots \cdot 
(n + m + i - 2)}{(i - k + 1) 
\cdot \ldots \cdot (i - 1)} = 
\left(\frac{1 + a}{a } \right)^{k - 1},
$$
$$
{\lim\limits_{m \to \infty}}{(1 - r_{ij}}) =
\lim\limits_{m \to \infty} 
\left\lbrack 
1 - \frac{(i - k - n) \cdot \ldots \cdot (i - j)}{m \cdot \ldots \cdot (m + n + k - j)} 
\right\rbrack 
= 
\left\lbrack 1 - a^{n + k - j + 1} \right\rbrack,$$
and
$$\lim\limits_{m \to \infty}
\frac{n + m}{n + m + k -i} = \frac{1}{1 - a}.
$$

As a consequence of the above formulas, we have
\be\label{eq= p^n}
p_{j}^{(n, \infty, k)} =  
\frac{a^{j - k}}{{(1 + a)}^{n + j - k}}
\cdot 
\frac{\left\lbrack 1 - a^{n + k - j + 1} \right\rbrack}{1 - a},
\ee
and 
$$q_{j}^{(n, \infty, k)} =  h_j^{(n)} p_{j}^{(n, \infty, k)}
$$
for $j = k, k + 1, \ldots, k + n - 1$.

It remains to check that the sequence of vectors
$$
{\overline{p}}^{(n,\infty,k)} = \left\langle p_{j}^{(n,\infty,k)}
\ |\ j = k, \ldots, k + n - 1 \right\rangle
$$
satisfies relation (\ref{eq:formula_p_n}) of Theorem \ref{BKMS_measures=invlimits}. 
Recall that ${F'}_{n} = {G'}^{(n,1)}$.

We should verify that, for every $l = k, \ldots, k + n - 1$,
the relation 
$$\sum\limits_{j = l}^{k + n} p_{j}^{(n + 1, \infty, k)}= p_{l}^{(n, \infty, k)}$$ 
holds.  For this, 
$$\sum\limits_{j = l}^{k + n} p_{j}^{(n + 1, \infty, k)} =
\frac{1}{{(1 + a)}^{n + 1} (1 - a)} 
\sum\limits_{j = l}^{k + n} \frac{a^{j - k}}{{(1 + a)}^{j - k}} 
(1 - a^{n + k - j + 2} )
$$
$$
= \frac{1}{{(1 + a)}^{n + 1} (1 - a)} \left\lbrack
\left( \frac{a}{1 + a} \right)^{l - k} 
\sum\limits_{j = 0}^{n + k - l} 
\left( \frac{a}{1 + a} \right)^{j} - a^{n + 2} 
\left(\frac{1}{1 + a} \right)^{l - k} 
\sum\limits_{j = 0}^{n + k - l} \left( \frac{1}{1 + a} \right)^{j}
\right\rbrack
$$
$$= \frac{1}{ (1 + a)^{n + 1} \cdot (1 - a)} \times 
$$
$$
\times \left\{
(1+a)   \left( \frac{a}{1 + a} \right)^{l - k} 
\left\lbrack
1 - \left( \frac{a}{1 + a} \right)^{n + k + 1 - l} 
\right\rbrack
- \left( \frac{1 + a}{a} \right) \cdot 
\frac{a^{n + 2}}{{(1 + a)}^{l - k}} \cdot 
\left\lbrack 
1 - \left( \frac{1}{1 + a} \right)^{n + k + 1 - l} 
\right\rbrack\right\}
$$
$$
= \frac{1}{{(1 + a)}^{n + k - l}  (1 - a)} \cdot   
\left\lbrack
a^{l - k} - \frac{a^{n + 1}}{ {(1 + a)}^{n + k + 1 - l} } -
{a}^{n + 1} + \frac{ a^{n + 1} }{ {(1 + a)}^{n + k + 1 - l} }
\right\rbrack 
$$ 
$$ 
= \frac{a^{l - k}}{(1 + a)^{n + k - l} (1 - a)}
\cdot   \left[ 1 - a^{n + k - l + 1}  \right] = 
p_{l}^{(n, \infty, k)}
$$
as follows from \eqref{eq= p^n}. 

We have proved that the sequence of vectors
$\{ {\overline{p}}^{(n, \infty, k)},\ n = 1, 2, \ldots \}$ 
determines an invariant probability measure $\nu_{a}$ on the subdiagram
${\overline{B}}(W, k)$ for every $0 \leq a \leq 1$. 
\end{proof}

We discuss the ergodicity of $\nu_a$ in the next proposition.

\begin{lemma}\label{lem p defines nu}
Let $\nu$ be a tail invariant probability measure on the
subdiagram $\overline{B}(W,k)$ and let
$$
{\overline{p}}^{(n, \infty, k)} =  \left\langle p_{j}^{(n, \infty, k)}  \ | \ k \leq j \leq k+n-1 \right\rangle
$$ 
be the sequence of vectors defining the measure $\nu$ as in Theorem 
\ref{BKMS_measures=invlimits}. 
Then the measure $\nu$ is completely determined by a sequence
of numbers $\{ p_{n},\ n =1, 2, \ldots \}$, where 
$p_{n} = p_{k + n - 1}^{(n, \infty, k)}$.
\end{lemma}

\begin{proof} (Sketch)
We know that, for every $n = 1, 2, \ldots$ and for $l = k, k+1, 
\ldots, k+n-1$, the following relation holds
\begin{equation}\label{eq=p_n}
\sum\limits_{j = l}^{k + n} p_{j}^{(n + 1, \infty, k)} = p_{l}^{(n, \infty, k)}.
\end{equation}
It follows directly from \eqref{eq=p_n} that 
$p_{k + n - 1}^{(n + 1, \infty, k)} = p_{k + n - 1}^{(n, \infty, k)} - 
p_{k + n}^{(n + 1, \infty, k)} = p_{n} - p_{n + 1}$ where 
$n \in \N$. Using \eqref{eq=p_n} and beginning with the first level, we consequently find that 
$p_k^{(k,\infty, 1)} = p_1$, $p_k^{(k,\infty, 2)} = p_1 - p_2$, 
$p_{k+1}^{(k+1,\infty, 2)} = p_2$, and so on. In particular, 
the vector $\ol p^{(k+1,\infty, 4)}$ has the entries 
$\langle p_1 - 3p_2 + p_3, p_2 - 3p_3, p_3 - p_4, p_4 \rangle$.

This computation shows that, for every vector 
$\ol p^{(n, \infty, k)}$, all coordinates $p_j^{(n, \infty, k)}$ 
are represented as a linear combination of the terms of the 
sequence $\{p_n\}$. Moreover, to write  $p_j^{(n, \infty, k)}$ 
as a linear combination, we use only numbers $p_1, \ldots, p_n$.
\end{proof}

\begin{lemma}\label{lem compl mon 1}
For $\nu$ as in Lemma \ref{lem p defines nu}, the sequence of numbers $\{p_{n}\}$ where 
$p_n = p_{k + n - 1}^{(n, \infty, k)}$ is completely monotonic. 
\end{lemma}

\begin{proof}
Now we prove that for every tail invariant probability measure 
$\nu$ on the subdiagram ${\overline{B}}_{k}$, the sequence 
$\{ p_{n}\ |\ n = 1, 2, \ldots \}$ 
defined above is completely monotonic. 
 Let $\nu = \nu_a$.

By \eqref{eq= p^n}, we have 
$$p_n(a)= p_{n} = p^{(n,\infty, k)}_{k+n-1} =
\frac{a^{n - 1}}{{(1 + a)}^{2n - 1}} \frac{(1 - a^{2})}{(1 - a)} =
\frac{a^{n - 1}}{{(1 + a)}^{2n - 2}}.
$$
Therefore, 
$$
(\Delta p)_{n} = p_{n} - p_{n + 1} = 
\frac{a^{n - 1}}{{(1 + a)}^{2n - 2}}- \frac{a^{n}}{{(1 + a)}^{2n}} = 
\frac{a^{n - 1} (1 + a + a^{2})}{ {(1 + a)}^{2n} },$$
Similarly, we find
$$
 ( \Delta^{2} p)_{n} =
\frac{a^{n - 1} \cdot  \left(1 + a + a^{2} \right)^{2} }{ {(1 + a)}^{2n + 2} },
$$
and, in general,
$$
 (\Delta^{l} p)_{n} =
\frac{a^{n - 1}  {(1 + a + a^{2})}^{l}}{ {(1 + a)}^{2n + 2l - 2}},
\quad 
n = 1, 2, \ldots,\ l = 1, 2, \ldots.
$$

We have proved that $p_n(a)$   is a completely monotonic sequence. 
Now take any invariant probability measure  $\nu$  on
$\overline{B}(W,k)$. Because  the set $M_1 (X_{\ol B(W,k)})$ of all ergodic measures  on $\overline{B}(W,k)$  is contained in  $\{ \nu_a : 0 < a <1 \}$, 
then  $\nu$  is an integral over the measures $\nu_a$,  i.e., 
for every continuous function  $f \colon X_{\overline{B}(W,k)} \to \mathbb{R}$,
$$
\int\limits_{X_{\overline{B}(W,k)}}  f(x) \; d\nu(x)  =  \int\limits_{0}^{1}\left[ \int\limits_{X_{\overline{B}(W,k)}} f(x) \; d\nu_a(x)\right] d\rho(a),
$$  
where  $\rho$ is a probability measure on $(0, 1)$.  

In particular, 
$p_n (\nu)= \int\limits_{0}^{1}  p_n (a)\; d\rho(a)$, $n \in \N$.
This equality implies that the following relations hold
$$
(\Delta^l p_n(\nu))_n = \int\limits_{0}^{1} (\Delta^l p(a))_n  \rho(da)$$ 
for $n, l =1, 2, \ldots$. This means that the
sequence $\{ p_n (\nu) \}$  is completely monotonic. 
\end{proof}

\begin{prop} \label{prop nu_a ergodc}
Every measure $\nu_a$ defined in Proposition \ref{prop meas nu_a}   
is ergodic. 
\end{prop}

\begin{proof}  We will show that every measure $\nu_a$ 
described in Proposition \ref{prop meas nu_a} is ergodic.
The proof is similar to the proof of Theorem \ref{thm measures for Binfty}.
Each invariant probability measure $\nu$ on
${\overline{B}}(W, k)$ determines a completely monotonic sequence
$\{p_n(\nu) | \ n = 1, 2, \ldots \}$ as shown in Lemmas \ref{lem p defines nu} and \ref{lem compl mon 1}. By Hausdorff theorem 
(Theorem \ref{thm-Hausdorff}) there exists 
a unique probability measure $m$ on the interval $[0, 1]$ such
that 
$\int\limits_{0}^{1} x^{n} m(dx) = p_n(\nu)$, $n \in \N$. For 
$\nu = \nu_a$, we have $m = m_a$ and
$$
\int\limits_{0}^{1} x^{n} m_{a}(dx) =
\frac{a^{(n - 1)}}{ (1 + a)^{2n - 2}} =
\frac{{(1 + a)}^{2}}{a} \cdot 
\int\limits_{0}^{1} x^{n} \delta_{\frac{a}{{(1 + a)}^{2}}} (dx)
$$ 
where $n \in \N$. This implies that
$$
m_{a} = \frac{{(1 + a)}^{2}}{a} \cdot \delta_{\frac{a}{{(1 + a)}^{2}}}.$$

Now, we use the same arguments as in Theorem \ref{thm measures for Binfty} to prove that
$\nu_{a}$ cannot be represented as a linear convex combination  
$\nu_{a} = \lambda  \nu^{(1)} + (1- \lambda)  \nu^{(2)}$, 
where $0 < \lambda < 1$ and 
$\nu^{(1)}$ and $\nu^{(2)}$ are different invariant finite measures 
on ${\overline{B}}_{k}$. Therefore each measure $\nu_{a}$ 
is an ergodic measure.
\end{proof}

\begin{remark}
We observe that for $a = 1$ the measure $\nu_1$ is infinite. 
Indeed, it follows from \eqref{eq= p^n} that, for $a = 1$, 
$$
\ol p^{(n, \infty, k)} = \left\langle \frac{n+1}{2^n}, 
\frac{n}{2^{n - 1}}, \ldots, 1 \right\rangle
$$
and therefore 
$\ol q^{(n, \infty, k)}$ has the following form:
$$
\ol q^{(n, \infty, k)} =  \left\langle h_{k}^{(n)}
\frac{n + 1}{2^{n}}, \ h_{k+1}^{(n)} \frac{n}{2^{n - 1}}, \ldots, h_{k+n -1}^{(n)} \right\rangle.
$$
Since $h^{(n)}_j \to \infty$ as $n \to \infty$, we conclude that 
$\nu_1$ is infinite.
\end{remark}

\begin{prop}
The extension $\widehat{\nu}_{a}$ of each measure 
$\nu_a$ onto the set $\mc R(X_{\ol B(W, k)})$ is finite, $0< a <1$. Moreover, the measures 
$\widehat{\nu}_{a}$ are pairwise singular, $0 < a <1$. 
\end{prop}

\begin{proof}
The proof follows from condition (ii) of Theorem \ref{TheoremIV_1}. 
Indeed, for $k \geq 2$, we have
$$
\sum_{n = 1}^{\infty} \left\lbrack \sum\limits_{i = k}^{k + n}
\sum\limits_{j = 1}^{k - 1} {f'}_{ij}^{(n)}  H_{j}^{(n)} 
p_{i}^{(n + 1, \infty, k)} + H_{k + n}^{(n)} \ 
p_{k + n}^{(n + 1, \infty, k)} \right\rbrack
$$
$$
= \sum\limits_{n = 1}^{\infty}\left\lbrack
\sum\limits_{i = k}^{k + n} p_{i}^{(n + 1, \infty, k)} \cdot 
\left( \sum\limits_{j = 1}^{k - 1}
H_{j}^{(n)}\right) + H_{k + n}^{(n)} 
p_{k + n}^{(n + 1, \infty, k)} \right\rbrack
$$
$$
= \sum\limits_{n = 1}^{\infty} \left\lbrack 
\frac{1}{{(1 + a)}^{n}} 
\sum\limits_{i = k}^{k + n}
\frac{a^{i - k}}{{(1 + a)}^{i - k}} \cdot 
\frac{\left(1 - a^{n + k - i + 2} \right)}{(1 - a)} 
\sum\limits_{j = 1}^{k - 1} 
H_{j}^{(n)} + H_{k + n}^{(n)} 
\frac{a^{n}}{(1 + a)^{2n}} %(1 - a^{2})}{{(1 + a)}^{2n}(1 - a)}
\right\rbrack 
$$
$$
=\sum\limits_{n = 1}^{\infty}\frac{1}{{(1 + a)}^{n}} \cdot
\sum\limits_{i = 0}^{n}
\left(\frac{a}{1 + a} \right)^{i} \cdot \frac{\left(1 - a^{n -i + 2} \right) }{(1 - a)} \cdot 
\sum\limits_{j = 1}^{k - 1} H_{j}^{(n)} +
\sum\limits_{n = 1}^{\infty} H_{k + n}^{(n)}  \frac{a^n}
{(1 + a)^{2n}}
$$
$$
< \frac{k-1}{1-a}\cdot 
\sum\limits_{n = 1}^{\infty} {H_{k - 1}^{(n)}} \frac{1}{(1 + a)^{n}} \ 
\sum\limits_{i = 0}^{n} \left(\frac{a}{1 + a} \right)^{i} +
 \sum\limits_{n = 1}^{\infty}
{H_{k + n}^{(n)} \ \frac{a^{n}}{{(1 + a)}^{2n}}}
$$
$$
< \frac{k - 1}{1 - a} \cdot \sum\limits_{n = 1}^{\infty} 
\left( 
\begin{array}{c}
n + k - 3 \\
n - 1 \\
\end{array} 
\right)
\frac{2}{{(1 + a)}^{n}} + \sum\limits_{n = 1}^{\infty}
\left( 
\begin{array}{c}
2n + k - 2 \\
n - 1 \\
\end{array} 
\right)
\cdot \left( \frac{a}{(1 + a)^{2}} \right)^{n} 
$$
$$
< \frac{k - 1}{1 - a}\cdot 
\sum\limits_{n = 1}^{\infty} {(n + k - 3)}^{k - 2} 
\frac{2}{{(1 + a)}^{n}} + 
 \sum\limits_{n = 1}^{\infty}
\left( 
\begin{array}{c}
2n + k - 2 \\
n - 1 \\
\end{array} 
\right)
\cdot \left( \frac{a}{(1 + a)^{2}} \right)^{n}.
$$
Note that both power series are convergent. 
The series 
$$\sum\limits_{n = 1}^{\infty} {(n + k - 3)}^{(k - 2)} 
\frac{1}{{(1 + a)}^{n}} < \infty$$
because his radius of convergence is $1$ and $\frac{1}{a + 1} < 1$. 
Similarly, the series 
$$\sum\limits_{n = 1}^{\infty}
\left( 
\begin{array}{c}
2n + k - 2 \\
n - 1 \\
\end{array} 
\right)
\cdot 
 \left(\frac{a}{(1 + a)^{2}} \right)^{n} < \infty 
 $$
because his radius of convergence is $\frac{1}{4}$ and $\frac{a}{(1 + a)^{2}} < \frac{1}{4}$ for  $0 \leq a < 1$.

\end{proof}

\ignore{
We can also find a formula for the value of the measure 
$\wh{\ol \nu}_a$ on every cylinder set $[\ol e]$. Let $r(\ol e) = j 
\in V_n$. Then using (\ref{eq: value of ext neas on e}) we have
$$
\ba
\widehat{\overline{\nu}}_{a}([\overline{e}]) = &
\lim\limits_{m \to \infty}{\sum\limits_{i = k}^{n + m + k - 1} {g'}_{ij}^{(n,m)}} \cdot 
p_{i}^{(n + m,\infty,k)} = \\
= &
\lim\limits_{m \to \infty}
\left\lbrack 
\frac{1}{(1 - a)  (1 + a)^{(n + m)}} 
\sum\limits_{i = 0}^{n + m - 1} S_{i + k - j + 1}^{(m - 1)} \cdot 
\left( \frac{a}{1 + a} \right)^{i} 
\left(1 - a^{n + m - i + 1}\right) 
\right\rbrack = \\
= &
\frac{1}{(1 - a)  (1 + a)^{n}} \cdot 
\left\lbrack
\lim\limits_{m \to \infty} 
\left\lbrack
\frac{1}{(1+a)^m}
\sum_{i = 0}^{n + m - 1} S_{i + k - j + 1}^{(m - 1)} \cdot 
\left( \frac{a}{1 + a} \right)^{i} 
\right\rbrack
\right.
+ \\
 &\qquad  - 
\left.
\lim\limits_{m \to \infty}
\left\lbrack 
\frac{a^{ n+ m + 1}}{(1 + a)^{m}} \cdot 
\sum\limits_{i = 0}^{n + m - 1} S_{i + k - j + 1}^{(m - 1)} \cdot \left( \frac{1}{1 + a} \right)^{i}
\right\rbrack \right\rbrack.
\ea
$$
\end{remark}
}

\textit{\textbf{Question.}} The generalized Bratteli diagram 
$B_\infty$ supports two collections of tail invariant measures:
$\{\mu_a \ |\ 0 < a < \infty\}$ and $\{\wh \nu_a \ |\ 0 < a < 1\}$.
How are the measures from the two collections related?

\begin{remark}
We note that the extended measures $\widehat{\nu}_{a}$ are pairwise singular. Indeed, we know that the measures  $\nu_{a}$  are pairwise singular, hence they are supported by pairwise disjoint sets 
$Y_{a}$. The set $\widehat{Y}_{a}$ consists of all infinite paths 
tail equivalent to a path from $Y_a$. Then the sets $\widehat{Y}_{a}$ are pairwise disjoint and this means that $\widehat{\overline{\nu}}_{a}$ are pairwise singular.                  \end{remark}

%%%%%%%%%%%%%%%%%%%%%%%%%%%%%%%%%%%%%%%%%%%%%%%%%%%%%%%%%%%%%%%%%%

\section{Generalized Bratteli-Vershik systems}\label{sect VershikMapPascal}

In this section, we discuss orders and the corresponding Vershik maps for generalized Bratteli diagrams. In particular, we consider the cases of $\mathbb{Z}$-infinite and $\mathbb{N}$-infinite Pascal-Bratteli diagrams (see Section~\ref{sect Bratteli-Pascal}) and
the diagram $B_\infty$ studied in Section \ref{ssect measures on B_infty}.

\subsection{Vershik maps on infinite Pascal-Bratteli diagrams}

In this subsection, we consider infinite Pascal-Bratteli-Vershik dynamical systems. First, we define a so-called natural order on $\mathbb{Z}$-infinite and $\mathbb{N}$-infinite Pascal-Bratteli diagrams. 

First, let $B$ be a $\mathbb{Z}$-infinite Pascal-Bratteli diagram.  
Let $\overline{t} = (\ldots t_{-1}, t_{0}, t_{1}, \ldots)$ be a vertex in $V_{n}$ for $n \geq 2$. 
Recall that we denote $I_{\ov t} = \{i_1, \ldots, i_l\}$, where $i_1 < i_2 < \ldots < i_l$, the set of indexes such that $\ov t$ has nonzero entries exactly at positions $\{i_j\}_{j = 1}^l$. Then we have $\ov t = \sum_{j = 1}^l t_{i_j} \ov e^{(i_j)} $ and $\sum_{j = 1}^{l}t_{i_{j}} = n$. Then $s(r^{-1}(\ov t)) = \{\ov s \in V_{n-1}: \ov s  = \ov s (\ov t, i)= \ov t - \ov e^{(i)},\; i \in I_{\ov t}\}$ and the set $r^{-1}(\ov t)$ is in a one-to-one correspondence with the set of vertices $s(r^{-1}(\ov t))$. Define an order on $r^{-1}( \overline{t})$ as follows: for any two edges $e,f \in r^{-1}( \overline{t})$ with $s(e) = \ov s (\ov t, i)$ and $s(f) = \ov s (\ov t, j)$, where $i,j \in I_{\ov t}$, we have $e < f$ if $i < j$. The order on $\mathbb{Z}$-IPB diagram defined by this rule is called the
\textit{natural order}. In the same manner, we define the natural order on $\mathbb{N}$-IPB diagram.

It is convenient to present any
$\overline{t} \in  V_{n}$ as a pair 
$( I_{\overline{t}}, c(\overline{t}) )$, 
where $c(\overline{t}) =  \left\langle  t_{i_{1}}, \ldots, t_{i_{l}} \right\rangle$.
Let $x = ( {\overline{t}}^{(1)}, {\overline{t}}^{(2)}, \ldots )$ 
be a path of $X_{B}$, where 
${\overline{t}}^{(1)}, {\overline{t}}^{(2)}, \ldots$ are the 
vertices,
${\overline{t}}^{(n)} \in V_{n}$, $n = 1, 2, \ldots$
Then 
$I_{\overline{t}^{(1)}} \subset 
I_{\overline{t}^{(2)}} \subset \ldots$
and
$c({\overline{t}}^{(n + 1)})$ is obtained from
$c({\overline{t}}^{(n)})$ by adding the number "1" 
either to a component of $c({\overline{t}}^{(n)})$
(then $I_{\overline{t}^{(n)}} = I_{\overline{t}^{(n+1)}}$) 
or to an additional position.

Now we are able to determine the sets $X_{\max}$ and $X_{\min}$
for both versions of IPB diagrams. 
The sets $X_{\max}$ and $X_{\min}$ are described as follows.

\begin{remark}
(1) The set $X_{\max}$ consists of all infinite paths $x = ({\overline{t}}^{(1)}, {\overline{t}}^{(2)}, \ldots)$ such that for every $n = 1,2,\ldots$ either $c( {\overline{t}}^{(n + 1)}) = \left\langle t_{i_{1}}, \ldots, t_{i_{l}} + 1 \right\rangle$ or $c({\overline{t}}^{(n + 1)}) = \left\langle t_{i_{1}}, \ldots, t_{i_{l}}, 1 \right\rangle$, where the number "1"  is at position $i_{l + 1} > i_{l}$ and $c({\overline{t}}^{(n)}) = \left\langle t_{i_{1}}, \ldots, t_{i_{l}} \right\rangle$.

(2) Similarly, the set $X_{\min}$ consists of all infinite paths $x = ({\overline{t}}^{(1)}, {\overline{t}}^{(2)}, \ldots)$ such that for every $n = 1,2,\ldots$ either $c( {\overline{t}}^{(n + 1)}) = \left\langle t_{i_{1}}+1, \ldots, t_{i_{l}} \right\rangle$ or $c({\overline{t}}^{(n + 1)}) = \left\langle 1,t_{i_{1}}, \ldots, t_{i_{l}} \right\rangle$, where the number "1"  is at position $i_{0} < i_{1}$ and $c({\overline{t}}^{(n)}) = \left\langle t_{i_{1}}, \ldots, t_{i_{l}} \right\rangle$.
\end{remark}

Moreover, we can represent the set ${X}_{\max}$ as a disjoint union 
of two subsets $X^{u}_{\max}$ and $X^{c}_{\max}$, where
\begin{enumerate}
    \item $X^{u}_{\max} = \{x = ({\overline{t}}^{(1)}, {\overline{t}}^{(2)}, \ldots ) \in  X_{\max}$  such that  for infinitely many  $n$, if 
$c( {\overline{t}}^{(n)}) = \left\langle t_{i_{1}}, \ldots, t_{i_{l}} \right\rangle$  then $c({\overline{t}}^{(n + 1)}) = \left\langle t_{i_{1}}, \ldots, t_{i_{l}}, 1 \right\rangle \}$
    \item $X^{c}_{\max} = 
\{ x = ({\overline{t}}^{(1)}, {\overline{t}}^{(2)}, \ldots ) \in  X_{\max}$ for which there exists $n$ with $c\left( {\overline{t}}^{(n)} \right) = \left\langle t_{i_{1}}, \ldots, t_{i_{l}} \right\rangle$ such that for $m = 1, 2, \ldots$
we have $c\left( {\overline{t}}^{(n + m)} \right) = \left\langle t_{i_{1}}, \ldots, t_{i_{l}} + m \right\rangle 
\}$.
\end{enumerate}

In other words, the set $X_{\max}^c$ consists of all infinite maximal paths such that, starting from some level, we obtain the next vertex of the path by adding ``1'' at the same coordinate $i_l$. The set $X_{\max}^u$ is the complement of the set $X_{\max}^c$.
In a similar way, we divide $X_{\min}$ into the union of the sets $X^{u}_{\min}$ and $X^{c}_{\min}$.

\begin{lemma}
    For both versions of infinite Pascal-Bratteli diagrams, the set ${X}^{u}_{\max}$ is uncountable and the sets ${X}^{c}_{\max}$ and ${X}^{c}_{\min}$ are infinite countable. For $\mathbb{Z}$-infinite Pascal-Bratteli diagram the set ${X}^{u}_{\min}$ is uncountable, and for $\mathbb{N}$-infinite Pascal-Bratteli diagram the set ${X}^{u}_{\min}$ is empty.
\end{lemma}

\begin{proof}
 First we prove that the set ${X}^{u}_{\max}$ is uncountable for both versions of IPB. Indeed, for an infinite path $x = (\ol t^{(n)}) \in X_{\max}$, we can obtain the next vertex $\ov t^{(n+1)}$ from the vertex $\ov t^{(n)}$ by adding ``1'' either to the last non-zero coordinate $i_l$ of the vector $\ov t^{(n)}$ or any of the zero coordinates after the coordinate $i_l$. Recall that each path in $X^{u}_{\max}$ is obtained by adding ``1'' to the zero coordinates infinitely many times. Thus, each path
$x \in X^{u}_{\max}$ is completely determined by two
infinite sequences 
$\overline{i} = \{ i_{1} <  i_{2} < i_{3} < \ldots \}$ 
and 
$c(x) = \{ t_{i_{1}}, t_{i_{2}}, t_{i_{3}}, \ldots \}$, 
where
$i_{1}, i_{2}, i_{3}, \ldots$ are positions in $\mathbb{Z}$ (or in $\mathbb{N}$), such that all but finitely many vertices of $x$ have non-zero coordinates at those positions, and each natural number $t_{i_{l}}$ is the maximal possible value of the non-zero coordinate $i_l$ for a vertex of $x$. 
To find a vertex ${\overline{t}}^{(n)}$ of the path $x$, 
we choose a number $l \geq 1$ such that
$t_{i_{1}} + \ldots + t_{i_{l-1}} < n \leq
t_{i_{1}} + \ldots + t_{i_{l}}$ 
and then set
$I_{{\overline{t}}^{(n)}} = \{ i_{1}, \ldots, i_{l} \}$ and 
$c \left( {\overline{t}}^{(n)} \right) = 
\left\langle t_{i_{1}}, \ldots, t_{i_{l-1}}, n - ( t_{i_{1}} + \ldots + t_{i_{l-1}}) \right\rangle$. Therefore, the set ${X}^{u}_{\max}$ is uncountable. Similarly, for $\mathbb{Z}$-infinite Pascal-Bratteli diagram, the set $X_{\min}^u$ is uncountable. It is easy to see that for $\mathbb{N}$-infinite Pascal-Bratteli diagram, the set $X_{\min}^u$ is empty.

Now we prove that the set $X^{c}_{\max}$ is countable. Each $x \in X^{c}_{\max}$ is completely determined 
by two finite sequences $\overline{i} = \{ i_{1} < \ldots < i_{s} \}$ and $c(x) = \{ t_{i_{1}}, \ldots, t_{i_{(s - 1)}}, t_{i_{s}} = \infty \}$, where $i_{1}, \ldots, i_{s}$ are positions in $\mathbb{Z}$ (or in $\mathbb{N}$) such that all but finitely many vertices of $x$ have non-zero coordinates at those positions, each number $t_{i_{l}}$, for $1 \leq l \leq s - 1$ is a natural number which is the maximal possible value of the non-zero coordinate $i_l$ for a vertex of $x$, and $t_{i_{s}} = \infty$. To find a vertex ${\overline{t}}^{(n)}$ of the path $x$, 
we repeat the procedure above whenever $n \leq t_{i_{1}} + \ldots + t_{i_{s - 1}}$. For $n > t_{i_{1}} +\ldots + t_{i_{s - 1}}$ we have $I_{{\overline{t}}^{(n)}} = \{ i_{1}, \ldots, i_{s} \}$ and $c({\overline{t}}^{(n)}) =  \left\langle t_{i_{1}}, \ldots, t_{i_{s - 1}}, n - (t_{i_{1}} + \ldots + t_{i_{s - 1}}) \right\rangle$. Thus, the set $X^{c}_{\max}$ is countable.
Similarly, the set $X^{c}_{\min}$ is countable. 

\end{proof}

\begin{remark}\label{star}

(1) For both IPB, the sets $X^{c}_{\max}$ and $X^{c}_{\min}$ contain a special subset $X^{s}$ consisting of those paths $x = x(i)$ such that $\overline{i}= \{ i \}$ and $c(x) =\{ \infty \}$ for $i \in  \mathbb{Z}$ or $i\in  \mathbb{N}$. Then for $x = (\ol t^{(n)})_{n \in \mathbb{N}} \in X^s$, we have $I_{{\overline{t}}^{(n)}} = \{ i \}$, and $c({\overline{t}}^{(n)}) = \left\{ n \right\}$. The paths of $X^{s}$ are maximal and minimal simultaneously. Moreover, $X^{s} = X_{\max} \cap X_{\min}$. 

(2) For $\mathbb{N}$-IPB diagram, we have $X^{u}_{\min} =  \emptyset$ which implies $X_{\min} = X^{c}_{\min}$.

(3)  For any ${\overline{t}}^{(n)} \in V_n$, it is easy to find a unique finite minimal path and a unique finite maximal path that join ${\overline{t}}^{(n)}$ with $V_1$. Namely,
let ${\overline{t}}^{(n)}$ be defined by a pair $\left( {I}_{\overline{t}^{(n)}},  c({\overline{t}}^{(n)}) \right)$, where
$$I_{{\overline{t}}^{(n)}} = \left\{ i_{1} < i_{2} < \ldots < i_{l} \right\} \mbox{ and }
c( {\overline{t}}^{(n)} ) = \left\langle  t_{i_{1}}, \ldots, t_{i_{l}} \right\rangle.$$

Let
$x_{\min} = ( {\overline{t}}^{(1)}, \ldots, {\overline{t}}^{(n)} )$ 
and 
$x_{\max} = ( {\overline{s}}^{(1)}, \ldots, {\overline{s}}^{(n)} )$, where  ${\overline{s}}^{(n)} = {\overline{t}}^{(n)}$,
be the minimal and the maximal paths between $V_{1}$ and ${\overline{t}}^{(n)}$. 
Then
$$
c({\overline{t}}^{(n - 1)}) =  \left\langle t_{i_{1}} - 1, \ldots, t_{i_{l}} \right\rangle
$$
and
$$
I_{{\overline{t}}^{(n-1)}} = \begin{cases}
    \left\{ i_{1} < \ldots < i_{l} \right\} & \text{if } \left( t_{i_{1}} - 1 \right) > 0\\
    \left\{ i_{2} < \ldots < i_{l} \right\} & \text{if } \left( t_{i_{1}} - 1 \right) = 0.\\
\end{cases}
$$

To determine the next vertices
${\overline{t}}^{(n - 2)}, \ldots, {\overline{t}}^{(1)}$ 
one must subtract in turn the number "1" from the first component of the
vectors ${\overline{t}}^{(n - 1)}, \ldots, {\overline{t}}^{(2)}$
and then determine the sets 
$I_{{\overline{t}}^{(n - 2)}}, \ldots, I_{{\overline{t}}^{(1)}}$ 
as above.
In the end, we will have
$I_{{\overline{t}}^{(1)}} = \{ i_{l} \}$,
$c( {\overline{t}}^{(1)} ) = \left\langle 1 \right\rangle$.

Similarly, we determine $x_{\max}$. 
We have
$c ({\overline{s}}^{(n - 1)} ) = 
\left\langle t_{i_{1}}, \ldots, t_{i_{l}} - 1  \right\rangle$
and
$$
I_{{\overline{s}}^{(n-1)}} = \begin{cases}
    \left\{ i_{1} < \ldots < i_{l} \right\} & \text{if } \left( t_{i_{l}} - 1 \right) > 0\\
    \left\{ i_{1} < \ldots < i_{l-1} \right\} & \text{if } \left( t_{i_{l}} - 1 \right) = 0.\\
\end{cases}
$$

To determine the next vertices
${\overline{s}}^{(n - 2)}, \ldots, {\overline{s}}^{(1)}$ 
one must subtract in turn the number ``1'' from the last component of the vectors
${\overline{s}}^{(n - 1)}, \ldots, {\overline{s}}^{(2)}$ 
and then determine the sets 
$I_{{\overline{s}}^{(n - 2)}}, \ldots, I_{{\overline{s}}^{(1)}}$ 
as above. We will have
$I_{{\overline{s}}^{(1)}} = \{ i_{1}\}$ and
$c( {\overline{t}}^{(1)} )= \left\langle 1 \right\rangle.$

\end{remark}

In Theorem \ref{A}, we describe the sets $Succ(x)$ and $Pred(y)$ for $\mathbb{Z}$-IPB and $\mathbb{N}$-IPB diagrams (see Subsection~\ref{ssect Vmap} for definitions).

\begin{theorem}\label{A}
Let $B$ be a $\mathbb{Z}$- or $\mathbb{N}$-infinite Pascal-Bratteli diagram and the sets $X^{c}_{\min}$, $X^{u}_{\min}$, $X^{c}_{\max}$, $X^{u}_{\max}$ be as above. Then

(i) if $x \in X^{u}_{\max}$ then $Succ(x) = \emptyset$;

(ii) if $x = ({\overline{t}}^{(1)}, {\overline{t}}^{(2)}, \ldots ) \in X^{c}_{\max}$ 
such that
$$
I_{{\overline{t}}^{(n)}} = \{ i_{1} < \ldots < i_{l} \} \mbox{ and }
c( {\overline{t}}^{(n)} ) = \{ t_{i_{1}}, \ldots, t_{i_{(l - 1)}}, t_{i_{l}} = \infty \}$$ 
then
$Succ(x) = \{ x(i_{l}) \} \subset X^s$ (see Part (1) of Remark \ref{star});

(iii) if $y \in X^{u}_{\min}$ then $Pred(y) = \emptyset$.

(iv) if $y =( {\overline{s}}^{(1)}, {\overline{s}}^{(2)}, \ldots ) \in X^{c}_{\min}$ 
such that 
$$
I_{{\overline{s}}^{(n)}} = \{ i_{1} > \ldots > i_{l} \} \mbox{ and } c( {\overline{s}}^{(n)} ) = \{ t_{i_{1}}, \ldots, t_{i_{(l - 1)}}, t_{i_{l}} = \infty \}
$$ 
then
$Pred(y) = \{ x(i_{l}) \} \subset X^{s}$.
\end{theorem}

\begin{proof}
(i) Let $x = ({\overline{t}}^{(1)}, {\overline{t}}^{(2)}, \ldots ) \in X^{u}_{\max}$ with $I_{{\overline{t}}^{(n)}} = \left\{ i_{1} < \ldots < i_{l} \right\}$ and $c\left( {\overline{t}}^{(n)} \right) = \left\langle t_{i_{1}}, \ldots, t_{i_{l}} \right\rangle$, 
where
$\left\{ i_{1} < \ldots < i_{l} \right\}$ 
and
$\left\langle t_{i_{1}}, \ldots, t_{i_{l}} \right\rangle$ 
depend on $n$ and $i_{l} \to \infty$ as $n \to \infty$. 
Assume that $y = ( {\overline{s}}^{(1)}, {\overline{s}}^{(2)}, \ldots ) \in Succ(x)$,
$y \in  X_{\min}$. Then
for infinitely many $n$, there are $\overline{z} \in V_{n + 1}$ and edges
$e_{\overline{z}{\overline{t}}^{(n)}}^{(n)}, e_{\overline{z}{\overline{s}}^{(n)}}^{(n)} 
\in {r}^{-1}(\overline{z})$ with $s(e_{\overline{z}{\overline{t}}^{(n)}}^{(n)}) = \overline{t}^{(n)}$ and $s(e_{\overline{z}{\overline{s}}^{(n)}}^{(n)}) = \overline{s}^{(n)}$
such that $e_{\overline{z}{\overline{s}}^{(n)}}^{(n)}$ is the successor of
$e_{\overline{z}{\overline{t}}^{(n)}}^{(n)}$. 
Then 
$$
\overline{z} = {\overline{t}}^{(n)} + \ov{e}^{(i)} = {\overline{s}}^{(n)} + \ov{e}^{(j)}
$$ 
for some $i,j \in \mathbb{Z}$
(or $i, j \in \mathbb{N}$).

Note that $i < i_{l}$ because otherwise
$e_{\overline{z}{\overline{t}}^{(n)}}^{(n)}$ would be the maximal edge of
${r}^{( -1)}(\overline{z})$ and would have no successor.
Thus, $i_{k - 1} \leq i < i_{k}$ for some $k \in \{1, \ldots, l\}$. 
If $i = i_{k - 1}$, since $e_{\overline{z}{\overline{s}}^{(n)}}^{(n)}$ is the successor of
$e_{\overline{z}{\overline{t}}^{(n)}}^{(n)}$, we have
$$
I_{{\overline{s}}^{(n)}} = \left\{ i_{1} < \ldots <  i_{l} \right\} \mbox{ and }
c\left( {\overline{s}}^{(n)} \right) = \left\langle t_{i_{1}}, \ldots, t_{i_{k - 1}} + 1,  t_{i_{k}} - 1, \ldots, t_{i_l} \right\rangle
$$
If $i > i_{k - 1}$, we have
$$
I_{{\overline{s}}^{(n)}} = \left\{ i_{1} < \ldots < i_{k - 1} < i <  \ldots < i_{l} \right\} \mbox{ and }
c\left( {\overline{s}}^{(n)} \right) = \left\langle t_{i_{1}}, \ldots, t_{i_{k - 1}}, 1, t_{i_{k}} - 1, \ldots, t_{i_l} \right\rangle.
$$
Let $y^{(n)}_{\min}$ be the minimal path between
${\overline{s}}^{(n)}$ and $V_{1}$. According to the Part (3) of Remark \ref{star}, the path
$y^{(n)}_{\min}$ starts from the position $i_{l}$ or $i \geq i_{l - 1}$ of $V_{1}$.
Since ${y^{(n)}}_{\min}$ is a part of $y$, the minimal path $y$ starts from the same position of $V_{1}$. Thus, there is no a such $y$ because both $i_{l} \to \infty$ and $i_{l - 1} \to \infty$ as $n \to \infty$.

(ii) Let $x = ({\overline{t}}^{(1)}, {\overline{t}}^{(2)}, \ldots ) \in X_{\max}^{c}$ 
and let $N$ be such that
$c\left( {\overline{t}}^{(N)} \right) = \left\langle t_{i_{1}}, \ldots, t_{i_{l}} \right\rangle$
and for $m = 1, 2, \ldots$ we have
$c\left( {\overline{t}}^{(N + m)} \right) = \left\langle t_{i_{1}}, \ldots, t_{i_{l}} + m \right\rangle$.
Assume that $y = ({\overline{s}}^{1}, {\overline{s}}^{2}, \ldots ) \in Succ(x)$, 
$y \in X_{\min}$. Choose infinitely many $n > N$, such that there are $\overline{z} \in V_{n + 1}$ and $e_{\overline{z}{\overline{t}}^{(n)}}^{(n)}, e_{\overline{z}{\overline{s}}^{(n)}}^{(n)} \in 
{r}^{-1}(\overline{z})$
such that $e_{\overline{z}{\overline{s}}^{(n)}}^{(n)}$ 
is the successor of $e_{\overline{z}{\overline{t}}^{(n)}}^{(n)}$. 
Then $\overline{z} = {\overline{t}}^{(n)} + \ov e^{(i)}$ for some $i \in  \mathbb{Z}$ (or $i \in  \mathbb{N}$). 
Let us observe that $i < i_{l}$ 
and by the same arguments as before we get:
if $i = i_{k-1}$,
$$
I_{{\overline{s}}^{(n)}} = \left\{ i_{1} <  \ldots < i_{l} \right\} \mbox{ and }
c\left( {\overline{s}}^{(n)} \right) = \left\langle t_{i_{1}}, \ldots, t_{i_{k - 1}} + 1,  t_{i_{k}} - 1,\ldots, t_{i_{l}} + (n - N) \right\rangle.
$$
If $i_{k - 1} < i < i_k$ then
$$
I_{{\overline{s}}^{(n)}} = \left\{ i_{1} < \ldots < i_{k - 1} < i < \ldots < i_{l} \right\} \mbox{ and } c\left( {\overline{s}}^{(n)} \right) = \left\langle t_{i_{1}},\ldots t_{i_{k - 1}}, 1, t_{i_{k}} - 1, \ldots, t_{i_{l}} + (n - N) \right\rangle.
$$
Form the minimal path $y^{(n)}_{\min}$ between
${\overline{s}}^{(n)}$ and $V_{1}$ according to the procedure from Part (3) of Remark \ref{star}. Then the vertices of this path between the levels $1$ and $n- (t_{i_{1}} + \ldots + t_{i_{l}})$, coincide with
the vertices between the same levels of the minimal path $x(i_{l}) \in  X^{s}$. Since $(n-(t_{i_{1}} + \ldots + t_{i_{l}})) \to \infty$ as $n \rightarrow \infty$, we have $Succ(x) = \{ x(i_{l})\}$. 

(iii) and (iv) In a similar way, we determine the sets $Pred(y)$ for $y \in X_{\min}$. 
\end{proof}

Theorem~\ref{Theorem B)} concerns Vershik maps on $\mathbb{Z}$- and $\mathbb{N}$-infinite Pascal Bratteli diagrams and properties of the sets $X_{\min}$, $X_{\max}$.

\begin{theorem}\label{Theorem B)}

(1) The Vershik map $\varphi_{B} \colon X \setminus X_{\max} \rightarrow X \setminus X_{\min}$
can be extended to a continuous surjection
$\varphi \colon \lbrack ( X \setminus X_{\max} ) \cup X^{c}_{\max} \rbrack 
\rightarrow  \lbrack (X \setminus X_{\min}\ ) \cup X^{s} \rbrack$ 
such that $\varphi = \varphi_{B}$ on  $X \setminus X_{\max}$ and 
$\varphi(x) = x(i_{l})$, where $x$ and $x(i_{l})$ are described in the proof of Theorem (\ref{A}).
Similarly, the inverse map 
$\varphi_{B}^{-1} \colon  X \setminus X_{\min} \rightarrow X \setminus X_{\max}$ 
can be extended to a surjection
$\psi \colon \lbrack (X \setminus X_{\min}) \cup X^{c}_{\min} \rbrack
\rightarrow \lbrack (X \setminus X_{\max}) \cup X^{s} \rbrack$ 
such that
$\psi = \varphi_{B}^{-1}$ on $X \setminus X_{\min}$ and $\psi(y) = x{(i_{l})}$, where $y$ and $x(i_{l})$ are as above.

(2) There exists a continuous one-to-one map
$f \colon X_{\max} \rightarrow X_{\min}$ for $\mathbb{Z}$-IPB diagram and a continuous surjection $f \colon X_{\max} \rightarrow X_{\min}$ for $\mathbb{N}$-IPB diagram.

(3)
The sets $X_{\max}$ and $X_{\min}$ are nowhere dense in $X_{B}$ in both cases of $\mathbb{Z}$-IPB and $\mathbb{N}$-IPB.

(4) Let $\overline{d}= \langle d_{1}, d_{2}, d_{3}, \ldots \rangle$ be a probability vector and
$\mu_{d\overline{}}$ be an invariant measure on $\mathbb{Z}$-IPB or $\mathbb{N}$-IPB described in Section~\ref{sect Bratteli-Pascal}. Then $\mu_{\overline{d}}(X_{\max}) = \mu_{\overline{d}}(X_{\min}) = 0$ whenever $d_{i} < 1 $ for each $i = 1, 2, \ldots$.
If $\overline{d}$ is a such vector that $d_{j} = 1$ and
$d_{i} = 0$ for $i \neq j$, $j \in \mathbb{Z}$ or $j \in \mathbb{N}$ then $\mu_{\overline{d}}$ is a $\delta$-measure concentrated on the path ${x\overline{}}{(j)} = ( {\overline{s}}^{(1,j)}, {\overline{s}}^{(2,j)}, \ldots)$,
where $s_{i}^{(n,j)}= n$ for $i = j$ 
and $s_{i}^{(n,j)} = 0$ for $i \neq j$, $i, n= 1, 2, \ldots$. 
Moreover,
${x\overline{}}{(j)}$ is both a maximal and a minimal path.

\end{theorem}

\begin{proof}

Part (1) is a consequence of Theorem \ref{A}.

(2) First consider a $\mathbb{Z}$-IPB diagram. 
Let $x = ({\overline{t}}^{(1)}, {\overline{t}}^{(2)}, \ldots ) \in  X_{\max}$ with $I_{{\overline{t}}^{(n)}} = \{ i_{1} < \ldots < i_{l} \}$ 
and
$c( {\overline{t}}^{(n)} ) = \left\langle t_{i_{1}}, \ldots, t_{i_{l}} \right\rangle$, where $l$ depends on $n$.
Define a path 
$y = ({\overline{s}}^{(1)}, {\overline{}\overline{s}}^{(2)}, \ldots )$, 
such that
$c ( {\overline{s}}^{(n)} ) = \left\langle t_{i_{1}}, \ldots, t_{i_{l}} \right\rangle$
and $I_{{\overline{s}}^{(n)}} = \{ i^{'}_{1} > i^{'}_{2} > \ldots > i^{'}_{l} \}$, 
where 
$$
i^{'}_{1} =  i_{1}, \;
i^{'}_{2} = 2 i_{1} - i_{2}, \ldots,
i^{'}_{l} = 2 i_{1} - i_{l},
$$  
for $n = 1, 2, \ldots$. 
It is obvious that $y \in  X_{\min}$. 
We set $f(x) = y$. 
Then it is not hard to see that 
$f \colon X_{\max} \rightarrow X_{\min}$ is a continuous one-to-one mapping. 
Moreover, $f( X_{\max}^{u} ) = X_{\min}^{u}$, 
$f( X_{\max}^{c} ) =  X_{\min}^{c}$ 
and $f$ is equal to the identity on $X^{s}$.

Similarly we define a continuous surjection $f \colon X_{\max} \rightarrow X_{\min} = X_{\min}^{c}$ 
for $\mathbb{N}$-IPB diagram. 
We define vertices ${\overline{s}}^{(n)}$, $n \geq 1$ in the same manner if $i^{'}_{l} \geq 1$. 
If $i^{'}_{l} < 1 < i^{'}_{l - 1}$ then we determine ${\overline{s}}^{(n)}$ by the sets $I_{{\overline{s}}^{(n)}} = \{ i^{'}_{1} > i^{'}_{2} > \ldots > i^{'}_{l - 1} > 1 \}$ 
and $c ( {\overline{s}}^{(n)} ) = \left\langle t_{i_{1}}, \ldots, t_{i_{l}} \right\rangle$.

(3) It is easy to prove the property (3) if we take into consideration the
structure of the sets $X_{\max}$ and $X_{\min}$ .

(4) Assume that $B = (V, E)$ is a $\mathbb{Z}$-IPB diagram. 
Denote
$$
X_{\max}^{(n)} = \bigcup_{\ov s \in V_n} [E_{\max}(V_1,\overline{s})],
$$
where $[E_{\max}(V_1,\overline{s})]$ is a cylinder set corresponding to the maximal path $E_{\max}(V_1, \ov s)$ of $E(V_1, \overline{s})$ for $\overline{s} \in V_{n}$.
Then $X_{\max}^{(n)}$ is an open set for every $n \in \mathbb{N}$ and 
$$
X_{\max} = \bigcap_{n \in \mathbb{N}} X_{\max}^{(n)}.
$$
Let 
$\overline{d}=  \langle d_{i},\ i\in \mathbb{Z} \rangle$ 
be a probability vector such that 
$d^{*} = \sup_{i \geq 1}d_{i} < 1$. We have
$$
\mu_{\overline{d}}(X_{\max}^{(n)}) =
\mu_{\overline{d}} \left(\bigcup_{\overline{s} \in V_{n}} E_{\max}({\overline{s}}^{(0)},\overline{s}) \right) =
\sum_{\overline{s}\in V_{n}} \prod_{i \in \mathbb{Z}} d_{i}^{s_{i}}.
$$
We also have
$$
\mu_{\overline{d}}\left( X_{\max}^{(n + 1)}\right) =
\sum_{\overline{s}\in V_{n + 1}} \prod_{i \in \mathbb{Z}} d_{i}^{s_{i}}
\leq d^{\ast} \cdot
\sum_{\overline{s} \in V_{n}} \prod_{i \in \mathbb{Z}} d_{i}^{s_{i}}
= d^{\ast} \cdot \mu_{\overline{d}}\left(X_{\max}^{(n)}\right).
$$
Hence 
$$
\mu_{\overline{d}}(X_{\max}) =
\mu_{\overline{d}} \left( \bigcap_{n = 1}^{\infty}X_{\max}^{(n)}\right) = 0.
$$
Similarly,
$\mu_{\overline{d}} (X_{\min}) = 0$.
The same properties are true for  $\mathbb{N}$-IPB diagram and can be proved similarly. The proof of the last part of (4) is obvious.

\end{proof}

\begin{remark}\label{2stars} 
The map $g \colon X_{B} \rightarrow X_{B}$ 
such that $g = \varphi_{B}$ on $X \setminus X_{\max}$ and 
$g = f$ on $X_{\max}$ is not continuous on the entire space $X_{B}$. 
It is an one-to-one Borel mapping if $({X}_{B}, E)$ is a $\mathbb{Z}$-IPB diagram and it
is a Borel surjection if $({X}_{B}, E)$ is a $\mathbb{N}$-IPB diagram.
\end{remark}

\subsection{The Vershik map on the generalized Bratteli diagram 
\texorpdfstring{$B_\infty$}{B_\infty}}

In this subsection, we consider Vershik maps defined on the path 
spaces of the generalized Bratteli diagram $B_\infty$ and its 
subdiagram $\ol B(W, k)$ where $k \geq 1$ is fixed and 
$W = (W_n), W_n = \{k, \ldots, k+n-1\}$. 
We denote by $X_{B_\infty}$ and 
$X_k := X_{\ol B(W, k)}$ the corresponding path spaces. 
For both diagrams, $B_{\infty}$ and $\overline{B}(W,k)$, the set of 
all orders is uncountable and we do not consider all possible
orders.  We present here a few examples of orders such 
that the corresponding Vershik maps have different properties.

Let $\omega$ be the left-to-right order for every vertex $v \in V_n$,
$n > 1$. For the standard Bratteli diagram $\ol B(W, k)$, we can 
easily find the sets $X_{\max}(\om)$ and $X_{min(\om)}$ of infinite 
maximal and minimal paths: $X_{min(\om)}$ is a singleton containing
of the vertical path $X_{\min}$ through vertices $k$ for every level,
$X_{\max}(\om) $ is a countable set $\{y_1, y_2, \ldots, y_\infty\}$ where 
$y_i$ is the path going through the vertices $k, k+1, \ldots k+i, 
k+i, \ldots$ for every $i \geq 1$ and $y_\infty$ is the rightmost
infinite path. In other words, $y_i$ is a slanting path 
connecting $k$ and $k+i$ and then going vertically through the vertices $k+i$.

Let $C$ be the countable set of eventually vertical paths in 
$X_k$. The set $C$ contains $X_{\max}\setminus  \{y_\infty\} $.

\begin{lemma}
(1) The Vershik map $\varphi_B : X_k \setminus X_{\max}(\om) \to
X_k \setminus X_{\min}(\om) $ admits a continuous surjective 
extension $\varphi$ by setting 
$\varphi (y_i) = X_{\min}$ for all $i = 1, 2, \ldots, \infty$.

(2) Let $\mathcal R$ denote the tail equivalence relation. Then
$C = \bigcup_{i \geq 1} \mathcal R(y_i)$ where $y_i \in 
X_{\max}(\om)$.
\end{lemma}

The \textit{proof} is obvious. One only checks that this map 
$\varphi$ is continuous. 
\vskip 0.3cm 

The set $D = X_k \setminus C$ is characterized by the property:
$$
x = (x_n) \in D \ \Longleftrightarrow\ |\{ n \in \N : r(x_n) >
s(x_n)\}| = \infty.
$$

We can show that the tail equivalence relation is minimal up to 
a countable set.

\begin{prop}\label{prop Vm minimal}
Let $x \in D$. Then $\mathcal{R}(x)$ is dense in $X_k$. 
\end{prop}

\begin{proof}
We show that for every fixed $x \in D$ and every cylinder
set $[\ol e] $, one has $\mathcal{R}(x) \cap [\ol e]  \neq \emptyset$.
Let $r(\ol e) = j \in V_n$. Take $m >n$ such that $r(x_m) = l > j $.
By definition of the subdiagram $\ol B(W, k)$, there exists
a finite path connecting the vertices $j$ and $l$. 
This means that that the $\mathcal R$-orbit of $x$ will visit 
$[\ol e]$. We also note that $\mathcal{R}(x_\infty)$ is a dense
orbit. 
\end{proof}

\begin{remark}
    (1) A similar result holds for the diagram $B_\infty$. As in 
the case of $\ol B(W, k)$, we have countably many maximal
infinite paths and the unique minimal path through the leftmost 
vertex. The extension of the Vershik map is continuous and surjective.

(2) Proposition \ref{prop Vm minimal} is also true for the 
diagram $B_\infty$. 
\end{remark}

\begin{lemma}
(1) For the generalized Bratteli diagram $B_{\infty}$ and 
its subdiagram $\overline{B}(W,k)$ there exist orders $\tau$ and 
$\tau(k)$, respectfully,  such that 
the Vershik map is an essentially minimal homeomorphism of the path space with a unique fixed point.

(2) There exists an order $\omega$ of $B_{\infty}$ such that 
both sets
$X_{\max}$ and $X_{\min}$ are countable and $Succ(x) = \emptyset$
for every $x \in X_{\max}$ and $Pred(x) = \emptyset$ for every $x \in X_{\min}$.
\end{lemma}

\begin{proof} (1) 
We define an order $\tau$ on $B_\infty$ by the following rule:
$\tau$ is the left-to-right order for every vertex $v \in V_{2n +1}$
and $\tau$ is the right-to-left order for every vertex 
$v \in V_{2n}$, $n = 0, 1, \ldots$. Similarly, we consider the case
of the subdiagram $\ol B(W, k)$. Then, it is obvious that  
$X_{\max} = X_{\min} = \{ z\}$, where $z$ is the vertical path 
that goes through the first vertex of $B_\infty$ (or the vertex
$k$ for $\ol B(W, k)$.

(2) Now we define an order $\omega$ on the diagram $B_{\infty}$ 
that satisfies statement (2).  
For each $i > 2$, let $\omega_{i}$ be defined as follows:
$$
\omega_i = \{ (i-1) <1 < 2 <\ldots < (i-2) < i\}
$$
and $\omega_2 = \{ 1 < 2\}$ (note that $|r^{-1}(1)| =1$).  
Then, for every level $V_{n}$, $n \geq 1$ and for
every vertex $i$, the cycle $\omega_i$ defines a stationary 
linear order on 
$r^{-1}(i)$. In such a way the diagram $B_\infty$ is supplied with 
the stationary order $\omega$. It follows from this definition that the set 
 $X_{\max}(\omega)$ of all infinite maximal paths consists
of all vertical paths $x_{i}$, that is $x_i$ passes through the 
vertex $i$ for all levels $V_n$, $i\in \N$. The set 
$X_{\min}(\omega)$ is formed by the infinite path $x_1$ and two
sequences of infinite paths, $\{y_{i} = (y_i(n): i \geq 1\}$ and 
$\{z_{m}  = (z_m(n)) : m \geq 2\}$, where $s(y_i(n)) = i+n -1,
r(y_i(n)) = i+n$ and $s(z_m(n)) = r(z_m(n)) =1$ for $1 \leq n < m$
and $s(z_m(n)) = n -m + 1, r(z_m(n)) = n -m +2$ for $n \geq m$.
In other words, $y_i$ is the right-slanting pass that begins 
at $i$ and goes through the vertices $i+1, i+2, ...$, and 
$z_m$ is the path that goes vertically through the vertex 1 
up to the level $V_m$ and then is parallel to $y_i$. It can be 
easily checked that for this order $\omega$ we have  
 $Succ(x) = \emptyset$ and 
$Pred(y) = \emptyset$ for $x \in X_{\max}$ or for 
$y \in X_{\min}$ respectively.   
\end{proof}

\begin{corol}
For the the ordered Bratteli diagram $({B}_{\infty}, \omega)$, 
the Vershik map $\varphi \colon X \setminus X_{\max} \to X \setminus X_{\min}$ cannot be extended 
to a continuous map of the entire space $X_{B_\infty}$. 
However, there exists a one-to-one Borel extension of
$\varphi$ acting on the space $X_{B_\infty}$.  
\end{corol}  

\medskip
\textbf{Acknowledgments.} 
We are very grateful to our colleagues, 
especially, P. Jorgensen, T. Raszeja, and S. Sanadhya for the numerous valuable discussions. S.B. and O.K. are also grateful to the Nicolas Copernicus 
University in Toru\'n for its hospitality and support. S.B. acknowledges the hospitality of AGH University of Krakow during his visit. O.K. is supported by the NCN (National Science Center, Poland) Grant 2019/35/D/ST1/01375 and by the program ``Excellence Initiative - Research University'' for the AGH University of Krakow. 

\bibliographystyle{alpha}
\bibliography{bibliographyPascal}

\end{document}